\documentclass[a4paper,11pt]{article}
\usepackage{amssymb,amsmath,amsthm,amscd}
\usepackage[mathscr]{eucal}
\usepackage{fullpage}
\newtheorem{thm}{Theorem}[section]
\newtheorem{lem}[thm]{Lemma}
\newtheorem{cor}[thm]{Corollary}
\newtheorem{prop}[thm]{Proposition}

\theoremstyle{definition}
\newtheorem{defn}[thm]{Definition}
\theoremstyle{remark}

\newtheorem{exmp}[thm]{Example}

\newcommand{\Eroot}[8]{%
\begin{picture}(56,20)(0,-20)%
\put(4,-20){\hbox to0pt{\hss${#1}$\hss}}%
\put(20,-10){\hbox to0pt{\hss${#2}$\hss}}%
\put(12,-20){\hbox to0pt{\hss${#3}$\hss}}%
\put(20,-20){\hbox to0pt{\hss${#4}$\hss}}%
\put(28,-20){\hbox to0pt{\hss${#5}$\hss}}%
\put(36,-20){\hbox to0pt{\hss${#6}$\hss}}%
\put(44,-20){\hbox to0pt{\hss${#7}$\hss}}%
\put(52,-20){\hbox to0pt{\hss${#8}$\hss}}%
\end{picture}%
}

\title{On finite factors of centralizers of parabolic subgroups in Coxeter groups}
\author{Koji Nuida}
\begin{document}
\maketitle
\begin{abstract}
It has been known that the centralizer $Z_W(W_I)$ of a parabolic subgroup $W_I$ of a Coxeter group $W$ is a split extension of a naturally defined reflection subgroup by a subgroup defined by a $2$-cell complex $\mathcal{Y}$.
In this paper, we study the structure of $Z_W(W_I)$ further and show that, if $I$ has no irreducible components of type $A_n$ with $2 \leq n < \infty$, then every element of finite irreducible components of the inner factor is fixed by a natural action of the fundamental group of $\mathcal{Y}$.
This property has an application to the isomorphism problem in Coxeter groups.
\end{abstract}

\section{Introduction}
\label{sec:intro}

A pair $(W,S)$ of a group $W$ and its (possibly infinite) generating set $S$ is called a \emph{Coxeter system} if $W$ admits the following presentation
\begin{displaymath}
W=\langle S \mid (st)^{m(s,t)}=1 \mbox{ for all } s,t \in S \mbox{ with } m(s,t)<\infty \rangle \enspace,
\end{displaymath}
where $m \colon (s,t) \mapsto m(s,t) \in \{1,2,\dots\} \cup \{\infty\}$ is a symmetric mapping in $s,t \in S$ with the property that we have $m(s,t)=1$ if and only if $s=t$.
A group $W$ is called a \emph{Coxeter group} if $(W,S)$ is a Coxeter system for some $S \subseteq W$.
Since Coxeter systems and some associated objects, such as root systems, appear frequently in various topics of mathematics, algebraic or combinatorial properties of Coxeter systems and those associated objects have been investigated very well, forming a long history and establishing many beautiful theories (see e.g., \cite{Hum} and references therein).
For example, it has been well known that, given an arbitrary Coxeter system $(W,S)$, the mapping $m$ by which the above group presentation defines the same group $W$ is uniquely determined.

In recent decades, not only the properties of a Coxeter group $W$ associated to a specific generating set $S$, but also the group-theoretic properties of an arbitrary Coxeter group $W$ itself have been studied well.
One of the recent main topics in the study of group-theoretic properties of Coxeter groups is the \emph{isomorphism problem}, that is, the problem of determining which of the Coxeter groups are isomorphic to each other as abstract groups.
In other words, the problem is to investigate the possible \lq\lq types'' of generating sets $S$ for a given Coxeter group $W$.
For example, it has been known that for a Coxeter group $W$ in certain classes, the set of reflections $S^W := \{wsw^{-1} \mid w \in W \mbox{ and } s \in S\}$ associated to any possible generating set $S$ of $W$ (as a Coxeter group) is equal to each other and independent of the choice of $S$ (see e.g., \cite{Bah}).
A Coxeter group $W$ having this property is called \emph{reflection independent}.
A simplest nontrivial example of a Coxeter group which is not reflection independent is Weyl group of type $G_2$ (or the finite Coxeter group of type $I_2(6)$) with two simple reflections $s,t$, which admits another generating set $\{s,ststs,(st)^3\}$ of type $A_1 \times A_2$ involving an element $(st)^3$ that is not a reflection with respect to the original generating set.
One of the main branches of the isomorphism problem in Coxeter groups is to determine the possibilities of a group isomorphism between two Coxeter groups which preserves the sets of reflections (with respect to some specified generating sets).
Such an isomorphism is called \emph{reflection-preserving}.

In a recent study by the author of this paper, it is revealed that some properties of the centralizers $Z_W(r)$ of reflections $r$ in a Coxeter group $W$ (with respect to a generating set $S$) can be applied to the study of reflection independent Coxeter groups and reflection-preserving isomorphisms.
An outline of the idea is as follows.
First, by a general result on the structures of the centralizers of parabolic subgroups \cite{Nui11} or the normalizers of parabolic subgroups \cite{Bri-How} in Coxeter groups applied to the case of a single reflection, we have a decomposition $Z_W(r) = \langle r \rangle \times (W^{\perp r} \rtimes Y_r)$, where $W^{\perp r}$ denotes the subgroup generated by all the reflections except $r$ itself that commute with $r$, and $Y_r$ is a subgroup isomorphic to the fundamental group of a certain graph associated to $(W,S)$.
The above-mentioned general results also give a canonical presentation of $W^{\perp r}$ as a Coxeter group.
Then the unique maximal reflection subgroup (i.e., subgroup generated by reflections) of $Z_W(r)$ is $\langle r \rangle \times W^{\perp r}$.
Now suppose that $W^{\perp r}$ has no finite irreducible components.
In this case, the maximal reflection subgroup of $Z_W(r)$ has only one finite irreducible component, that is $\langle r \rangle$.
Now it can be shown that, if the image $f(r)$ of $r$ by a group isomorphism $f$ from $W$ to another Coxeter group $W'$ is not a reflection with respect to a generating set of $W'$, then the finite irreducible components of the unique maximal reflection subgroup of the centralizer of $f(r)$ in $W'$ have more elements than $\langle r \rangle$, which is a contradiction.
Hence, in such a case of $r$, the image of $r$ by any group isomorphism from $W$ to another Coxeter group is always a reflection.
See the author's preprint \cite{Nui_ref} for more detailed arguments.

As we have seen in the previous paragraph, it is worthy to look for a class of Coxeter groups $W$ for which the above subgroup $W^{\perp r}$ of the centralizer $Z_W(r)$ of each reflection $r$ has no finite irreducible components.
The aim of this paper is to establish a tool for finding Coxeter groups having the desired property.
The main theorem (in a special case) of this paper can be stated as follows:
\begin{quote}
\textbf{Main Theorem (in a special case).}
Let $r \in W$ be a reflection, and let $s_{\gamma}$ be a generator of $W^{\perp r}$ (as a Coxeter group) which belongs to a finite irreducible component of $W^{\perp r}$.
Then $s_{\gamma}$ commutes with every element of $Y_r$.
(See the previous paragraph for the notations.)
\end{quote}
By virtue of this result, to show that $W^{\perp r}$ has no finite irreducible components, it suffices to find (by using the general structural results in \cite{Nui11} or \cite{Bri-How}) for each generator $s_{\gamma}$ of $W^{\perp r}$ an element of $Y_r$ that does not commute with $s_{\gamma}$.
A detailed argument along this strategy is given in the preprint \cite{Nui_ref}.

In fact, the main theorem (Theorem \ref{thm:YfixesWperpIfin}) of this paper is not only proven for the above-mentioned case of single reflection $r$, but also generalized to the case of centralizers $Z_W(W_I)$ of parabolic subgroups $W_I$ generated by some subsets $I \subseteq S$, with the property that $I$ has no irreducible components of type $A_n$ with $2 \leq n < \infty$.
(We notice that there exists a counterexample when the assumption on $I$ is removed; see Section \ref{sec:counterexample} for details.)
In the generalized statement, the group $W^{\perp r}$ is replaced naturally with the subgroup of $W$ generated by all the reflections except those in $I$ that commute with every element of $I$, while the group $Y_r$ is replaced with a subgroup of $W$ isomorphic to the fundamental group of a certain $2$-cell complex defined in \cite{Nui11}.
We emphasize that, although the general structures of these subgroups of $Z_W(W_I)$ have been described in \cite{Nui11} (or \cite{Bri-How}), the main theorem of this paper is still far from being trivial; moreover, to the author's best knowledge, no other results on the structures of the centralizers $Z_W(W_I)$ which is in a significantly general form and involves much detailed information than those given in the general structural results \cite{Bri-How,Nui11} have been known in the literature.

The paper is organized as follows.
In Section \ref{sec:Coxetergroups}, we summarize some fundamental properties and definitions for Coxeter groups.
In Section \ref{sec:properties_centralizer}, we summarize some properties of the centralizers of parabolic subgroups relevant to our argument in the following sections, which have been shown in some preceding works (mainly in \cite{Nui11}).
In Section \ref{sec:main_result}, we give the statement of the main theorem of this paper (Theorem \ref{thm:YfixesWperpIfin}), and give a remark on its application to the isomorphism problem in Coxeter groups (also mentioned in a paragraph above).
The proof of the main theorem is divided into two main steps: First, Section \ref{sec:proof_general} presents some auxiliary results which do not require the assumption, put in the main theorem, on the subset $I$ of $S$ that $I$ has no irreducible components of type $A_n$ with $2 \leq n < \infty$.
Then, based on the results in Section \ref{sec:proof_general}, Section \ref{sec:proof_special} deals with the special case as in the main theorem that $I$ has no such irreducible components, and completes the proof of the main theorem.
The proof of the main theorem makes use of the list of positive roots given in Section \ref{sec:Coxetergroups} several times.
Finally, in Section \ref{sec:counterexample}, we describe in detail a counterexample of our main theorem when the assumption that $I$ has no irreducible components of type $A_n$ with $2 \leq n < \infty$ is removed.

\paragraph*{Acknowledgments.}
The author would like to express his deep gratitude to everyone who helped him, especially to Professor Itaru Terada who was the supervisor of the author during the graduate course in which a part of this work was done, and to Professor Kazuhiko Koike, for their invaluable advice and encouragement.
The author would also like to the anonymous referee for the precious comments, especially for suggestion to reduce the size of the counterexample shown in Section \ref{sec:counterexample} which was originally of larger size.
A part of this work was supported by JSPS Research Fellowship (No.~16-10825).

\section{Coxeter groups}
\label{sec:Coxetergroups}

The basics of Coxeter groups summarized here are found in \cite{Hum} unless otherwise noticed.
For some omitted definitions, see also \cite{Hum} or the author's preceding paper \cite{Nui11}.

\subsection{Basic notions}
\label{sec:defofCox}

A pair $(W,S)$ of a group $W$ and its (possibly infinite) generating set $S$ is called a \emph{Coxeter system}, and $W$ is called a \emph{Coxeter group}, if $W$ admits the following presentation
\begin{displaymath}
W=\langle S \mid (st)^{m(s,t)}=1 \mbox{ for all } s,t \in S \mbox{ with } m(s,t)<\infty \rangle \enspace,
\end{displaymath}
where $m \colon (s,t) \mapsto m(s,t) \in \{1,2,\dots\} \cup \{\infty\}$ is a symmetric mapping in $s,t \in S$ with the property that we have $m(s,t)=1$ if and only if $s=t$.
Let $\Gamma$ denote the \emph{Coxeter graph} of $(W,S)$, which is a simple undirected graph with vertex set $S$ in which two vertices $s,t \in S$ are joined by an edge with label $m(s,t)$ if and only if $m(s,t) \geq 3$ (by usual convention, the label is omitted when $m(s,t)=3$; see Figure \ref{fig:finite_irreducible_Coxeter_groups} below for example).
If $\Gamma$ is connected, then $(W,S)$ is called \emph{irreducible}.
Let $\ell$ denote the length function of $(W,S)$.
For $w,u \in W$, we say that $u$ is a \emph{right divisor} of $w$ if $\ell(w) = \ell(wu^{-1}) + \ell(u)$.
For each subset $I \subseteq S$, the subgroup $W_I := \langle I \rangle$ of $W$ generated by $I$ is called a \emph{parabolic subgroup} of $W$.
Let $\Gamma_I$ denote the Coxeter graph of the Coxeter system $(W_I,I)$.

For two subsets $I,J \subseteq S$, we say that $I$ is \emph{adjacent to} $J$ if an element of $I$ is joined by an edge with an element of $J$ in the Coxeter graph $\Gamma$.
We say that $I$ is \emph{apart from} $J$ if $I \cap J = \emptyset$ and $I$ is not adjacent to $J$.
For the terminologies, we often abbreviate a set $\{s\}$ with a single element of $S$ to $s$ for simplicity.

\subsection{Root systems and reflection subgroups}
\label{sec:rootsystem}

Let $V$ denote the \emph{geometric representation space} of $(W,S)$, which is an $\mathbb{R}$-linear space equipped with a basis $\Pi = \{\alpha_s \mid s \in S\}$ and a $W$-invariant symmetric bilinear form $\langle \,,\, \rangle$ determined by
\begin{displaymath}
\langle \alpha_s, \alpha_t \rangle =
\begin{cases}
-\cos(\pi / m(s,t)) & \mbox{if } m(s,t) < \infty \enspace; \\
-1 & \mbox{if } m(s,t) = \infty \enspace,
\end{cases}
\end{displaymath}
where $W$ acts faithfully on $V$ by $s \cdot v=v-2\langle \alpha_s, v\rangle \alpha_s$ for $s \in S$ and $v \in V$.
Then the \emph{root system} $\Phi=W \cdot \Pi$ consists of unit vectors with respect to the bilinear form $\langle \,,\, \rangle$, and $\Phi$ is the disjoint union of $\Phi^+ := \Phi \cap \mathbb{R}_{\geq 0}\Pi$ and $\Phi^- := -\Phi^+$ where $\mathbb{R}_{\geq 0}\Pi$ signifies the set of nonnegative linear combinations of elements of $\Pi$.
Elements of $\Phi$, $\Phi^+$, and $\Phi^-$ are called \emph{roots}, \emph{positive roots}, and \emph{negative roots}, respectively.
For a subset $\Psi \subseteq \Phi$ and an element $w \in W$, define
\begin{displaymath}
\Psi^+ := \Psi \cap \Phi^+ \,,\, \Psi^- := \Psi \cap \Phi^- \,,\, 
\Psi[w] :=\{\gamma \in \Psi^+ \mid w \cdot \gamma \in \Phi^-\} \enspace.
\end{displaymath}
It is well known that the length $\ell(w)$ of $w$ is equal to $|\Phi[w]|$.

For an element $v=\sum_{s \in S}c_s\alpha_s$ of $V$, define the \emph{support} $\mathrm{Supp}\,v$ of $v$ to be the set of all $s \in S$ with $c_s \neq 0$.
For a subset $\Psi$ of $\Phi$, define the support $\mathrm{Supp}\,\Psi$ of $\Psi$ to be the union of $\mathrm{Supp}\,\gamma$ over all $\gamma \in \Psi$.
For each $I \subseteq S$, define
\begin{displaymath}
\Pi_I := \{\alpha_s \mid s \in I\} \subseteq \Pi \,,\, V_I := \mathrm{span}\,\Pi_I \subseteq V \,,\, \Phi_I := \Phi \cap V_I \enspace.
\end{displaymath}
It is well known that $\Phi_I$ coincides with the root system $W_I \cdot \Pi_I$ of $(W_I,I)$.
We notice the following well-known fact:
\begin{lem}
\label{lem:support_is_irreducible}
The support of any root $\gamma \in \Phi$ is irreducible.
\end{lem}
\begin{proof}
Note that $\gamma \in \Phi_I = W_I \cdot \Pi_I$, where $I= \mathrm{Supp}\,\gamma$.
On the other hand, it follows by induction on the length of $w$ that, for any $w \in W_I$ and $s \in I$, the support of $w \cdot \alpha_s$ is contained in the irreducible component of $I$ containing $s$.
Hence the claim follows.
\end{proof}

For a root $\gamma=w \cdot \alpha_s \in \Phi$, let $s_\gamma := wsw^{-1}$ be the \emph{reflection} along $\gamma$, which acts on $V$ by $s_\gamma \cdot v=v-2 \langle \gamma, v \rangle \gamma$ for $v \in V$.
For any subset $\Psi \subseteq \Phi$, let $W(\Psi)$ denote the \emph{reflection subgroup} of $W$ generated by $\{s_{\gamma} \mid \gamma \in \Psi\}$.
It was shown by Deodhar \cite{Deo_refsub} and by Dyer \cite{Dye} that $W(\Psi)$ is a Coxeter group.
To determine their generating set $S(\Psi)$ for $W(\Psi)$, let $\Pi(\Psi)$ denote the set of all \lq\lq simple roots'' $\gamma \in (W(\Psi) \cdot \Psi)^+$ in the \lq\lq root system'' $W(\Psi) \cdot \Psi$ of $W(\Psi)$, that is, all the $\gamma$ for which any expression $\gamma=\sum_{i=1}^{r}c_i\beta_i$ with $c_i>0$ and $\beta_i \in (W(\Psi) \cdot \Psi)^+$ satisfies that $\beta_i=\gamma$ for every index $i$.
Then the set $S(\Psi)$ is given by
\begin{displaymath}
S(\Psi) := \{s_\gamma \mid \gamma \in \Pi(\Psi)\} \enspace.
\end{displaymath}
We call $\Pi(\Psi)$ the \emph{simple system} of $(W(\Psi),S(\Psi))$.
Note that the \lq\lq root system'' $W(\Psi) \cdot \Psi$ and the simple system $\Pi(\Psi)$ for $(W(\Psi),S(\Psi))$ have several properties that are similar to the usual root systems $\Phi$ and simple systems $\Pi$ for $(W,S)$; see e.g., Theorem 2.3 of \cite{Nui11} for the detail.
In particular, we have the following result:
\begin{thm}
[{e.g., \cite[Theorem 2.3]{Nui11}}]
\label{thm:reflectionsubgroup_Deodhar}
Let $\Psi \subseteq \Phi$, and let $\ell_\Psi$ be the length function of $(W(\Psi),S(\Psi))$.
Then for $w \in W(\Psi)$ and $\gamma \in (W(\Psi) \cdot \Psi)^+$, we have $\ell_\Psi(ws_\gamma)<\ell_\Psi(w)$ if and only if $w \cdot \gamma \in \Phi^-$.
\end{thm}

We say that a subset $\Psi \subseteq \Phi^+$ is a \emph{root basis} if for each pair $\beta,\gamma \in \Psi$, we have
\begin{displaymath}
\begin{cases}
\langle \beta,\gamma \rangle=-\cos(\pi/m) & \mbox{if } s_\beta s_\gamma \mbox{ has order } m<\infty \enspace;\\
\langle \beta,\gamma \rangle \leq -1 & \mbox{if } s_\beta s_\gamma \mbox{ has infinite order}.
\end{cases}
\end{displaymath}
For example, it follows from Theorem \ref{thm:conditionforrootbasis} below that the simple system $\Pi(\Psi)$ of $(W(\Psi),S(\Psi))$ is a root basis for any $\Psi \subseteq \Phi$.
For two root bases $\Psi_1,\Psi_2 \subseteq \Phi^+$, we say that a mapping from $\Psi_1 = \Pi(\Psi_1)$ to $\Psi_2 = \Pi(\Psi_2)$ is an isomorphism if it induces an isomorphism from $S(\Psi_1)$ to $S(\Psi_2)$.
We show some properties of root bases:
\begin{thm}
[{\cite[Theorem 4.4]{Dye}}]
\label{thm:conditionforrootbasis}
Let $\Psi \subseteq \Phi^+$.
Then we have $\Pi(\Psi)=\Psi$ if and only if $\Psi$ is a root basis.
\end{thm}
\begin{prop}
[{\cite[Corollary 2.6]{Nui11}}]
\label{prop:fintyperootbasis}
Let $\Psi \subseteq \Phi^+$ be a root basis with $|W(\Psi)|<\infty$.
Then $\Psi$ is a basis of a positive definite subspace of $V$ with respect to the bilinear form $\langle \,,\, \rangle$.
\end{prop}
\begin{prop}
[{\cite[Proposition 2.7]{Nui11}}]
\label{prop:finitesubsystem}
Let $\Psi \subseteq \Phi^+$ be a root basis with $|W(\Psi)|<\infty$, and $U= \mathrm{span}\,\Psi$.
Then there exist an element $w \in W$ and a subset $I \subseteq S$ satisfying that $|W_I|<\infty$ and $w \cdot (U \cap \Phi^+)=\Phi_I^+$.
Moreover, the action of this $w$ maps $U \cap \Pi$ into $\Pi_I$.
\end{prop}

\subsection{Finite parabolic subgroups}
\label{sec:longestelement}

We say that a subset $I \subseteq S$ is of \emph{finite type} if $|W_I|<\infty$.
The finite irreducible Coxeter groups have been classified as summarized in \cite[Chapter 2]{Hum}.
Here we determine a labelling $r_1,r_2,\dots,r_n$ (where $n = |I|$) of elements of an irreducible subset $I \subseteq S$ of each finite type in the following manner, where the values $m(r_i,r_j)$ not listed here are equal to $2$ (see Figure \ref{fig:finite_irreducible_Coxeter_groups}):
\begin{description}
\item[Type $A_n$ ($1 \leq n < \infty$):] $m(r_i,r_{i+1})=3$ ($1 \leq i \leq n-1$);
\item[Type $B_n$ ($2 \leq n < \infty$):] $m(r_i,r_{i+1})=3$ ($1 \leq i \leq n-2$) and $m(r_{n-1},r_n)=4$;
\item[Type $D_n$ ($4 \leq n < \infty$):] $m(r_i,r_{i+1})=m(r_{n-2},r_n)=3$ ($1 \leq i \leq n-2$);
\item[Type $E_n$ ($n=6,7,8$):] $m(r_1,r_3)=m(r_2,r_4)=m(r_i,r_{i+1})=3$ ($3 \leq i \leq n-1$);
\item[Type $F_4$:] $m(r_1,r_2)=m(r_3,r_4)=3$ and $m(r_2,r_3)=4$;
\item[Type $H_n$ ($n=3,4$):] $m(r_1,r_2)=5$ and $m(r_i,r_{i+1})=3$ ($2 \leq i \leq n-1$);
\item[Type $I_2(m)$ ($5 \leq m < \infty$):] $m(r_1,r_2)=m$.
\end{description}
We call the above labelling $r_1,\dots,r_n$ the \emph{standard labelling} of $I$.
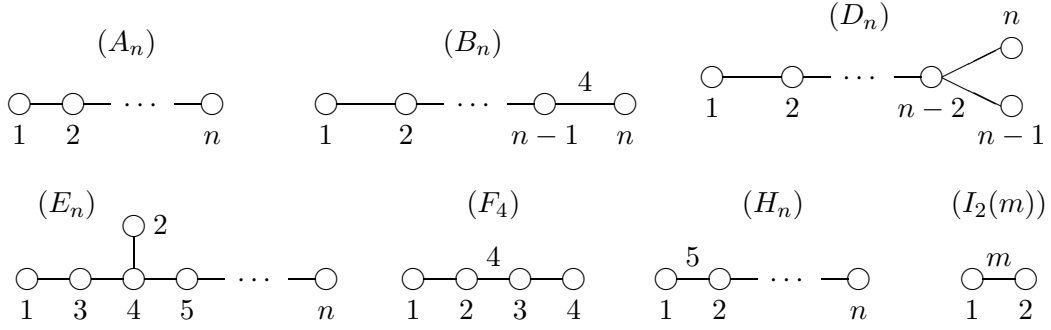
\begin{figure}[hbt]
\centering
\begin{picture}(100,45)(0,-45)
\put(50,-10){\hbox to0pt{\hss{($A_n$)}\hss}}
\multiput(10,-30)(20,0){2}{\circle{8}}
\put(14,-30){\line(1,0){12}}
\put(34,-30){\line(1,0){10}}
\put(56,-33){\hbox to0pt{\hss$\cdots$\hss}}
\put(68,-30){\line(1,0){10}}
\put(82,-30){\circle{8}}
\put(10,-45){\hbox to0pt{\hss$1$\hss}}
\put(30,-45){\hbox to0pt{\hss$2$\hss}}
\put(82,-45){\hbox to0pt{\hss$n$\hss}}
\end{picture}
\quad
\begin{picture}(130,45)(0,-45)
\put(65,-10){\hbox to0pt{\hss{($B_n$)}\hss}}
\multiput(10,-30)(30,0){2}{\circle{8}}
\put(14,-30){\line(1,0){22}}
\put(44,-30){\line(1,0){10}}
\put(66,-33){\hbox to0pt{\hss$\cdots$\hss}}
\put(78,-30){\line(1,0){10}}
\multiput(92,-30)(30,0){2}{\circle{8}}
\put(96,-30){\line(1,0){22}}
\put(107,-25){\hbox to0pt{\hss$4$\hss}}
\put(10,-45){\hbox to0pt{\hss$1$\hss}}
\put(40,-45){\hbox to0pt{\hss$2$\hss}}
\put(92,-45){\hbox to0pt{\hss$n-1$\hss}}
\put(122,-45){\hbox to0pt{\hss$n$\hss}}
\end{picture}
\quad
\begin{picture}(140,55)(0,-55)
\put(65,-10){\hbox to0pt{\hss{($D_n$)}\hss}}
\multiput(10,-30)(30,0){2}{\circle{8}}
\put(14,-30){\line(1,0){22}}
\put(44,-30){\line(1,0){10}}
\put(66,-33){\hbox to0pt{\hss$\cdots$\hss}}
\put(78,-30){\line(1,0){10}}
\put(92,-30){\circle{8}}
\multiput(122,-19)(0,-22){2}{\circle{8}}
\put(96,-30){\line(2,1){22}}
\put(96,-30){\line(2,-1){22}}
\put(10,-45){\hbox to0pt{\hss$1$\hss}}
\put(40,-45){\hbox to0pt{\hss$2$\hss}}
\put(92,-45){\hbox to0pt{\hss$n-2$\hss}}
\put(122,-55){\hbox to0pt{\hss$n-1$\hss}}
\put(122,-10){\hbox to0pt{\hss$n$\hss}}
\end{picture}
\\
\begin{picture}(130,65)(0,-65)
\put(25,-25){\hbox to0pt{\hss{($E_n$)}\hss}}
\multiput(10,-50)(20,0){4}{\circle{8}}
\multiput(14,-50)(20,0){3}{\line(1,0){12}}
\put(74,-50){\line(1,0){10}}
\put(96,-53){\hbox to0pt{\hss$\cdots$\hss}}
\put(108,-50){\line(1,0){10}}
\put(122,-50){\circle{8}}
\put(50,-46){\line(0,1){12}}
\put(50,-30){\circle{8}}
\put(10,-65){\hbox to0pt{\hss$1$\hss}}
\put(60,-33){\hbox to0pt{\hss$2$\hss}}
\put(30,-65){\hbox to0pt{\hss$3$\hss}}
\put(50,-65){\hbox to0pt{\hss$4$\hss}}
\put(70,-65){\hbox to0pt{\hss$5$\hss}}
\put(122,-65){\hbox to0pt{\hss$n$\hss}}
\end{picture}
\quad
\begin{picture}(80,65)(0,-65)
\put(40,-25){\hbox to0pt{\hss{($F_4$)}\hss}}
\multiput(10,-50)(20,0){4}{\circle{8}}
\multiput(14,-50)(20,0){3}{\line(1,0){12}}
\put(40,-45){\hbox to0pt{\hss$4$\hss}}
\put(10,-65){\hbox to0pt{\hss$1$\hss}}
\put(30,-65){\hbox to0pt{\hss$2$\hss}}
\put(50,-65){\hbox to0pt{\hss$3$\hss}}
\put(70,-65){\hbox to0pt{\hss$4$\hss}}
\end{picture}
\quad
\begin{picture}(100,55)(0,-55)
\put(50,-15){\hbox to0pt{\hss{($H_n$)}\hss}}
\multiput(10,-40)(20,0){2}{\circle{8}}
\put(14,-40){\line(1,0){12}}
\put(34,-40){\line(1,0){10}}
\put(56,-43){\hbox to0pt{\hss$\cdots$\hss}}
\put(68,-40){\line(1,0){10}}
\put(82,-40){\circle{8}}
\put(20,-35){\hbox to0pt{\hss$5$\hss}}
\put(10,-55){\hbox to0pt{\hss$1$\hss}}
\put(30,-55){\hbox to0pt{\hss$2$\hss}}
\put(82,-55){\hbox to0pt{\hss$n$\hss}}
\end{picture}
\quad
\begin{picture}(40,55)(0,-55)
\put(20,-15){\hbox to0pt{\hss{($I_2(m)$)}\hss}}
\multiput(10,-40)(20,0){2}{\circle{8}}
\put(14,-40){\line(1,0){12}}
\put(20,-35){\hbox to0pt{\hss$m$\hss}}
\put(10,-55){\hbox to0pt{\hss$1$\hss}}
\put(30,-55){\hbox to0pt{\hss$2$\hss}}
\end{picture}
\caption{Coxeter graphs of the finite irreducible Coxeter groups (here we write $i$ instead of $r_i$ for each vertex)}
\label{fig:finite_irreducible_Coxeter_groups}
\end{figure}

Let $w_0(I)$ denote the (unique) longest element of a finite parabolic subgroup $W_I$.
It is well known that $w_0(I)^2 = 1$ and $w_0(I) \cdot \Pi_I = -\Pi_I$.
Now let $I$ be irreducible of finite type.
If $I$ is of type $A_n$ ($n \geq 2$), $D_k$ ($k$ odd), $E_6$ or $I_2(m)$ ($m$ odd), then the automorphism of the Coxeter graph $\Gamma_I$ of $W_I$ induced by (the conjugation action of) $w_0(I)$ is the unique nontrivial automorphism of $\Gamma_I$.
Otherwise, $w_0(I)$ lies in the center $Z(W_I)$ of $W_I$ and the induced automorphism of $\Gamma_I$ is trivial, in which case we say that $I$ is of \emph{$(-1)$-type}.
Moreover, if $W_I$ is finite but not irreducible, then $w_0(I)=w_0(I_1) \dotsm w_0(I_k)$ where the $I_i$ are the irreducible components of $I$.

\section{Known properties of the centralizers}
\label{sec:properties_centralizer}

This section summarizes some known properties (mainly proven in \cite{Nui11}) of the centralizers $Z_W(W_I)$ of parabolic subgroups $W_I$ in Coxeter groups $W$, especially those relevant to the argument in this paper.

First, we fix an abstract index set $\Lambda$ with $|\Lambda| = |I|$, and define $S^{(\Lambda)}$ to be the set of all injective mappings $x \colon \Lambda \to S$.
For $x \in S^{(\Lambda)}$ and $\lambda \in \Lambda$, we put $x_\lambda = x(\lambda)$; thus $x$ may be regarded as a duplicate-free \lq\lq $\Lambda$-tuple'' $(x_\lambda)=(x_\lambda)_{\lambda \in \Lambda}$ of elements of $S$.
For each $x \in S^{(\Lambda)}$, let $[x]$ denote the image of the mapping $x$; $[x] = \{x_{\lambda} \mid \lambda \in \Lambda\}$.
In the following argument, we fix an element $x_I \in S^{(\Lambda)}$ with $[x_I] = I$.
We define
\begin{displaymath}
C_{x,y} := \{w \in W \mid \alpha_{x_\lambda}=w \cdot \alpha_{y_\lambda} \mbox{ for every } \lambda \in \Lambda\} \mbox{ for } x,y \in S^{(\Lambda)} \enspace.
\end{displaymath}
Note that $C_{x,y} \cdot C_{y,z} \subseteq C_{x,z}$ and $C_{x,y}{}^{-1} = C_{y,x}$ for $x,y,z \in S^{(\Lambda)}$.
Now we define
\begin{displaymath}
w \ast y_{\lambda} := x_{\lambda} \mbox{ for } x,y \in S^{(\Lambda)}, w \in C_{x,y} \mbox{ and } \lambda \in \Lambda \enspace,
\end{displaymath}
therefore we have $w \cdot \alpha_s = \alpha_{w \ast s}$ for any $w \in C_{x,y}$ and $s \in [y]$.
(This $\ast$ can be interpreted as the conjugation action of elements of $C_{x,y}$ to the elements of $[y]$.)
Moreover, we define
\begin{displaymath}
w \ast y := x \mbox{ for } x,y \in S^{(\Lambda)} \mbox{ and } w \in C_{x,y}
\end{displaymath}
(this $\ast$ can be interpreted as the diagonal action on the $\Lambda$-tuples).
We define $C_I = C_{x_I,x_I}$, therefore we have
\begin{displaymath}
C_I = \{w \in W \mid w \cdot \alpha_s=\alpha_s \mbox{ for every } s \in I\} \enspace,
\end{displaymath}
which is a normal subgroup of $Z_W(W_I)$.

To describe generators of $C_I$, we introduce some notations.
For subsets $J,K \subseteq S$, let $J_{\sim K}$ denote the set of elements of $J \cup K$ that belongs to the same connected component of $\Gamma_{J \cup K}$ as an element of $K$.
Now for $x \in S^{(\Lambda)}$ and $s \in S \smallsetminus [x]$ for which $[x]_{\sim s}$ is of finite type, there exists a unique $y \in S^{(\Lambda)}$ for which the element
\begin{displaymath}
w_x^s := w_0([x]_{\sim s})w_0([x]_{\sim s} \smallsetminus \{s\})
\end{displaymath}
belongs to $C_{y,x}$.
In this case, we define
\begin{displaymath}
\varphi(x,s) := y \enspace,
\end{displaymath}
therefore $\varphi(x,s) = w_x^s \ast x$ in the above notations.
We have the following result:
\begin{prop}
[{see \cite[Theorem 3.5(iii)]{Nui11}}]
\label{prop:generator_C}
Let $x,y \in S^{(\Lambda)}$ and $w \in C_{x,y}$.
Then there are a finite sequence $z_0 = y,z_1,\dots,z_{n-1},z_n = x$ of elements of $S^{(\Lambda)}$ and a finite sequence $s_0,s_1,\dots,s_{n-1}$ of elements of $S$ satisfying that $s_i \not\in [z_i]$, $[z_i]_{\sim s_i}$ is of finite type and $\varphi(z_i,s_i) = z_{i+1}$ for each index $0 \leq i \leq n-1$, and we have $w = w_{z_{n-1}}^{s_{n-1}} \cdots w_{z_1}^{s_1} w_{z_0}^{s_0}$.
\end{prop}

For subsets $J,K \subseteq S$, define
\begin{displaymath}
\Phi_J^{\perp K} := \{\gamma \in \Phi_J \mid \langle \gamma,\alpha_s \rangle = 0 \mbox{ for every } s \in K\} \,,\, W_J^{\perp K} := W(\Phi_J^{\perp K})
\end{displaymath}
(see Section \ref{sec:rootsystem} for notations).
Then $(W_J^{\perp K},R^{J,K})$ is a Coxeter system with root system $\Phi_J^{\perp K}$ and simple system $\Pi^{J,K}$, where
\begin{displaymath}
R^{J,K} := S(\Phi_J^{\perp K}) \,,\, \Pi^{J,K} := \Pi(\Phi_J^{\perp K})
\end{displaymath}
(see \cite[Section 3.1]{Nui11}).
In the notations, the symbol $J$ will be omitted when $J=S$; hence we have
\begin{displaymath}
W^{\perp I}=W_S^{\perp I}=\langle \{s_\gamma \mid \gamma \in \Phi^{\perp I}\} \rangle \enspace.
\end{displaymath}
On the other hand, we define
\begin{displaymath}
Y_{x,y} := \{w \in C_{x,y} \mid w \cdot (\Phi^{\perp [y]})^+ \subseteq \Phi^+\} \mbox{ for } x,y \in S^{(\Lambda)} \enspace.
\end{displaymath}
Note that $Y_{x,y} = \{w \in C_{x,y} \mid (\Phi^{\perp [x]})^+=w \cdot (\Phi^{\perp [y]})^+\}$ (see \cite[Section 3.1]{Nui11}).
Note also that $Y_{x,y} \cdot Y_{y,z} \subseteq Y_{x,z}$ and $Y_{x,y}{}^{-1} = Y_{y,x}$ for $x,y,z \in S^{(\Lambda)}$.
Now we define $Y_I = Y_{x_I,x_I}$, therefore we have
\begin{displaymath}
Y_I = \{w \in C_I \mid (\Phi^{\perp I})^+ = w \cdot (\Phi^{\perp I})^+\} \enspace.
\end{displaymath}
We have the following results:
\begin{prop}
[{see \cite[Lemma 4.1]{Nui11}}]
\label{prop:charofBphi}
For $x \in S^{(\Lambda)}$ and $s \in S \smallsetminus [x]$, the three conditions are equivalent:
\begin{enumerate}
\item $[x]_{\sim s}$ is of finite type, and $\varphi(x,s) = x$;
\item $[x]_{\sim s}$ is of finite type, and $\Phi^{\perp [x]}[w_x^s] \neq \emptyset$;
\item $\Phi_{[x] \cup \{s\}}^{\perp [x]} \neq \emptyset$.
\end{enumerate}
If these three conditions are satisfied, then we have $\Phi^{\perp [x]}[w_x^s]=(\Phi_{[x] \cup \{s\}}^{\perp [x]})^+=\{\gamma(x,s)\}$ for a unique positive root $\gamma(x,s)$ satisfying $s_{\gamma(x,s)}=w_x^s$.
\end{prop}
\begin{prop}
\label{prop:factorization_C}
Let $x,y \in S^{(\Lambda)}$.
\begin{enumerate}
\item \label{item:prop_factorization_C_decomp}
(See \cite[Theorem 4.6(i)(iv)]{Nui11}.)
The group $C_{x,x}$ admits a semidirect product decomposition $C_{x,x} = W^{\perp [x]} \rtimes Y_{x,x}$.
Moreover, if $w \in Y_{x,y}$, then the conjugation action by $w$ defines an isomorphism $u \mapsto wuw^{-1}$ of Coxeter systems from $(W^{\perp [y]},R^{[y]})$ to $(W^{\perp [x]},R^{[x]})$.
\item \label{item:prop_factorization_C_generator_Y}
(See \cite[Theorem 4.6(ii)]{Nui11}.)
Let $w \in Y_{x,y}$.
Then there are a finite sequence $z_0 = y,z_1,\dots,z_{n-1},z_n = x$ of elements of $S^{(\Lambda)}$ and a finite sequence $s_0,s_1,\dots,s_{n-1}$ of elements of $S$ satisfying that $z_{i+1} \neq z_i$, $s_i \not\in [z_i]$, $[z_i]_{\sim s_i}$ is of finite type and $w_{z_i}^{s_i} \in Y_{z_{i+1},z_i}$ for each index $0 \leq i \leq n-1$, and we have $w = w_{z_{n-1}}^{s_{n-1}} \cdots w_{z_1}^{s_1} w_{z_0}^{s_0}$.
\item \label{item:prop_factorization_C_generator_perp}
(See \cite[Theorem 4.13]{Nui11}.)
The generating set $R^{[x]}$ of $W^{\perp [x]}$ consists of elements of the form $w s_{\gamma(y,s)} w^{-1}$ satisfying that $y \in S^{(\Lambda)}$, $w \in Y_{x,y}$ and $\gamma(y,s)$ is a positive root as in the statement of Proposition \ref{prop:charofBphi} (hence $[y]_{\sim s}$ is of finite type and $\varphi(y,s) = y$).
\end{enumerate}
\end{prop}
\begin{prop}
[{see \cite[Proposition 4.8]{Nui11}}]
\label{prop:Yistorsionfree}
For any $x \in S^{(\Lambda)}$, the group $Y_{x,x}$ is torsion-free.
\end{prop}

For the structure of the entire centralizer $Z_W(W_I)$, a general result (Theorem 5.2 of \cite{Nui11}) implies the following proposition in a special case (a proof of the proposition from Theorem 5.2 of \cite{Nui11} is straightforward by noticing the fact that, under the hypothesis of the following proposition, the group $\mathcal{A}$ defined in the last paragraph before Theorem 5.2 of \cite{Nui11} is trivial and hence the group $B_I$ used in Theorem 5.2 of \cite{Nui11} coincides with $Y_I$):
\begin{prop}
[{see \cite[Theorem 5.2]{Nui11}}]
\label{prop:Z_for_special_case}
If every irreducible component of $I$ of finite type is of $(-1)$-type (see Section \ref{sec:longestelement} for the terminology), then we have $Z_W(W_I) = Z(W_I) \times (W^{\perp I} \rtimes Y_I)$.
\end{prop}

We also present an auxiliary result, which will be used later:
\begin{lem}
[{see \cite[Lemma 3.2]{Nui11}}]
\label{lem:rightdivisor}
Let $w \in W$ and $J,K \subseteq S$, and suppose that $w \cdot \Pi_J \subseteq \Pi$ and $w \cdot \Pi_K \subseteq \Phi^-$.
Then $J \cap K=\emptyset$, the set $J_{\sim K}$ is of finite type, and $w_0(J_{\sim K})w_0(J_{\sim K} \smallsetminus K)$ is a right divisor of $w$ (see Section \ref{sec:defofCox} for the terminology).
\end{lem}

\section{Main results}
\label{sec:main_result}

In this section, we state the main results of this paper, and give some relevant remarks.
The proof will be given in the following sections.

The main results deal with the relations between the \lq\lq finite part'' of the reflection subgroup $W^{\perp I}$ and the subgroup $Y_I$ of the centralizer $Z_W(W_I)$.
In general, for any Coxeter group $W$, the product of the finite irreducible components of $W$ is called the \emph{finite part} of $W$; here we write it as $W_{\mathrm{fin}}$.
Then, since $W^{\perp I}$ is a Coxeter group (with generating set $R^I$ and simple system $\Pi^I$) as mentioned in Section \ref{sec:properties_centralizer}, $W^{\perp I}$ has its own finite part $W^{\perp I}{}_{\mathrm{fin}}$.

To state the main theorem, we introduce a terminology: We say that a subset $I$ of $S$ is \emph{$A_{>1}$-free} if $I$ has no irreducible components of type $A_n$ with $2 \leq n < \infty$.
Then the main theorem of this paper is stated as follows:
\begin{thm}
\label{thm:YfixesWperpIfin}
Let $I$ be an $A_{>1}$-free subset of $S$ (see above for the terminology).
Then for each $\gamma \in \Pi^I$ with $s_\gamma \in W^{\perp I}{}_{\mathrm{fin}}$, we have $w \cdot \gamma = \gamma$ for every $w \in Y_I$.
Hence each element of the subgroup $Y_I$ of $Z_W(W_I)$ commutes with every element of $W^{\perp I}{}_{\mathrm{fin}}$.
\end{thm}
Among the several cases for the subset $I$ of $S$ covered by Theorem \ref{thm:YfixesWperpIfin}, we emphasize the following important special case:
\begin{cor}
\label{cor:YfixesWperpIfin}
Let $I \subseteq S$.
If every irreducible component of $I$ of finite type is of $(-1)$-type (see Section \ref{sec:longestelement} for the terminology), then we have
\begin{displaymath}
Z_W(W_I) = Z(W_I) \times W^{\perp I}{}_{\mathrm{fin}} \times (W^{\perp I}{}_{\mathrm{inf}} \rtimes Y_I) \enspace,
\end{displaymath}
where $W^{\perp I}{}_{\mathrm{inf}}$ denotes the product of the infinite irreducible components of $W^{\perp I}$ (hence $W^{\perp I} = W^{\perp I}{}_{\mathrm{fin}} \times W^{\perp I}{}_{\mathrm{inf}}$).
\end{cor}
\begin{proof}
Note that the assumption on $I$ in Theorem \ref{thm:YfixesWperpIfin} is now satisfied.
In this situation, Proposition \ref{prop:Z_for_special_case} implies that $Z_W(W_I) = Z(W_I) \times (W^{\perp I} \rtimes Y_I)$.
Now by Theorem \ref{thm:YfixesWperpIfin}, both $Y_I$ and $W^{\perp I}{}_{\mathrm{inf}}$ centralize $W^{\perp I}{}_{\mathrm{fin}}$, therefore the latter factor of $Z_W(W_I)$ decomposes further as $W^{\perp I}{}_{\mathrm{fin}} \times (W^{\perp I}{}_{\mathrm{inf}} \rtimes Y_I)$.
\end{proof}
We notice that the conclusion of Theorem \ref{thm:YfixesWperpIfin} will not generally hold when we remove the $A_{>1}$-freeness assumption on $I$.
A counterexample will be given in Section \ref{sec:counterexample}.

Here we give a remark on an application of the main results to a study of the isomorphism problem in Coxeter groups.
An important branch in the research on the isomorphism problem in Coxeter groups is to investigate, for two Coxeter systems $(W,S)$, $(W',S')$ and a group isomorphism $f \colon W \to W'$, the possibilities of \lq\lq shapes'' of the images $f(r) \in W'$ by $f$ of reflections $r \in W$ (with respect to the generating set $S$); for example, whether $f(r)$ is always a reflection in $W'$ (with respect to $S'$) or not.
Now if $r \in S$, then Corollary \ref{cor:YfixesWperpIfin} and Proposition \ref{prop:Yistorsionfree} imply that the unique maximal reflection subgroup of the centralizer of $r$ in $W$ is $\langle r \rangle \times W^{\perp \{r\}}$, which has finite part $\langle r \rangle \times W^{\perp \{r\}}{}_{\mathrm{fin}}$.
Moreover, the property of $W^{\perp \{r\}}{}_{\mathrm{fin}}$ shown in Theorem \ref{thm:YfixesWperpIfin} can imply that the factor $W^{\perp \{r\}}{}_{\mathrm{fin}}$ becomes \lq\lq frequently'' almost trivial.
In such a case, the finite part of the unique maximal reflection subgroup of the centralizer of $f(r)$ in $W'$ should be very small, which can be shown to be impossible if $f(r)$ is too far from being a reflection.
Thus the possibilities of the shape of $f(r)$ in $W'$ can be restricted by using Theorem \ref{thm:YfixesWperpIfin}.
See \cite{Nui_ref} for a detailed study along this direction.
The author hope that such an argument can be generalized to the case that $r$ is not a reflection but an involution of \lq\lq type'' which is $A_{>1}$-free (in a certain appropriate sense).

\section{Proof of Theorem \ref{thm:YfixesWperpIfin}: General properties}
\label{sec:proof_general}

In this and the next sections, we give a proof of Theorem \ref{thm:YfixesWperpIfin}.
First, this section gives some preliminary results that hold for an arbitrary $I \subseteq S$ (not necessarily $A_{>1}$-free; see Section \ref{sec:main_result} for the terminology).
Then the next section will focus on the case that $I$ is $A_{>1}$-free as in Theorem \ref{thm:YfixesWperpIfin} and complete the proof of Theorem \ref{thm:YfixesWperpIfin}.

\subsection{Decompositions of elements of $Y_{z,y}$}
\label{sec:finitepart_decomposition_Y}

It is mentioned in Proposition \ref{prop:factorization_C}(\ref{item:prop_factorization_C_generator_Y}) that each element $u \in Y_{z,y}$ with $y,z \in S^{(\Lambda)}$ admits a kind of decomposition into elements of some $Y$.
Here we introduce a generalization of such decompositions, which will play an important role below.
We give a definition:
\begin{defn}
\label{defn:standard_decomposition}
Let $u \in Y_{z,y}$ with $y,z \in S^{(\Lambda)}$.
We say that an expression $\mathcal{D} := \omega_{n-1} \cdots \omega_1\omega_0$ of $u$ is a \emph{semi-standard decomposition} of $u$ with respect to a subset $J$ of $S$ if there exist $y^{(i)} = y^{(i)}(\mathcal{D}) \in S^{(\Lambda)}$ for $0 \leq i \leq n$, $t^{(i)} = t^{(i)}(\mathcal{D}) \in S$ for $0 \leq i \leq n-1$ and $J^{(i)} = J^{(i)}(\mathcal{D}) \subseteq S$ for $0 \leq i \leq n$, with $y^{(0)} = y$, $y^{(n)} = z$ and $J^{(0)} = J$, satisfying the following conditions for each index $0 \leq i \leq n-1$:
\begin{itemize}
\item We have $t^{(i)} \not\in [y^{(i)}] \cup J^{(i)}$ and $t^{(i)}$ is adjacent to $[y^{(i)}]$.
\item The subset $K^{(i)} = K^{(i)}(\mathcal{D}) := ([y^{(i)}] \cup J^{(i)})_{\sim t^{(i)}}$ of $S$ is of finite type (see Section \ref{sec:properties_centralizer} for the notation).
\item We have $\omega_i = \omega_{y^{(i)},J^{(i)}}^{t^{(i)}} := w_0(K^{(i)})w_0(K^{(i)} \smallsetminus \{t^{(i)}\})$.
\item We have $\omega_i \in Y_{y^{(i+1)},y^{(i)}}$ and $\omega_i \cdot \Pi_{J^{(i)}} = \Pi_{J^{(i+1)}}$.
\end{itemize}
We call the above subset $K^{(i)}$ of $S$ the \emph{support} of $\omega_i$.
We call a component $\omega_i$ of $\mathcal{D}$ a \emph{wide transformation} if its support $K^{(i)}$ intersects with $J^{(i)} \smallsetminus [y^{(i)}]$; otherwise, we call $\omega_i$ a \emph{narrow transformation}, in which case we have $\omega_i = \omega_{y^{(i)},J^{(i)}}^{t^{(i)}} = w_{y^{(i)}}^{t^{(i)}}$.
Moreover, we say that $\mathcal{D} = \omega_{n-1} \cdots \omega_1\omega_0$ is a \emph{standard decomposition} of $u$ if $\mathcal{D}$ is a semi-standard decomposition of $u$ and $\ell(u) = \sum_{j=0}^{n-1} \ell(\omega_j)$.
The integer $n$ is called the \emph{length} of $\mathcal{D}$ and is denoted by $\ell(\mathcal{D})$.
\end{defn}
\begin{exmp}
\label{exmp:semi-standard_decomposition}
We give an example of a semi-standard decomposition.
Let $(W,S)$ be a Coxeter system of type $D_7$, with standard labelling $r_1,\dots,r_7$ of elements of $S$ given in Section \ref{sec:longestelement}.
We put $n := 4$, and define the objects $y^{(i)}$, $t^{(i)}$ and $J^{(i)}$ as in Table \ref{tab:example_semi-standard_decomposition}, where we abbreviate each $r_i$ to $i$ for simplicity.
In this case, the subsets $K^{(i)}$ of $S$ introduced in Definition \ref{defn:standard_decomposition} are determined as in the last row of Table \ref{tab:example_semi-standard_decomposition}.
We have
\begin{displaymath}
\begin{split}
\omega_0 &= w_0(\{r_1,r_2,r_3,r_4,r_5\})w_0(\{r_1,r_2,r_3,r_5\}) = r_2r_3r_4r_5r_1r_2r_3r_4 \enspace, \\
\omega_1 &= w_0(\{r_3,r_4,r_5,r_6\})w_0(\{r_3,r_4,r_5\}) = r_3r_4r_5r_6 \enspace, \\
\omega_2 &= w_0(\{r_4,r_5,r_6,r_7\})w_0(\{r_4,r_5,r_6\}) = r_7r_5r_4r_6r_5r_7 \enspace, \\
\omega_3 &= w_0(\{r_3,r_4,r_5,r_6\})w_0(\{r_4,r_5,r_6\}) = r_6r_5r_4r_3 \enspace.
\end{split}
\end{displaymath}
Let $u$ denote the element $\omega_3\omega_2\omega_1\omega_0$ of $W$.
Then it can be shown that $u \in Y_{z,y}$ where $y := y^{(0)} = (r_1,r_2,r_3)$ and $z := y^{(n)} = (r_5,r_4,r_3)$, and the expression $\mathcal{D} = \omega_3\omega_2\omega_1\omega_0$ is a semi-standard decomposition of $u$ of length $4$ with respect to $J := J^{(0)} = \{r_5\}$.
Moreover, $\mathcal{D}$ is in fact a standard decomposition of $u$ (which is the same as the one obtained by using Proposition \ref{prop:standard_decomposition_existence} below).
Among the four component $\omega_i$, the first one $\omega_0$ is a wide transformation and the other three $\omega_1,\omega_2,\omega_3$ are narrow transformations.
\end{exmp}
\begin{table}[hbt]
\centering
\caption{The data for the example of semi-standard decompositions}
\label{tab:example_semi-standard_decomposition}
\begin{tabular}{|c||c|c|c|c|c|} \hline
$i$ & $4$ & $3$ & $2$ & $1$ & $0$ \\ \hline\hline
$y^{(i)}$ & $(5,4,3)$ & $(6,5,4)$ & $(4,5,6)$ & $(3,4,5)$ & $(1,2,3)$ \\ \hline
$t^{(i)}$ & --- & $3$ & $7$ & $6$ & $4$ \\ \hline
$J^{(i)}$ & $\{1\}$ & $\{1\}$ & $\{1\}$ & $\{1\}$ & $\{5\}$ \\ \hline
$K^{(i)}$ & --- & $\{3,4,5,6\}$ & $\{4,5,6,7\}$ & $\{3,4,5,6\}$ & $\{1,2,3,4,5\}$ \\ \hline
\end{tabular}
\end{table}

The next proposition shows existence of standard decompositions:
\begin{prop}
\label{prop:standard_decomposition_existence}
Let $u \in Y_{z,y}$ with $y,z \in S^{(\Lambda)}$, and let $J \subseteq S$ satisfying that $u \cdot \Pi_J \subseteq \Pi$.
Then there exists a standard decomposition of $u$ with respect to $J$.
\end{prop}
\begin{proof}
We proceed the proof by induction on $\ell(u)$.
For the case $\ell(u) = 0$, i.e., $u = 1$, the empty expression satisfies the conditions for a standard decomposition of $u$.
From now, we consider the case $\ell(u) > 0$.
Then there is an element $t = t^{(0)} \in S$ satisfying that $u \cdot \alpha_t \in \Phi^-$.
Since $u \in Y_{z,y}$ and $u \cdot \Pi_J \subseteq \Pi \subseteq \Phi^+$, we have $t \not\in [y] \cup J$ and $\alpha_t \not\in \Phi^{\perp [y]}$, therefore $t$ is adjacent to $[y]$.
Now by Lemma \ref{lem:rightdivisor}, $K = K^{(0)} := ([y] \cup J)_{\sim t}$ is of finite type and $\omega_0 := \omega_{y,J}^{t}$ is a right divisor of $u$ (see Section \ref{sec:defofCox} for the terminology).
By the definition of $\omega_{y,J}^{t}$ in Definition \ref{defn:standard_decomposition}, there exist unique $y^{(1)} \in S^{(\Lambda)}$ and $J^{(1)} \subseteq S$ satisfying that $y^{(1)} = \omega_0 \ast y$ (see Section \ref{sec:properties_centralizer} for the notation) and $\omega_0 \cdot \Pi_J = \Pi_{J^{(1)}}$.
Moreover, since $\omega_0$ is a right divisor of $u$, it follows that $\Phi[\omega_0] \subseteq \Phi[u]$ (see e.g., Lemma 2.2 of \cite{Nui11}), therefore $\Phi^{\perp [y]}[\omega_0] \subseteq \Phi^{\perp [y]}[u] = \emptyset$ and $\omega_0 \in Y_{y^{(1)},y}$.
Put $u' = u\omega_0{}^{-1}$.
Then we have $u' \in Y_{z,y^{(1)}}$, $u' \cdot \Pi_{J^{(1)}} \subseteq \Pi$ and $\ell(u') = \ell(u) - \ell(\omega_0) < \ell(u)$ (note that $\omega_0 \neq 1$).
Hence the concatenation of $\omega_0$ to a standard decomposition of $u' \in Y_{z,y^{(1)}}$ with respect to $J^{(1)}$ obtained by the induction hypothesis gives a desired standard decomposition of $u$.
\end{proof}

We present some properties of (semi-)standard decompositions.
First, we have the following:
\begin{lem}
\label{lem:another_decomposition_Y_no_loop}
For any semi-standard decomposition $\omega_{n-1} \cdots \omega_1\omega_0$ of an element of $W$, for each $0 \leq i \leq n-1$, there exists an element of $\Pi_{K^{(i)} \smallsetminus \{t^{(i)}\}}$ which is not fixed by $\omega_i$.
\end{lem}
\begin{proof}
Assume contrary that $\omega_i$ fixes $\Pi_{K^{(i)} \smallsetminus \{t^{(i)}\}}$ pointwise.
Then by applying Proposition \ref{prop:charofBphi} to the pair of $[y^{(i)}] \cup J^{(i)}$ and $t^{(i)}$ instead of the pair of $[x]$ and $s$, it follows that there exists a root $\gamma \in (\Phi_{K^{(i)}}^{\perp K^{(i)} \smallsetminus \{t^{(i)}\}})^+$ with $\omega_i \cdot \gamma \in \Phi^-$ (note that, in this case, the element $w_x^s$ in Proposition \ref{prop:charofBphi} coincides with $\omega_i$).
By the definition of the support $K^{(i)}$ of $\omega_i$, $K^{(i)}$ is apart from $[y^{(i)}] \smallsetminus K^{(i)}$, therefore this root $\gamma$ also belongs to $(\Phi^{\perp [y^{(i)}]})^+$.
Hence we have $\Phi^{\perp [y^{(i)}]}[\omega_i] \neq \emptyset$, contradicting the property $\omega_i \in Y_{y^{(i+1)},y^{(i)}}$ in Definition \ref{defn:standard_decomposition}.
Hence Lemma \ref{lem:another_decomposition_Y_no_loop} holds.
\end{proof}

For a semi-standard decomposition $\mathcal{D} = \omega_n \cdots \omega_1\omega_0$ of $u \in Y_{z,y}$, let $0 \leq i_1 < i_2 < \cdots < i_k \leq n$ be the indices $i$ with the property that $[y^{(i+1)}(\mathcal{D})] = [y^{(i)}(\mathcal{D})]$ and $J^{(i+1)}(\mathcal{D}) = J^{(i)}(\mathcal{D})$.
Then we define the \emph{simplification} $\widehat{\mathcal{D}}$ of $\mathcal{D}$ to be the expression $\omega_n \cdots \hat{\omega_{i_k}} \cdots \hat{\omega_{i_1}} \cdots \hat{\omega_{i_0}} \cdots \omega_0$ obtained from $\mathcal{D} = \omega_n \cdots \omega_1\omega_0$ by removing all terms $\omega_{i_j}$ with $1 \leq j \leq k$.
Let $\widehat{u}$ denote the element of $W$ expressed by the product $\widehat{\mathcal{D}}$.
The following lemma is straightforward to prove:
\begin{lem}
\label{lem:another_decomposition_Y_reduce_redundancy}
In the above setting, let $\sigma$ denote the mapping from $\{0,1,\dots,n-k\}$ to $\{0,1,\dots,n\}$ satisfying that $\widehat{\mathcal{D}} = \omega_{\sigma(n-k)} \cdots \omega_{\sigma(1)}\omega_{\sigma(0)}$.
Then we have $\widehat{u} \in Y_{\widehat{z},y}$ for some $\widehat{z} \in S^{(\Lambda)}$ with $[\widehat{z}] = [z]$; $\widehat{\mathcal{D}}$ is a semi-standard decomposition of $\widehat{u}$ with respect to $J^{(0)}(\widehat{\mathcal{D}}) = J^{(0)}(\mathcal{D})$; we have $J^{(n-k+1)}(\widehat{\mathcal{D}}) = J^{(n+1)}(\mathcal{D})$; and for each $0 \leq j \leq n-k$, we have $[y^{(j)}(\widehat{\mathcal{D}})] = [y^{(\sigma(j))}(\mathcal{D})]$, $[y^{(j+1)}(\widehat{\mathcal{D}})] = [y^{(\sigma(j)+1)}(\mathcal{D})]$, $J^{(j)}(\widehat{\mathcal{D}}) = J^{(\sigma(j))}(\mathcal{D})$ and $J^{(j+1)}(\widehat{\mathcal{D}}) = J^{(\sigma(j)+1)}(\mathcal{D})$.
\end{lem}
\begin{exmp}
\label{exmp:simplification}
For the case of Example \ref{exmp:semi-standard_decomposition}, the simplification $\widehat{\mathcal{D}}$ of the standard decomposition $\mathcal{D} = \omega_3\omega_2\omega_1\omega_0$ of $u$ is obtained by removing the third component $\omega_2$, therefore $\widehat{\mathcal{D}} = \omega_3\omega_1\omega_0$.
We have
\begin{displaymath}
\begin{split}
&y^{(0)}(\widehat{\mathcal{D}}) = y^{(0)}(\mathcal{D}) = (r_1,r_2,r_3) \,,\, y^{(1)}(\widehat{\mathcal{D}}) = y^{(1)}(\mathcal{D}) = (r_3,r_4,r_5) \enspace, \\
&y^{(2)}(\widehat{\mathcal{D}}) = y^{(2)}(\mathcal{D}) = (r_4,r_5,r_6) \,,\, y^{(3)}(\widehat{\mathcal{D}}) = (r_3,r_4,r_5) = \widehat{z} \enspace.
\end{split}
\end{displaymath}
Now since $\omega_3$ is the inverse of $\omega_1$, the semi-standard decomposition $\widehat{\mathcal{D}}$ of $\widehat{u}$ is not a standard decomposition of $\widehat{u}$.
\end{exmp}

Moreover, we have the following result:
\begin{lem}
\label{lem:another_decomposition_Y_shift_x}
Let $\mathcal{D} = \omega_n \cdots \omega_1\omega_0$ be a semi-standard decomposition of an element $u \in W$.
Let $r \in [y^{(0)}]$, and suppose that the support of each $\omega_i$ is apart from $r$.
Moreover, let $s \in J^{(0)}$, $s' \in J^{(n+1)}$ and suppose that $u \ast s = s'$.
Then we have $r \in [y^{(n+1)}]$, $u \ast r = r$ and $u \in Y_{z',z}$, where $z$ and $z'$ are elements of $S^{(\Lambda)}$ obtained from $y^{(0)}$ and $y^{(n+1)}$ by replacing $r$ with $s$ and with $s'$, respectively.
\end{lem}
\begin{proof}
We use induction on $n \geq 0$.
Put $\mathcal{D}' = \omega_{n-1} \cdots \omega_1\omega_0$, and let $u' \in Y_{y^{(n)},y^{(0)}}$ be the element expressed by the product $\mathcal{D}'$.
Let $s'' := u' \ast s \in J^{(n)}$.
By the induction hypothesis, we have $r \in [y^{(n)}]$, $u' \ast r = r$ and $u' \in Y_{z'',z}$, where $z''$ is the element of $S^{(\Lambda)}$ obtained from $y^{(n)}$ by replacing $r$ with $s''$.
Now, since the support $K^{(n)}$ of $\omega_n$ is apart from $r \in [y^{(n)}]$, it follows that $r \in [y^{(n+1)}]$ and $\omega_n \ast r = r$, therefore $u \ast r = \omega_n u' \ast r = r$.
On the other hand, we have $z' = \omega_n \ast z''$ by the construction of $z'$ and $z''$.
Moreover, by the definition of $\omega_n$, the set $K^{(n)}$ is apart from $([y^{(n)}] \cup J^{(n)}) \smallsetminus K^{(n)}$, therefore $K^{(n)}$ is also apart from the subset $([z''] \cup J^{(n)}) \smallsetminus K^{(n)}$ of $([y^{(n)}] \cup J^{(n)}) \smallsetminus K^{(n)}$.
Since $[y^{(n)}] \cap K^{(n)} \subseteq [z''] \cap K^{(n)}$, it follows that $\Phi^{\perp [z'']}[\omega_n] = \Phi_{K^{(n)}}^{\perp [z''] \cap K^{(n)}}[\omega_n] \subseteq \Phi_{K^{(n)}}^{\perp [y^{(n)}] \cap K^{(n)}}[\omega_n] = \Phi^{\perp [y^{(n)}]}[\omega_n] = \emptyset$ (note that $\omega_n \in Y_{y^{(n+1)},y^{(n)}}$), therefore we have $\omega_n \in Y_{z',z''}$.
Hence we have $u = \omega_n u' \in Y_{z',z}$, concluding the proof.
\end{proof}

\subsection{Reduction to a special case}
\label{sec:proof_reduction}

Here we give a reduction of our proof of Theorem \ref{thm:YfixesWperpIfin} to a special case where the possibility of the subset $I \subseteq S$ is restricted in a certain manner.

First, for $J \subseteq S$, let $\iota(J)$ denote temporarily the union of the irreducible components of $J$ that are not of finite type, and let $\overline{\iota}(J)$ denote temporarily the set of elements of $S$ that are not apart from $\iota(J)$ (hence $J \cap \overline{\iota}(J) = \iota(J)$).
For example, when $(W,S)$ is given by the Coxeter graph in Figure \ref{fig:example_notation_iota} (where we abbreviate each $r_i \in S$ to $i$) and $J = \{r_1,r_3,r_4,r_5,r_6\}$ (indicated in Figure \ref{fig:example_notation_iota} by the black vertices), we have $\iota(J) = \{r_1,r_5,r_6\}$ and $\overline{\iota}(J) = \{r_1,r_2,r_5,r_6,r_7\}$, therefore $J \cap \overline{\iota}(J) = \{r_1,r_5,r_6\} = \iota(J)$ as mentioned above.
Now we have the following:
\begin{figure}[hbt]
\centering
\begin{picture}(80,55)(0,-55)
\multiput(10,-20)(20,0){4}{\circle{8}}
\multiput(10,-40)(20,0){4}{\circle{8}}
\put(10,-20){\circle*{8}}
\put(50,-20){\circle*{8}}
\put(70,-20){\circle*{8}}
\put(10,-40){\circle*{8}}
\put(30,-40){\circle*{8}}
\multiput(14,-20)(20,0){3}{\line(1,0){12}}
\multiput(14,-40)(20,0){3}{\line(1,0){12}}
\put(10,-24){\line(0,-1){12}}
\put(13,-23){\line(1,-1){14}}
\put(67,-23){\line(-1,-1){14}}
\put(70,-24){\line(0,-1){12}}
\put(10,-12){\hbox to0pt{\hss$1$\hss}}
\put(30,-12){\hbox to0pt{\hss$2$\hss}}
\put(50,-12){\hbox to0pt{\hss$3$\hss}}
\put(70,-12){\hbox to0pt{\hss$4$\hss}}
\put(10,-54){\hbox to0pt{\hss$5$\hss}}
\put(30,-54){\hbox to0pt{\hss$6$\hss}}
\put(50,-54){\hbox to0pt{\hss$7$\hss}}
\put(70,-54){\hbox to0pt{\hss$8$\hss}}
\end{picture}
\caption{An example for the notations $\iota(J)$ and $\overline{\iota}(J)$; here $J = \{1,3,4,5,6\}$}
\label{fig:example_notation_iota}
\end{figure}
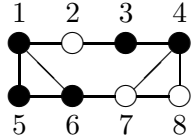
\begin{lem}
\label{lem:proof_reduction_root_perp}
Let $I$ be an arbitrary subset of $S$.
Then we have $w \in W_{S \smallsetminus \overline{\iota}(I)}$ for any $w \in Y_{y,x_I}$ with $y \in S^{(\Lambda)}$, and we have $\Phi^{\perp I} = \Phi_{S \smallsetminus \overline{\iota}(I)}^{\perp I \smallsetminus \overline{\iota}(I)}$.
\end{lem}
\begin{proof}
First, let $w \in Y_{y,x_I}$ with $y \in S^{(\Lambda)}$.
Then by Proposition \ref{prop:factorization_C}(\ref{item:prop_factorization_C_generator_Y}), there are a finite sequence $z_0 = x_I,z_1,\dots,z_{n-1},z_n = y$ of elements of $S^{(\Lambda)}$ and a finite sequence $s_0,s_1,\dots,s_{n-1}$ of elements of $S$ satisfying that $z_{i+1} \neq z_i$, $s_i \not\in [z_i]$, $[z_i]_{\sim s_i}$ is of finite type and $w_{z_i}^{s_i} \in Y_{z_{i+1},z_i}$ for each index $0 \leq i \leq n-1$, and we have $w = w_{z_{n-1}}^{s_{n-1}} \cdots w_{z_1}^{s_1} w_{z_0}^{s_0}$.
We show, by induction on $0 \leq i \leq n-1$, that $\iota([z_{i+1}]) = \iota(I)$, $\overline{\iota}([z_{i+1}]) = \overline{\iota}(I)$, and $w_{z_i}^{s_i} \in W_{S \smallsetminus \overline{\iota}(I)}$.
It follows from the induction hypothesis when $i > 0$, and is trivial when $i = 0$, that $\iota([z_i]) = \iota(I)$ and $\overline{\iota}([z_i]) = \overline{\iota}(I)$.
Since $s_i \not\in [z_i]$ and $[z_i]_{\sim s_i}$ is of finite type, it follows from the definition of $\overline{\iota}$ that $[z_i]_{\sim s_i} \subseteq S \smallsetminus \overline{\iota}([z_i])$, therefore we have $w_{z_i}^{s_i} \in W_{S \smallsetminus \overline{\iota}([z_i])} = W_{S \smallsetminus \overline{\iota}(I)}$, $\iota([z_{i+1}]) = \iota([z_i]) = \iota(I)$, and $\overline{\iota}([z_{i+1}]) = \overline{\iota}([z_i]) = \overline{\iota}(I)$, as desired.
This implies that $w = w_{z_{n-1}}^{s_{n-1}} \cdots w_{z_1}^{s_1} w_{z_0}^{s_0} \in W_{S \smallsetminus \overline{\iota}(I)}$, therefore the first part of the claim holds.

For the second part of the claim, the inclusion $\supseteq$ is obvious by the definitions of $\iota(I)$ and $\overline{\iota}(I)$.
For the other inclusion, it suffices to show that $\Phi^{\perp I} \subseteq \Phi_{S \smallsetminus \overline{\iota}(I)}$, or equivalently $\Pi^I \subseteq \Phi_{S \smallsetminus \overline{\iota}(I)}$.
Let $\gamma \in \Pi^I$.
By Proposition \ref{prop:factorization_C}(\ref{item:prop_factorization_C_generator_perp}), we have $\gamma = w \cdot \gamma(y,s)$ for some $y \in S^{(\Lambda)}$, $w \in Y_{x_I,y}$ and a root $\gamma(y,s)$ introduced in the statement of Proposition \ref{prop:charofBphi}.
Now by applying the result of the previous paragraph to $w^{-1} \in Y_{y,x_I}$, it follows that $\iota([y]) = \iota(I)$, $\overline{\iota}([y]) = \overline{\iota}(I)$, and $w \in W_{S \smallsetminus \overline{\iota}(I)}$.
Moreover, since $[y]_{\sim s}$ is of finite type (see Proposition \ref{prop:charofBphi}), a similar argument implies that $[y]_{\sim s} \subseteq S \smallsetminus \overline{\iota}([y]) = S \smallsetminus \overline{\iota}(I)$ and $w_y^s \in W_{S \smallsetminus \overline{\iota}(I)}$, therefore $\gamma(y,s) \in \Phi_{S \smallsetminus \overline{\iota}(I)}$.
Hence we have $\gamma = w \cdot \gamma(y,s) \in \Phi_{S \smallsetminus \overline{\iota}(I)}$, concluding the proof of Lemma \ref{lem:proof_reduction_root_perp}.
\end{proof}

For an arbitrary subset $I$ of $S$, suppose that $\gamma \in \Pi^I$, $s_\gamma \in W^{\perp I}{}_{\mathrm{fin}}$, and $w \in Y_I$.
Then by the second part of Lemma \ref{lem:proof_reduction_root_perp}, we have $\gamma \in \Pi^I = \Pi^{S \smallsetminus \overline{\iota}(I),I \smallsetminus \overline{\iota}(I)}$ and $s_\gamma$ also belongs to the finite part of $W_{S \smallsetminus \overline{\iota}(I)}^{\perp I \smallsetminus \overline{\iota}(I)}$.
Moreover, we have $w \in W_{S \smallsetminus \overline{\iota}(I)}$ by the first part of Lemma \ref{lem:proof_reduction_root_perp}, therefore $w$ also belongs to the group $Y_{I \smallsetminus \overline{\iota}(I)}$ constructed from the pair $S \smallsetminus \overline{\iota}(I)$, $I \smallsetminus \overline{\iota}(I)$ instead of the pair $S$, $I$.
Hence we have the following result: If the conclusion of Theorem \ref{thm:YfixesWperpIfin} holds for the pair $S \smallsetminus \overline{\iota}(I)$, $I \smallsetminus \overline{\iota}(I)$ instead of the pair $S$, $I$, then the conclusion of Theorem \ref{thm:YfixesWperpIfin} also holds for the pair $S$, $I$.
Note that $I \smallsetminus \overline{\iota}(I) = I \smallsetminus \iota(I)$ is the union of the irreducible components of $I$ of finite type.
As a consequence, we may assume without loss of generality that every irreducible component of $I$ is of finite type (note that the $A_{>1}$-freeness in the hypothesis of Theorem \ref{thm:YfixesWperpIfin} is preserved by considering $I \smallsetminus \iota(I)$ instead of $I$).

From now on, we assume that every irreducible component of $I$ is of finite type, as mentioned in the last paragraph.
For any $J \subseteq S$, we say that a subset $\Psi$ of the simple system $\Pi^J$ of $W^{\perp J}$ is an \emph{irreducible component} of $\Pi^J$ if $S(\Psi) = \{s_\beta \mid \beta \in \Psi\}$ is an irreducible component of the generating set $R^J$ of $W^{\perp J}$.
Now, as in the statement of Theorem \ref{thm:YfixesWperpIfin}, let $w \in Y_I$ and $\gamma \in \Pi^I$, and suppose that $s_{\gamma} \in W^{\perp I}{}_{\mathrm{fin}}$.
Let $\Psi$ denote the union of the irreducible components of $\Pi^I$ containing some $w^k \cdot \gamma$ with $k \in \mathbb{Z}$.
Then we have the following:
\begin{lem}
\label{lem:proof_reduction_Psi_finite}
In this setting, $\Psi$ is of finite type; in particular, $|\Psi| < \infty$.
Moreover, the two subsets $I \smallsetminus \mathrm{Supp}\,\Psi$ and $\mathrm{Supp}\,\Psi$ of $S$ are not adjacent.
\end{lem}
\begin{proof}
First, there exists a finite subset $K$ of $S$ for which $w \in W_K$ and $\gamma \in \Phi_K$.
Then, the number of mutually orthogonal roots of the form $w^k \cdot \gamma$ is at most $|K| < \infty$, since those roots are linearly independent and contained in the $|K|$-dimensional space $V_K$.
This implies that the number of irreducible components of $\Pi^I$ containing some $w^k \cdot \gamma$, which are of finite type by the property $s_\gamma \in W^{\perp I}{}_{\mathrm{fin}}$ and Proposition \ref{prop:factorization_C}(\ref{item:prop_factorization_C_decomp}), is finite.
Therefore, the union $\Psi$ of those irreducible components is also of finite type.
Hence the first part of the claim holds.

For the second part of the claim, assume contrary that some $s \in I \smallsetminus \mathrm{Supp}\,\Psi$ and $t \in \mathrm{Supp}\,\Psi$ are adjacent.
By the definition of $\mathrm{Supp}\,\Psi$, we have $t \in \mathrm{Supp}\,\beta \subseteq \mathrm{Supp}\,\Psi$ for some $\beta \in \Psi$.
Now we have $s \not\in \mathrm{Supp}\,\beta$.
Let $c > 0$ be the coefficient of $\alpha_t$ in $\beta$.
Then the property $s \not\in \mathrm{Supp}\,\beta$ implies that $\langle \alpha_s,\beta \rangle \leq c \langle \alpha_s,\alpha_t \rangle < 0$, contradicting the property $\beta \in \Phi^{\perp I}$.
Hence the claim holds, concluding the proof of Lemma \ref{lem:proof_reduction_Psi_finite}.
\end{proof}
We temporarily write $L = I \cap \mathrm{Supp}\,\Psi$, and put $\Psi' = \Psi \cup \Pi_L$.
Then we have $\mathrm{Supp}\,\Psi' = \mathrm{Supp}\,\Psi$, therefore by Lemma \ref{lem:proof_reduction_Psi_finite}, $I \smallsetminus \mathrm{Supp}\,\Psi'$ and $\mathrm{Supp}\,\Psi'$ are not adjacent.
On the other hand, we have $|\Psi| < \infty$ by Lemma \ref{lem:proof_reduction_Psi_finite}, therefore $\mathrm{Supp}\,\Psi' = \mathrm{Supp}\,\Psi$ is a finite set.
By these properties and the above-mentioned assumption that every irreducible component of $I$ is of finite type, it follows that $\Pi_L$ is of finite type as well as $\Psi$.
Note that $\Psi \subseteq \Pi^I \subseteq \Phi^{\perp L}$.
Hence the two root bases $\Psi$ and $\Pi_L$ are orthogonal, therefore their union $\Psi'$ is also a root basis by Theorem \ref{thm:conditionforrootbasis}, and we have $|W(\Psi')|<\infty$.
By Proposition \ref{prop:fintyperootbasis}, $\Psi'$ is a basis of a subspace $U := \mathrm{span}\,\Psi'$ of $V_{\mathrm{Supp}\,\Psi'}$.
By applying Proposition \ref{prop:finitesubsystem} to $W_{\mathrm{Supp}\,\Psi'}$ instead of $W$, it follows that there exist $u \in W_{\mathrm{Supp}\,\Psi'}$ and $J \subseteq \mathrm{Supp}\,\Psi'$ satisfying that $W_J$ is finite, $u \cdot (U \cap \Phi^+) = \Phi_J^+$ and $u \cdot (U \cap \Pi) \subseteq \Pi_J$.
Now we have the following:
\begin{lem}
\label{lem:proof_reduction_conjugate_by_Y}
In this setting, if we choose such an element $u$ of minimal length, then there exists an element $y \in S^{(\Lambda)}$ satisfying that $u \in Y_{y,x_I}$, the sets $[y] \smallsetminus J$ and $J$ are not adjacent, and $(u \cdot \Psi) \cup \Pi_{[y] \cap J}$ is a basis of $V_J$.
\end{lem}
\begin{proof}
Since $\Psi'$ is a basis of $U$, the property $u \cdot (U \cap \Phi^+) = \Phi_J^+$ implies that $u \cdot \Psi'$ is a basis of $V_J$.
Now we have $u \cdot \Pi_L \subseteq \Pi_J$ since $\Pi_L \subseteq U \cap \Pi$, while $u$ fixes $\Pi_{I \smallsetminus L}$ pointwise since the sets $I \smallsetminus \mathrm{Supp}\,\Psi' = I \smallsetminus L$ and $\mathrm{Supp}\,\Psi'$ are not adjacent.
By these properties, there exists an element $y \in S^{(\Lambda)}$ satisfying that $y = u \ast x_I$, $[y] \cap \mathrm{Supp}\,\Psi' \subseteq J$ and $[y] \smallsetminus \mathrm{Supp}\,\Psi' = I \smallsetminus \mathrm{Supp}\,\Psi'$.
Since $J \subseteq \mathrm{Supp}\,\Psi'$, it follows that $[y] \smallsetminus J$ and $J$ are not adjacent.
On the other hand, since $u \cdot \Pi_{I \smallsetminus L} = \Pi_{I \smallsetminus L}$, $u \cdot (U \cap \Phi^+) = \Phi_J^+$ and $\Pi_{I \smallsetminus L} \cap U = \emptyset$, it follows that $\Pi_{I \smallsetminus L} \cap \Phi_J^+ = \emptyset$, therefore we have $u \cdot \Pi_L = \Pi_{[y] \cap J}$.
Hence $u \cdot \Psi' = (u \cdot \Psi) \cup \Pi_{[y] \cap J}$ is a basis of $V_J$.

Finally, we show that such an element $u$ of minimal length satisfies that $u \cdot \Pi^I \subseteq \Phi^+$, hence $u \cdot (\Phi^{\perp I})^+ \subseteq \Phi^+$ and $u \in Y_{y,x_I}$.
We have $u \cdot \Psi \subseteq u \cdot (U \cap \Phi^+) = \Phi_J^+$.
Secondly, for any $\beta \in \Pi^I \smallsetminus \Psi$, assume contrary that $u \cdot \beta \in \Phi^-$.
Then we have $\beta \in \Phi_{\mathrm{Supp}\,\Psi'}$ since $u \in W_{\mathrm{Supp}\,\Psi'}$, therefore $s_{\beta} \in W_{\mathrm{Supp}\,\Psi'}$.
On the other hand, since $\Psi$ is the union of some irreducible components of $\Pi^I$, it follows that $\beta$ is orthogonal to $\Psi$, hence orthogonal to $\Psi'$.
By these properties, the element $u s_{\beta}$ also satisfies the above characteristics of the element $u$.
However, now the property $u \cdot \beta \in \Phi^-$ implies that $\ell(u s_{\beta}) < \ell(u)$ (see Theorem \ref{thm:reflectionsubgroup_Deodhar}), contradicting the choice of $u$.
Hence we have $u \cdot \beta \in \Phi^+$ for every $\beta \in \Pi^I \smallsetminus \Psi$, therefore $u \cdot \Pi^I \subseteq \Phi^+$, concluding the proof of Lemma \ref{lem:proof_reduction_conjugate_by_Y}.
\end{proof}
For an element $u \in Y_{y,x_I}$ as in Lemma \ref{lem:proof_reduction_conjugate_by_Y}, Proposition \ref{prop:factorization_C}(\ref{item:prop_factorization_C_decomp}) implies that $u \cdot \gamma \in \Pi^{[y]}$ and $s_{u \cdot \gamma} = u s_\gamma u^{-1} \in W^{\perp [y]}{}_{\mathrm{fin}}$.
Now $w$ fixes the root $\gamma$ if and only if the element $uwu^{-1} \in Y_{y,y}$ fixes the root $u \cdot \gamma$.
Moreover, the conjugation by $u$ defines an isomorphism of Coxeter systems $(W_I,I) \to (W_{[y]},[y])$.
Hence, by considering $[y] \subseteq S$, $uwu^{-1} \in Y_{[y]}$, $u \cdot \gamma \in \Pi^{[y]}$ and $u \cdot \Psi \subseteq \Pi^{[y]}$ instead of $I$, $w$, $\gamma$ and $\Psi$ if necessary, we may assume without loss of generality the following conditions:
\begin{description}
\item[(A1)] Every irreducible component of $I$ is of finite type.
\item[(A2)] There exists a subset $J \subseteq S$ of finite type satisfying that $I \smallsetminus J$ and $J$ are not adjacent and $\Psi \cup \Pi_{I \cap J}$ is a basis of $V_J$.
\end{description}
Moreover, if an irreducible component $J'$ of $J$ is contained in $I$, then a smaller subset $J \smallsetminus J'$ instead of $J$ also satisfies the assumption (A2); indeed, now $\Pi_{J'} \subseteq \Pi_{I \cap J}$ spans $V_{J'}$, and since $\Psi \cup \Pi_{I \cap J}$ is a basis of $V_J$ and the support of any root is irreducible (see Lemma \ref{lem:support_is_irreducible}), it follows that the support of any element of $\Psi \cup \Pi_{I \cap (J \smallsetminus J')}$ does not intersect with $J'$.
Hence, by choosing a subset $J \subseteq S$ in (A2) as small as possible, we may also assume without loss of generality the following condition:
\begin{description}
\item[(A3)] Any irreducible component of $J$ is not contained in $I$.
\end{description}
We also notice the following properties:
\begin{lem}
\label{lem:Psi_is_full}
In this setting, we have $\Psi = \Pi^{J,I \cap J}$, hence $\Pi^{J,I \cap J} \cup \Pi_{I \cap J}$ is a basis of $V_J$.
\end{lem}
\begin{proof}
The inclusion $\Psi \subseteq \Pi^{J,I \cap J}$ follows from the definition of $\Psi$ and the condition (A2).
Now assume contrary that $\beta \in \Pi^{J,I \cap J} \smallsetminus \Psi$.
Then we have $\beta \in \Pi^I$ by (A2).
Since $\Psi$ is the union of some irreducible components of $\Pi^I$, it follows that $\beta$ is orthogonal to $\Psi$ as well as to $\Pi_{I \cap J}$.
This implies that $\beta$ belongs to the radical of $V_J$, which should be trivial by Proposition \ref{prop:fintyperootbasis}.
This is a contradiction.
Hence the claim holds.
\end{proof}
\begin{lem}
\label{lem:property_w}
In this setting, the element $w \in Y_I$ satisfies that $w \cdot \Phi_J = \Phi_J$, and the subgroup $\langle w \rangle$ generated by $w$ acts transitively on the set of the irreducible components of $\Pi^{J,I \cap J}$.
\end{lem}
\begin{proof}
The second part of the claim follows immediately from the definition of $\Psi$ and Lemma \ref{lem:Psi_is_full}.
It also implies that $w \cdot \Pi^{J,I \cap J} = \Pi^{J,I \cap J}$, while $w \cdot \Pi_{I \cap J} = \Pi_{I \cap J}$ since $w \in Y_I$.
Moreover, $\Pi^{J,I \cap J} \cup \Pi_{I \cap J}$ is a basis of $V_J$ by Lemma \ref{lem:Psi_is_full}.
This implies that $w \cdot V_J = V_J$, therefore we have $w \cdot \Phi_J = \Phi_J$.
Hence the claim holds.
\end{proof}

\subsection{A key lemma}
\label{sec:proof_key_lemma}

Let $I^{\perp}$ denote the set of all elements of $S$ that are apart from $I$.
Then there are two possibilities: $\Pi^{J,I \cap J} \not\subseteq \Phi_{I^{\perp}}$, or $\Pi^{J,I \cap J} \subseteq \Phi_{I^{\perp}}$.
Here we present a key lemma regarding the former possibility (recall the three conditions (A1)--(A3) specified above):
\begin{lem}
\label{lem:finitepart_first_case_irreducible}
If $\Pi^{J,I \cap J} \not\subseteq \Phi_{I^{\perp}}$, then we have $I \cap J \neq \emptyset$ and $J$ is irreducible.
\end{lem}
\begin{proof}
First, take an element $\beta \in \Pi^{J,I \cap J} \smallsetminus \Phi_{I^{\perp}}$.
Then we have $\beta \not\in \Phi_I$ since $\Pi^{J,I \cap J} \subseteq \Phi^{\perp I}$.
Moreover, since the support $\mathrm{Supp}\,\beta$ of $\beta$ is irreducible (see Lemma \ref{lem:support_is_irreducible}), there exists an element $s \in \mathrm{Supp}\,\beta \smallsetminus I$ which is adjacent to an element of $I$, say $s' \in I$.
Now the property $\beta \in \Phi^{\perp I}$ implies that $s' \in \mathrm{Supp}\,\beta$, since otherwise we have $\langle \beta,\alpha_{s'} \rangle \leq c \langle \alpha_s,\alpha_{s'} \rangle < 0$ where $c > 0$ is the coefficient of $\alpha_s$ in $\beta$.
Hence we have $s' \in \mathrm{Supp}\,\Pi^{J,I \cap J} \subseteq J$.

Let $K$ denote the irreducible component of $J$ containing $s'$.
Put $\Psi' = \Pi^{J,I \cap J} \cap \Phi_K$.
Then, since $\Pi^{J,I \cap J} \cup \Pi_{I \cap J}$ is a basis of $V_J$ by Lemma \ref{lem:Psi_is_full} and the support of any root is irreducible (see Lemma \ref{lem:support_is_irreducible}), it follows that $\beta \in \Psi'$, $\Psi'$ is orthogonal to $\Pi^{J,I \cap J} \smallsetminus \Psi'$ and $\Psi' \cup \Pi_{I \cap K}$ is a basis of $V_K$.
Now $\Psi'$ is the union of some irreducible components of $\Pi^{J,I \cap J}$.
We show that $J$ is irreducible if we have $w \cdot \Phi_K = \Phi_K$.
In this case, we have $w \cdot \Psi' = \Psi'$, therefore $\Pi^{J,I \cap J} = \Psi' \subseteq \Phi_K$ by the second part of Lemma \ref{lem:property_w}.
Now by the condition (A3), $J$ has no irreducible components other than $K$ (indeed, if such an irreducible component $J'$ of $J$ exists, then the property $\Pi^{J,I \cap J} \subseteq \Phi_K$ implies that the space $V_{J'}$ should be spanned by a subset of $\Pi_{I \cap J}$, therefore $J' \subseteq I$).
Hence $J = K$ is irreducible.

Thus it suffices to show that $w \cdot \Phi_K = \Phi_K$.
For the purpose, it also suffices to show that $w \cdot \Phi_K \subseteq \Phi_K$ (since $K$ is of finite type as well as $J$), or equivalently $w \cdot \Pi_K \subseteq \Phi_K$.
Moreover, by the three properties that $K$ is irreducible, $K \cap I \neq \emptyset$ and $w \cdot \Pi_{K \cap I} = \Pi_{K \cap I}$, it suffices to show that $w \cdot \alpha_{t'} \in \Phi_K$ provided $t' \in K$ is adjacent to some $t \in K$ with $w \cdot \alpha_t \in \Phi_K$.
Now note that $w \cdot \Phi_J = \Phi_J$ by Lemma \ref{lem:property_w}.
Assume contrary that $w \cdot \alpha_{t'} \not\in \Phi_K$.
Then we have $w \cdot \alpha_{t'} \in \Phi_J \smallsetminus \Phi_K = \Phi_{J \smallsetminus K}$ since $K$ is an irreducible component of $J$, therefore $w \cdot \alpha_{t'}$ is orthogonal to $w \cdot \alpha_t \in \Phi_K$.
This contradicts the property that $t'$ is adjacent to $t$, since $w$ leaves the bilinear form $\langle\,,\,\rangle$ invariant.
Hence we have $w \cdot \alpha_{t'} \in \Phi_K$, as desired.
\end{proof}

\section{Proof of Theorem \ref{thm:YfixesWperpIfin}: On the special case}
\label{sec:proof_special}

In this section, we introduce the assumption in Theorem \ref{thm:YfixesWperpIfin} that $I$ is $A_{>1}$-free, and continue the argument in Section \ref{sec:proof_general}.
Recall the properties (A1), (A2) and (A3) of $I$, $J$ and $\Psi = \Pi^{J,I \cap J}$ (see Lemma \ref{lem:Psi_is_full}) given in Section \ref{sec:proof_reduction}.
Our aim here is to prove that $w$ fixes $\Pi^{J,I \cap J}$ pointwise, which implies our goal $w \cdot \gamma = \gamma$ since $\gamma \in \Psi = \Pi^{J,I \cap J}$ by the definition of $\Psi$.
We divide the following argument into two cases: $\Pi^{J,I \cap J} \not\subseteq \Phi_{I^{\perp}}$, or $\Pi^{J,I \cap J} \subseteq \Phi_{I^{\perp}}$ (see Section \ref{sec:proof_key_lemma} for the definition of $I^{\perp}$).

\subsection{The first case $\Pi^{J,I \cap J} \not\subseteq \Phi_{I^{\perp}}$}
\label{sec:proof_special_first_case}

Here we consider the case that $\Pi^{J,I \cap J} \not\subseteq \Phi_{I^{\perp}}$.
In this case, the subset $J \subseteq S$ of finite type is irreducible by Lemma \ref{lem:finitepart_first_case_irreducible}, therefore we can apply the classification of finite irreducible Coxeter groups.
Let $J = \{r_1,r_2,\dots,r_N\}$, where $N = |J|$, be the standard labelling of $J$ (see Section \ref{sec:longestelement}).
We write $\alpha_i = \alpha_{r_i}$ for simplicity.

We introduce some temporal terminology.
We say that an element $y \in S^{(\Lambda)}$ satisfies \emph{Property P} if $[y] \smallsetminus J = I \smallsetminus J$ (hence $[y] \smallsetminus J$ is apart from $J$ by the condition (A2)) and $\Pi^{J,[y] \cap J} \cup \Pi_{[y] \cap J}$ is a basis of $V_J$.
For example, $x_I$ itself satisfies Property P.
For any $y \in S^{(\Lambda)}$ satisfying Property P and any element $s \in J \smallsetminus [y]$ with $\varphi(y,s) \neq y$, we say that the isomorphism $t \mapsto w_y^s \ast t$ from $[y] \cap J$ to $[\varphi(y,s)] \cap J$ is a \emph{local transformation} (note that now $[y]_{\sim s} \subseteq J$ and $w_y^s \in W_J$ by the above-mentioned property that $[y] \smallsetminus J$ is apart from $J$).
By abusing the terminology, in such a case we also call the correspondence $y \mapsto \varphi(y,s)$ a local transformation.
Note that, in this case, $\varphi(y,s)$ also satisfies Property P, we have $w_y^s \in Y_{\varphi(y,s),y}$ and $w_y^s \ast t = t$ for any $t \in [y] \smallsetminus J$, and the action of $w_y^s$ induces an isomorphism from $\Pi^{J,[y] \cap J}$ to $\Pi^{J,[\varphi(y,s)] \cap J}$.

Since $w \cdot \Pi^{J,I \cap J} = \Pi^{J,I \cap J}$, the claim is trivial if $|\Pi^{J,I \cap J}| = 1$.
From now, we consider the case that $|\Pi^{J,I \cap J}| \geq 2$, therefore we have $N = |J| \geq |I \cap J| + 2 \geq 3$ (note that $I \cap J \neq \emptyset$ by Lemma \ref{lem:finitepart_first_case_irreducible}).
In particular, $J$ is not of type $I_2(m)$.
On the other hand, we have the following results:
\begin{lem}
\label{lem:J_not_A_N}
In this setting, $J$ is not of type $A_N$.
\end{lem}
\begin{proof}
We show that $\Pi^{J,I \cap J} \cup \Pi_{I \cap J}$ cannot span $V_J$ if $J$ is of type $A_N$, which deduces a contradiction and hence concludes the proof.
By the $A_{>1}$-freeness of $I$, each irreducible component of $I \cap J$ (which is also an irreducible component of $I$) is of type $A_1$.
Now by applying successive local transformations, we may assume without loss of generality that $r_1 \in I$ (indeed, if the minimal index $i$ with $r_i \in I$ satisfies $i \geq 2$, then we have $\varphi(x_I,r_{i-1}) \ast r_i = r_{i-1}$).
In this case, we have $r_2 \not\in I$, while we have $\Phi_J^{\perp I} \subseteq \Phi_{J \smallsetminus \{r_1,r_2\}}$ by the fact that any positive root in the root system $\Phi_J$ of type $A_N$ is of the form $\alpha_i + \alpha_{i+1} + \cdots + \alpha_{i'}$ with $1 \leq i \leq i' \leq N$.
This implies that the subset $\Pi^{J,I \cap J} \cup \Pi_{I \cap J}$ of $\Phi_J^{\perp I} \cup \Pi_{I \cap J}$ cannot span $V_J$, as desired.
\end{proof}
To prove the next lemma (and some other results below), we give a list of all positive roots of the Coxeter group of type $E_8$.
The list is divided into six parts (Tables \ref{tab:positive_roots_E_8_1}--\ref{tab:positive_roots_E_8_6}).
In the lists, we use the standard labelling $r_1,\dots,r_8$ of generators.
The coefficients of each root are placed at the same relative positions as the corresponding vertices of the Coxeter graph of type $E_8$ in Figure \ref{fig:finite_irreducible_Coxeter_groups}; for example, the last root $\gamma_{120}$ in Table \ref{tab:positive_roots_E_8_6} is $2\alpha_1 + 3\alpha_2 + 4\alpha_3 + 6\alpha_4 + 5\alpha_5 + 4\alpha_6 + 3\alpha_7 + 2\alpha_8$ (which is the highest root of type $E_8$).
For the columns for actions of generators (4th to 11th columns), a blank cell means that the generator $r_j$ fixes the root $\gamma_i$ (or equivalently, $\langle \alpha_j,\gamma_i \rangle = 0$); while a cell filled by \lq\lq ---'' means that $\gamma_i = \alpha_j$.
Moreover, the positive roots of the parabolic subgroup of type $E_6$ (respectively, $E_7$) generated by $\{r_1,\dots,r_6\}$ (respectively, $\{r_1,\dots,r_7\}$) correspond to the rows indicated by \lq\lq $E_6$'' (respectively, \lq\lq $E_7$'').
By the data for actions of generators, it can be verified that the list indeed exhausts all the positive roots.
\begin{table}[p]
\centering
\caption{List of positive roots for Coxeter group of type $E_8$ (part $1$)}
\label{tab:positive_roots_E_8_1}
\begin{tabular}{|c||c|c|c|c|c|c|c|c|c|c|cc} \cline{1-11}
height & $i$ & root $\gamma_i$ & \multicolumn{8}{|c|}{index $k$ with $r_j \cdot \gamma_i = \gamma_k$} \\ \cline{4-11}
& & & $r_1$ & $r_2$ & $r_3$ & $r_4$ & $r_5$ & $r_6$ & $r_7$ & $r_8$ \\ \cline{1-11}\cline{1-11}
$1$ & $1$ & \Eroot{1}{0}{0}{0}{0}{0}{0}{0} & --- & & $9$ & & & & & & $E_6$ & $E_7$ \\ \cline{2-11}
& $2$ & \Eroot{0}{1}{0}{0}{0}{0}{0}{0} & & --- & & $10$ & & & & & $E_6$ & $E_7$ \\ \cline{2-11}
& $3$ & \Eroot{0}{0}{1}{0}{0}{0}{0}{0} & $9$ & & --- & $11$ & & & & & $E_6$ & $E_7$ \\ \cline{2-11}
& $4$ & \Eroot{0}{0}{0}{1}{0}{0}{0}{0} & & $10$ & $11$ & --- & $12$ & & & & $E_6$ & $E_7$ \\ \cline{2-11}
& $5$ & \Eroot{0}{0}{0}{0}{1}{0}{0}{0} & & & & $12$ & --- & $13$ & & & $E_6$ & $E_7$ \\ \cline{2-11}
& $6$ & \Eroot{0}{0}{0}{0}{0}{1}{0}{0} & & & & & $13$ & --- & $14$ & & $E_6$ & $E_7$ \\ \cline{2-11}
& $7$ & \Eroot{0}{0}{0}{0}{0}{0}{1}{0} & & & & & & $14$ & --- & $15$ & & $E_7$ \\ \cline{2-11}
& $8$ & \Eroot{0}{0}{0}{0}{0}{0}{0}{1} & & & & & & & $15$ & --- \\ \cline{1-11}
$2$ & $9$ & \Eroot{1}{0}{1}{0}{0}{0}{0}{0} & $3$ & & $1$ & $16$ & & & & & $E_6$ & $E_7$ \\ \cline{2-11}
& $10$ & \Eroot{0}{1}{0}{1}{0}{0}{0}{0} & & $4$ & $17$ & $2$ & $18$ & & & & $E_6$ & $E_7$ \\ \cline{2-11}
& $11$ & \Eroot{0}{0}{1}{1}{0}{0}{0}{0} & $16$ & $17$ & $4$ & $3$ & $19$ & & & & $E_6$ & $E_7$ \\ \cline{2-11}
& $12$ & \Eroot{0}{0}{0}{1}{1}{0}{0}{0} & & $18$ & $19$ & $5$ & $4$ & $20$ & & & $E_6$ & $E_7$ \\ \cline{2-11}
& $13$ & \Eroot{0}{0}{0}{0}{1}{1}{0}{0} & & & & $20$ & $6$ & $5$ & $21$ & & $E_6$ & $E_7$ \\ \cline{2-11}
& $14$ & \Eroot{0}{0}{0}{0}{0}{1}{1}{0} & & & & & $21$ & $7$ & $6$ & $22$ & & $E_7$ \\ \cline{2-11}
& $15$ & \Eroot{0}{0}{0}{0}{0}{0}{1}{1} & & & & & & $22$ & $8$ & $7$ \\ \cline{1-11}
$3$ & $16$ & \Eroot{1}{0}{1}{1}{0}{0}{0}{0} & $11$ & $23$ & & $9$ & $24$ & & & & $E_6$ & $E_7$ \\ \cline{2-11}
& $17$ & \Eroot{0}{1}{1}{1}{0}{0}{0}{0} & $23$ & $11$ & $10$ & & $25$ & & & & $E_6$ & $E_7$ \\ \cline{2-11}
& $18$ & \Eroot{0}{1}{0}{1}{1}{0}{0}{0} & & $12$ & $25$ & & $10$ & $26$ & & & $E_6$ & $E_7$ \\ \cline{2-11}
& $19$ & \Eroot{0}{0}{1}{1}{1}{0}{0}{0} & $24$ & $25$ & $12$ & & $11$ & $27$ & & & $E_6$ & $E_7$ \\ \cline{2-11}
& $20$ & \Eroot{0}{0}{0}{1}{1}{1}{0}{0} & & $26$ & $27$ & $13$ & & $12$ & $28$ & & $E_6$ & $E_7$ \\ \cline{2-11}
& $21$ & \Eroot{0}{0}{0}{0}{1}{1}{1}{0} & & & & $28$ & $14$ & & $13$ & $29$ & & $E_7$ \\ \cline{2-11}
& $22$ & \Eroot{0}{0}{0}{0}{0}{1}{1}{1} & & & & & $29$ & $15$ & & $14$ \\ \cline{1-11}
\end{tabular}
\end{table}
\begin{table}[p]
\centering
\caption{List of positive roots for Coxeter group of type $E_8$ (part $2$)}
\label{tab:positive_roots_E_8_2}
\begin{tabular}{|c||c|c|c|c|c|c|c|c|c|c|cc} \cline{1-11}
height & $i$ & root $\gamma_i$ & \multicolumn{8}{|c|}{index $k$ with $r_j \cdot \gamma_i = \gamma_k$} \\ \cline{4-11}
& & & $r_1$ & $r_2$ & $r_3$ & $r_4$ & $r_5$ & $r_6$ & $r_7$ & $r_8$ \\ \cline{1-11}
$4$ & $23$ & \Eroot{1}{1}{1}{1}{0}{0}{0}{0} & $17$ & $16$ & & & $30$ & & & & $E_6$ & $E_7$ \\ \cline{2-11}
& $24$ & \Eroot{1}{0}{1}{1}{1}{0}{0}{0} & $19$ & $30$ & & & $16$ & $31$ & & & $E_6$ & $E_7$ \\ \cline{2-11}
& $25$ & \Eroot{0}{1}{1}{1}{1}{0}{0}{0} & $30$ & $19$ & $18$ & $32$ & $17$ & $33$ & & & $E_6$ & $E_7$ \\ \cline{2-11}
& $26$ & \Eroot{0}{1}{0}{1}{1}{1}{0}{0} & & $20$ & $33$ & & & $18$ & $34$ & & $E_6$ & $E_7$ \\ \cline{2-11}
& $27$ & \Eroot{0}{0}{1}{1}{1}{1}{0}{0} & $31$ & $33$ & $20$ & & & $19$ & $35$ & & $E_6$ & $E_7$ \\ \cline{2-11}
& $28$ & \Eroot{0}{0}{0}{1}{1}{1}{1}{0} & & $34$ & $35$ & $21$ & & & $20$ & $36$ & & $E_7$ \\ \cline{2-11}
& $29$ & \Eroot{0}{0}{0}{0}{1}{1}{1}{1} & & & & $36$ & $22$ & & & $21$ \\ \cline{1-11}
$5$ & $30$ & \Eroot{1}{1}{1}{1}{1}{0}{0}{0} & $25$ & $24$ & & $37$ & $23$ & $38$ & & & $E_6$ & $E_7$ \\ \cline{2-11}
& $31$ & \Eroot{1}{0}{1}{1}{1}{1}{0}{0} & $27$ & $38$ & & & & $24$ & $39$ & & $E_6$ & $E_7$ \\ \cline{2-11}
& $32$ & \Eroot{0}{1}{1}{2}{1}{0}{0}{0} & $37$ & & & $25$ & & $40$ & & & $E_6$ & $E_7$ \\ \cline{2-11}
& $33$ & \Eroot{0}{1}{1}{1}{1}{1}{0}{0} & $38$ & $27$ & $26$ & $40$ & & $25$ & $41$ & & $E_6$ & $E_7$ \\ \cline{2-11}
& $34$ & \Eroot{0}{1}{0}{1}{1}{1}{1}{0} & & $28$ & $41$ & & & & $26$ & $42$ & & $E_7$ \\ \cline{2-11}
& $35$ & \Eroot{0}{0}{1}{1}{1}{1}{1}{0} & $39$ & $41$ & $28$ & & & & $27$ & $43$ & & $E_7$ \\ \cline{2-11}
& $36$ & \Eroot{0}{0}{0}{1}{1}{1}{1}{1} & & $42$ & $43$ & $29$ & & & & $28$ \\ \cline{1-11}
$6$ & $37$ & \Eroot{1}{1}{1}{2}{1}{0}{0}{0} & $32$ & & $44$ & $30$ & & $45$ & & & $E_6$ & $E_7$ \\ \cline{2-11}
& $38$ & \Eroot{1}{1}{1}{1}{1}{1}{0}{0} & $33$ & $31$ & & $45$ & & $30$ & $46$ & & $E_6$ & $E_7$ \\ \cline{2-11}
& $39$ & \Eroot{1}{0}{1}{1}{1}{1}{1}{0} & $35$ & $46$ & & & & & $31$ & $47$ & & $E_7$ \\ \cline{2-11}
& $40$ & \Eroot{0}{1}{1}{2}{1}{1}{0}{0} & $45$ & & & $33$ & $48$ & $32$ & $49$ & & $E_6$ & $E_7$ \\ \cline{2-11}
& $41$ & \Eroot{0}{1}{1}{1}{1}{1}{1}{0} & $46$ & $35$ & $34$ & $49$ & & & $33$ & $50$ & & $E_7$ \\ \cline{2-11}
& $42$ & \Eroot{0}{1}{0}{1}{1}{1}{1}{1} & & $36$ & $50$ & & & & & $34$ \\ \cline{2-11}
& $43$ & \Eroot{0}{0}{1}{1}{1}{1}{1}{1} & $47$ & $50$ & $36$ & & & & & $35$ \\ \cline{1-11}
\end{tabular}
\end{table}
\begin{table}[p]
\centering
\caption{List of positive roots for Coxeter group of type $E_8$ (part $3$)}
\label{tab:positive_roots_E_8_3}
\begin{tabular}{|c||c|c|c|c|c|c|c|c|c|c|cc} \cline{1-11}
height & $i$ & root $\gamma_i$ & \multicolumn{8}{|c|}{index $k$ with $r_j \cdot \gamma_i = \gamma_k$} \\ \cline{4-11}
& & & $r_1$ & $r_2$ & $r_3$ & $r_4$ & $r_5$ & $r_6$ & $r_7$ & $r_8$ \\ \cline{1-11}
$7$ & $44$ & \Eroot{1}{1}{2}{2}{1}{0}{0}{0} & & & $37$ & & & $51$ & & & $E_6$ & $E_7$ \\ \cline{2-11}
& $45$ & \Eroot{1}{1}{1}{2}{1}{1}{0}{0} & $40$ & & $51$ & $38$ & $52$ & $37$ & $53$ & & $E_6$ & $E_7$ \\ \cline{2-11}
& $46$ & \Eroot{1}{1}{1}{1}{1}{1}{1}{0} & $41$ & $39$ & & $53$ & & & $38$ & $54$ & & $E_7$ \\ \cline{2-11}
& $47$ & \Eroot{1}{0}{1}{1}{1}{1}{1}{1} & $43$ & $54$ & & & & & & $39$ \\ \cline{2-11}
& $48$ & \Eroot{0}{1}{1}{2}{2}{1}{0}{0} & $52$ & & & & $40$ & & $55$ & & $E_6$ & $E_7$ \\ \cline{2-11}
& $49$ & \Eroot{0}{1}{1}{2}{1}{1}{1}{0} & $53$ & & & $41$ & $55$ & & $40$ & $56$ & & $E_7$ \\ \cline{2-11}
& $50$ & \Eroot{0}{1}{1}{1}{1}{1}{1}{1} & $54$ & $43$ & $42$ & $56$ & & & & $41$ \\ \cline{1-11}
$8$ & $51$ & \Eroot{1}{1}{2}{2}{1}{1}{0}{0} & & & $45$ & & $57$ & $44$ & $58$ & & $E_6$ & $E_7$ \\ \cline{2-11}
& $52$ & \Eroot{1}{1}{1}{2}{2}{1}{0}{0} & $48$ & & $57$ & & $45$ & & $59$ & & $E_6$ & $E_7$ \\ \cline{2-11}
& $53$ & \Eroot{1}{1}{1}{2}{1}{1}{1}{0} & $49$ & & $58$ & $46$ & $59$ & & $45$ & $60$ & & $E_7$ \\ \cline{2-11}
& $54$ & \Eroot{1}{1}{1}{1}{1}{1}{1}{1} & $50$ & $47$ & & $60$ & & & & $46$ \\ \cline{2-11}
& $55$ & \Eroot{0}{1}{1}{2}{2}{1}{1}{0} & $59$ & & & & $49$ & $61$ & $48$ & $62$ & & $E_7$ \\ \cline{2-11}
& $56$ & \Eroot{0}{1}{1}{2}{1}{1}{1}{1} & $60$ & & & $50$ & $62$ & & & $49$ \\ \cline{1-11}
$9$ & $57$ & \Eroot{1}{1}{2}{2}{2}{1}{0}{0} & & & $52$ & $63$ & $51$ & & $64$ & & $E_6$ & $E_7$ \\ \cline{2-11}
& $58$ & \Eroot{1}{1}{2}{2}{1}{1}{1}{0} & & & $53$ & & $64$ & & $51$ & $65$ & & $E_7$ \\ \cline{2-11}
& $59$ & \Eroot{1}{1}{1}{2}{2}{1}{1}{0} & $55$ & & $64$ & & $53$ & $66$ & $52$ & $67$ & & $E_7$ \\ \cline{2-11}
& $60$ & \Eroot{1}{1}{1}{2}{1}{1}{1}{1} & $56$ & & $65$ & $54$ & $67$ & & & $53$ \\ \cline{2-11}
& $61$ & \Eroot{0}{1}{1}{2}{2}{2}{1}{0} & $66$ & & & & & $55$ & & $68$ & & $E_7$ \\ \cline{2-11}
& $62$ & \Eroot{0}{1}{1}{2}{2}{1}{1}{1} & $67$ & & & & $56$ & $68$ & & $55$ \\ \cline{1-11}
\end{tabular}
\end{table}
\begin{table}[p]
\centering
\caption{List of positive roots for Coxeter group of type $E_8$ (part $4$)}
\label{tab:positive_roots_E_8_4}
\begin{tabular}{|c||c|c|c|c|c|c|c|c|c|c|cc} \cline{1-11}
height & $i$ & root $\gamma_i$ & \multicolumn{8}{|c|}{index $k$ with $r_j \cdot \gamma_i = \gamma_k$} \\ \cline{4-11}
& & & $r_1$ & $r_2$ & $r_3$ & $r_4$ & $r_5$ & $r_6$ & $r_7$ & $r_8$ \\ \cline{1-11}
$10$ & $63$ & \Eroot{1}{1}{2}{3}{2}{1}{0}{0} & & $69$ & & $57$ & & & $70$ & & $E_6$ & $E_7$ \\ \cline{2-11}
& $64$ & \Eroot{1}{1}{2}{2}{2}{1}{1}{0} & & & $59$ & $70$ & $58$ & $71$ & $57$ & $72$ & & $E_7$ \\ \cline{2-11}
& $65$ & \Eroot{1}{1}{2}{2}{1}{1}{1}{1} & & & $60$ & & $72$ & & & $58$ \\ \cline{2-11}
& $66$ & \Eroot{1}{1}{1}{2}{2}{2}{1}{0} & $61$ & & $71$ & & & $59$ & & $73$ & & $E_7$ \\ \cline{2-11}
& $67$ & \Eroot{1}{1}{1}{2}{2}{1}{1}{1} & $62$ & & $72$ & & $60$ & $73$ & & $59$ \\ \cline{2-11}
& $68$ & \Eroot{0}{1}{1}{2}{2}{2}{1}{1} & $73$ & & & & & $62$ & $74$ & $61$ \\ \cline{1-11}
$11$ & $69$ & \Eroot{1}{2}{2}{3}{2}{1}{0}{0} & & $63$ & & & & & $75$ & & $E_6$ & $E_7$ \\ \cline{2-11}
& $70$ & \Eroot{1}{1}{2}{3}{2}{1}{1}{0} & & $75$ & & $64$ & & $76$ & $63$ & $77$ & & $E_7$ \\ \cline{2-11}
& $71$ & \Eroot{1}{1}{2}{2}{2}{2}{1}{0} & & & $66$ & $76$ & & $64$ & & $78$ & & $E_7$ \\ \cline{2-11}
& $72$ & \Eroot{1}{1}{2}{2}{2}{1}{1}{1} & & & $67$ & $77$ & $65$ & $78$ & & $64$ \\ \cline{2-11}
& $73$ & \Eroot{1}{1}{1}{2}{2}{2}{1}{1} & $68$ & & $78$ & & & $67$ & $79$ & $66$ \\ \cline{2-11}
& $74$ & \Eroot{0}{1}{1}{2}{2}{2}{2}{1} & $79$ & & & & & & $68$ & \\ \cline{1-11}
$12$ & $75$ & \Eroot{1}{2}{2}{3}{2}{1}{1}{0} & & $70$ & & & & $80$ & $69$ & $81$ & & $E_7$ \\ \cline{2-11}
& $76$ & \Eroot{1}{1}{2}{3}{2}{2}{1}{0} & & $80$ & & $71$ & $82$ & $70$ & & $83$ & & $E_7$ \\ \cline{2-11}
& $77$ & \Eroot{1}{1}{2}{3}{2}{1}{1}{1} & & $81$ & & $72$ & & $83$ & & $70$ \\ \cline{2-11}
& $78$ & \Eroot{1}{1}{2}{2}{2}{2}{1}{1} & & & $73$ & $83$ & & $72$ & $84$ & $71$ \\ \cline{2-11}
& $79$ & \Eroot{1}{1}{1}{2}{2}{2}{2}{1} & $74$ & & $84$ & & & & $73$ & \\ \cline{1-11}
\end{tabular}
\end{table}
\begin{table}[p]
\centering
\caption{List of positive roots for Coxeter group of type $E_8$ (part $5$)}
\label{tab:positive_roots_E_8_5}
\begin{tabular}{|c||c|c|c|c|c|c|c|c|c|c|cc} \cline{1-11}
height & $i$ & root $\gamma_i$ & \multicolumn{8}{|c|}{index $k$ with $r_j \cdot \gamma_i = \gamma_k$} \\ \cline{4-11}
& & & $r_1$ & $r_2$ & $r_3$ & $r_4$ & $r_5$ & $r_6$ & $r_7$ & $r_8$ \\ \cline{1-11}
$13$ & $80$ & \Eroot{1}{2}{2}{3}{2}{2}{1}{0} & & $76$ & & & $85$ & $75$ & & $86$ & & $E_7$ \\ \cline{2-11}
& $81$ & \Eroot{1}{2}{2}{3}{2}{1}{1}{1} & & $77$ & & & & $86$ & & $75$ \\ \cline{2-11}
& $82$ & \Eroot{1}{1}{2}{3}{3}{2}{1}{0} & & $85$ & & & $76$ & & & $87$ & & $E_7$ \\ \cline{2-11}
& $83$ & \Eroot{1}{1}{2}{3}{2}{2}{1}{1} & & $86$ & & $78$ & $87$ & $77$ & $88$ & $76$ \\ \cline{2-11}
& $84$ & \Eroot{1}{1}{2}{2}{2}{2}{2}{1} & & & $79$ & $88$ & & & $78$ & \\ \cline{1-11}
$14$ & $85$ & \Eroot{1}{2}{2}{3}{3}{2}{1}{0} & & $82$ & & $89$ & $80$ & & & $90$ & & $E_7$ \\ \cline{2-11}
& $86$ & \Eroot{1}{2}{2}{3}{2}{2}{1}{1} & & $83$ & & & $90$ & $81$ & $91$ & $80$ \\ \cline{2-11}
& $87$ & \Eroot{1}{1}{2}{3}{3}{2}{1}{1} & & $90$ & & & $83$ & & $92$ & $82$ \\ \cline{2-11}
& $88$ & \Eroot{1}{1}{2}{3}{2}{2}{2}{1} & & $91$ & & $84$ & $92$ & & $83$ & \\ \cline{1-11}
$15$ & $89$ & \Eroot{1}{2}{2}{4}{3}{2}{1}{0} & & & $93$ & $85$ & & & & $94$ & & $E_7$ \\ \cline{2-11}
& $90$ & \Eroot{1}{2}{2}{3}{3}{2}{1}{1} & & $87$ & & $94$ & $86$ & & $95$ & $85$ \\ \cline{2-11}
& $91$ & \Eroot{1}{2}{2}{3}{2}{2}{2}{1} & & $88$ & & & $95$ & & $86$ & \\ \cline{2-11}
& $92$ & \Eroot{1}{1}{2}{3}{3}{2}{2}{1} & & $95$ & & & $88$ & $96$ & $87$ & \\ \cline{1-11}
$16$ & $93$ & \Eroot{1}{2}{3}{4}{3}{2}{1}{0} & $97$ & & $89$ & & & & & $98$ & & $E_7$ \\ \cline{2-11}
& $94$ & \Eroot{1}{2}{2}{4}{3}{2}{1}{1} & & & $98$ & $90$ & & & $99$ & $89$ \\ \cline{2-11}
& $95$ & \Eroot{1}{2}{2}{3}{3}{2}{2}{1} & & $92$ & & $99$ & $91$ & $100$ & $90$ & \\ \cline{2-11}
& $96$ & \Eroot{1}{1}{2}{3}{3}{3}{2}{1} & & $100$ & & & & $92$ & & \\ \cline{1-11}
$17$ & $97$ & \Eroot{2}{2}{3}{4}{3}{2}{1}{0} & $93$ & & & & & & & $101$ & & $E_7$ \\ \cline{2-11}
& $98$ & \Eroot{1}{2}{3}{4}{3}{2}{1}{1} & $101$ & & $94$ & & & & $102$ & $93$ \\ \cline{2-11}
& $99$ & \Eroot{1}{2}{2}{4}{3}{2}{2}{1} & & & $102$ & $95$ & & $103$ & $94$ & \\ \cline{2-11}
& $100$ & \Eroot{1}{2}{2}{3}{3}{3}{2}{1} & & $96$ & & $103$ & & $95$ & & \\ \cline{1-11}
\end{tabular}
\end{table}
\begin{table}[p]
\centering
\caption{List of positive roots for Coxeter group of type $E_8$ (part $6$)}
\label{tab:positive_roots_E_8_6}
\begin{tabular}{|c||c|c|c|c|c|c|c|c|c|c|cc} \cline{1-11}
height & $i$ & root $\gamma_i$ & \multicolumn{8}{|c|}{index $k$ with $r_j \cdot \gamma_i = \gamma_k$} \\ \cline{4-11}
& & & $r_1$ & $r_2$ & $r_3$ & $r_4$ & $r_5$ & $r_6$ & $r_7$ & $r_8$ \\ \cline{1-11}
$18$ & $101$ & \Eroot{2}{2}{3}{4}{3}{2}{1}{1} & $98$ & & & & & & $104$ & $97$ \\ \cline{2-11}
& $102$ & \Eroot{1}{2}{3}{4}{3}{2}{2}{1} & $104$ & & $99$ & & & $105$ & $98$ & \\ \cline{2-11}
& $103$ & \Eroot{1}{2}{2}{4}{3}{3}{2}{1} & & & $105$ & $100$ & $106$ & $99$ & & \\ \cline{1-11}
$19$ & $104$ & \Eroot{2}{2}{3}{4}{3}{2}{2}{1} & $102$ & & & & & $107$ & $101$ & \\ \cline{2-11}
& $105$ & \Eroot{1}{2}{3}{4}{3}{3}{2}{1} & $107$ & & $103$ & & $108$ & $102$ & & \\ \cline{2-11}
& $106$ & \Eroot{1}{2}{2}{4}{4}{3}{2}{1} & & & $108$ & & $103$ & & & \\ \cline{1-11}
$20$ & $107$ & \Eroot{2}{2}{3}{4}{3}{3}{2}{1} & $105$ & & & & $109$ & $104$ & & \\ \cline{2-11}
& $108$ & \Eroot{1}{2}{3}{4}{4}{3}{2}{1} & $109$ & & $106$ & $110$ & $105$ & & & \\ \cline{1-11}
$21$ & $109$ & \Eroot{2}{2}{3}{4}{4}{3}{2}{1} & $108$ & & & $111$ & $107$ & & & \\ \cline{2-11}
& $110$ & \Eroot{1}{2}{3}{5}{4}{3}{2}{1} & $111$ & $112$ & & $108$ & & & & \\ \cline{1-11}
$22$ & $111$ & \Eroot{2}{2}{3}{5}{4}{3}{2}{1} & $110$ & $113$ & $114$ & $109$ & & & & \\ \cline{2-11}
& $112$ & \Eroot{1}{3}{3}{5}{4}{3}{2}{1} & $113$ & $110$ & & & & & & \\ \cline{1-11}
$23$ & $113$ & \Eroot{2}{3}{3}{5}{4}{3}{2}{1} & $112$ & $111$ & $115$ & & & & & \\ \cline{2-11}
& $114$ & \Eroot{2}{2}{4}{5}{4}{3}{2}{1} & & $115$ & $111$ & & & & & \\ \cline{1-11}
$24$ & $115$ & \Eroot{2}{3}{4}{5}{4}{3}{2}{1} & & $114$ & $113$ & $116$ & & & & \\ \cline{1-11}
$25$ & $116$ & \Eroot{2}{3}{4}{6}{4}{3}{2}{1} & & & & $115$ & $117$ & & & \\ \cline{1-11}
$26$ & $117$ & \Eroot{2}{3}{4}{6}{5}{3}{2}{1} & & & & & $116$ & $118$ & & \\ \cline{1-11}
$27$ & $118$ & \Eroot{2}{3}{4}{6}{5}{4}{2}{1} & & & & & & $117$ & $119$ & \\ \cline{1-11}
$28$ & $119$ & \Eroot{2}{3}{4}{6}{5}{4}{3}{1} & & & & & & & $118$ & $120$ \\ \cline{1-11}
$29$ & $120$ & \Eroot{2}{3}{4}{6}{5}{4}{3}{2} & & & & & & & & $119$ \\ \cline{1-11}
\end{tabular}
\end{table}

Then we have the following:
\begin{lem}
\label{lem:possibility_J_is_E_6}
In this setting, if $J$ is of type $E_6$, then $|I \cap J| = 1$.
\end{lem}
\begin{proof}
By the property $N \geq |I \cap J| + 2$ and the $A_{>1}$-freeness of $I$, it follows that $I \cap J$ is either $\{r_2,r_3,r_4,r_5\}$ (of type $D_4$) or the union of irreducible components of type $A_1$.
In the former case, we have $\Phi_J^{\perp I} = \emptyset$ (see Tables \ref{tab:positive_roots_E_8_1}--\ref{tab:positive_roots_E_8_6}), a contradiction.
Therefore, $I \cap J$ consists of irreducible components of type $A_1$.

Now assume contrary that $I \cap J$ is not irreducible.
Then, by applying successive local transformations and by using symmetry, we may assume without loss of generality that $r_1 \in I$ (cf., the proof of Lemma \ref{lem:J_not_A_N}).
Now we have $\Pi^{J,\{r_1\}} = \{\alpha_2,\alpha_4,\alpha_5,\alpha_6,\alpha'\}$ which is the standard labelling of type $A_5$, where $\alpha'$ is the root $\gamma_{44}$ in Table \ref{tab:positive_roots_E_8_3}.
Note that $\Pi_{(I \cap J) \smallsetminus \{r_1\}} \subseteq \Pi^{J,\{r_1\}}$.
Now the same argument as Lemma \ref{lem:J_not_A_N} implies that the subspace $V'$ spanned by $\Pi^{J,I \cap J} \cup \Pi_{(I \cap J) \smallsetminus \{r_1\}}$ is a proper subspace of the space spanned by $\Pi^{J,\{r_1\}}$, therefore $\dim V' < 5$.
This implies that the subspace spanned by $\Pi^{J,I \cap J} \cup \Pi_{I \cap J}$, which is the sum of $V'$ and $\mathbb{R}\alpha_1$, has dimension less than $6 = \dim V_J$, contradicting the fact that $\Pi^{J,I \cap J} \cup \Pi_{I \cap J}$ spans $V_J$ (see Lemma \ref{lem:Psi_is_full}).
Hence $I \cap J$ is irreducible, therefore the claim holds.
\end{proof}

We also give a list of all positive roots of the Coxeter group of type $D_n$ (Table \ref{tab:positive_roots_D_n}) in order to prove the next lemma (and some other results below).
Some notations are similar to the above case of type $E_8$.
For the data for actions of generators on the roots, if the action $r_k \cdot \gamma$ does not appear in the list, then it means either $r_k$ fixes $\gamma$ (or equivalently, $\gamma$ is orthogonal to $\alpha_k$), or $\gamma = \alpha_k$.
Again, these data imply that the list indeed exhausts all the positive roots.
\begin{table}[hbt]
\centering
\caption{List of positive roots for Coxeter group of type $D_n$}
\label{tab:positive_roots_D_n}
\begin{tabular}{|c|c|} \hline
roots & actions of generators \\ \hline
$\gamma^{(1)}_{i,j} := \sum_{h=i}^{j} \alpha_h$ & $r_{i-1} \cdot \gamma^{(1)}_{i,j} = \gamma^{(1)}_{i-1,j}$ ($i \geq 2$) \\
($1 \leq i \leq j \leq n-2$) & $r_i \cdot \gamma^{(1)}_{i,j} = \gamma^{(1)}_{i+1,j}$ ($i \leq j-1$) \\
($\gamma^{(1)}_{i,i} = \alpha_i$) & $r_j \cdot \gamma^{(1)}_{i,j} = \gamma^{(1)}_{i,j-1}$ ($i \leq j-1$) \\
& $r_{j+1} \cdot \gamma^{(1)}_{i,j} = \gamma^{(1)}_{i,j+1}$ ($j \leq n-3$) \\
& $r_{n-1} \cdot \gamma^{(1)}_{i,n-2} = \gamma^{(2)}_i$ \\
& $r_n \cdot \gamma^{(1)}_{i,n-2} = \gamma^{(3)}_i$ \\ \hline
$\gamma^{(2)}_i := \sum_{h=i}^{n-1} \alpha_h$ & $r_{i-1} \cdot \gamma^{(2)}_i = \gamma^{(2)}_{i-1}$ ($i \geq 2$) \\
($1 \leq i \leq n-1$) & $r_i \cdot \gamma^{(2)}_i = \gamma^{(2)}_{i+1}$ ($i \leq n-2$) \\
($\gamma^{(2)}_{n-1} = \alpha_{n-1}$) & $r_{n-1} \cdot \gamma^{(2)}_i = \gamma^{(1)}_{i,n-2}$ ($i \leq n-2$) \\
& $r_n \cdot \gamma^{(2)}_i = \gamma^{(4)}_{i,n-1}$ ($i \leq n-2$) \\ \hline
$\gamma^{(3)}_i := \sum_{h=i}^{n-2} \alpha_h + \alpha_n$ & $r_{i-1} \cdot \gamma^{(3)}_i = \gamma^{(3)}_{i-1}$ ($i \geq 2$) \\
($1 \leq i \leq n-1$) & $r_i \cdot \gamma^{(3)}_i = \gamma^{(3)}_{i+1}$ ($i \leq n-2$) \\
($\gamma^{(3)}_{n-1} = \alpha_n$) & $r_n \cdot \gamma^{(3)}_i = \gamma^{(1)}_{i,n-2}$ ($i \leq n-2$) \\
& $r_{n-1} \cdot \gamma^{(3)}_i = \gamma^{(4)}_{i,n-1}$ ($i \leq n-2$) \\ \hline
$\gamma^{(4)}_{i,j} := \sum_{h=i}^{j-1} \alpha_h + \sum_{h=j}^{n-2} 2\alpha_h + \alpha_{n-1} + \alpha_n$ & $r_{i-1} \cdot \gamma^{(4)}_{i,j} = \gamma^{(4)}_{i-1,j}$ ($i \geq 2$) \\
($1 \leq i < j \leq n-1$) & $r_i \cdot \gamma^{(4)}_{i,j} = \gamma^{(4)}_{i+1,j}$ ($i \leq j-2$) \\
& $r_{j-1} \cdot \gamma^{(4)}_{i,j} = \gamma^{(4)}_{i,j-1}$ ($i \leq j-2$) \\
& $r_j \cdot \gamma^{(4)}_{i,j} = \gamma^{(4)}_{i,j+1}$ ($j \leq n-2$) \\
& $r_{n-1} \cdot \gamma^{(4)}_{i,n-1} = \gamma^{(3)}_i$ \\
& $r_n \cdot \gamma^{(4)}_{i,n-1} = \gamma^{(2)}_i$ \\ \hline
\end{tabular}
\end{table}

Then we have the following:
\begin{lem}
\label{lem:possibility_J_is_D_N}
In this setting, suppose that $J$ is of type $D_N$.
\begin{enumerate}
\item \label{item:lem_possibility_J_is_D_N_case_1}
If $I \cap J$ has an irreducible component of type $D_k$ with $k \geq 4$ and $N - k$ is odd, then we have $|I \cap J| \leq k + (N-k-3)/2$.
\item \label{item:lem_possibility_J_is_D_N_case_2}
If $N$ is odd, $I \cap J$ does not have an irreducible component of type $D_k$ with $k \geq 4$ and $\{r_{N-1},r_N\} \not\subseteq I$, then we have $|I \cap J| \leq (N-3)/2$.
\item \label{item:lem_possibility_J_is_D_N_case_3}
If $N$ is odd, $I \cap J$ does not have an irreducible component of type $D_k$ with $k \geq 4$ and $\{r_{N-1},r_N\} \subseteq I$, then we have $|I \cap J| \leq (N-1)/2$.
\end{enumerate}
\end{lem}
\begin{proof}
Assume contrary that the hypothesis of one of the three cases in the statement is satisfied but the inequality in the conclusion does not hold.
We show that $\Pi^{J,I \cap J} \cup \Pi_{I \cap J}$ cannot span $V_J$, which is a contradiction and therefore concludes the proof.
First, recall the property $N \geq |I \cap J| + 2$ and the $A_{>1}$-freeness of $I$.
Then, in the case \ref{item:lem_possibility_J_is_D_N_case_1}, by applying successive local transformations, we may assume without loss of generality that $I \cap J$ consists of elements $r_{2j}$ with $1 \leq j \leq (N-k-1)/2$ and $r_j$ with $N-k+1 \leq j \leq N$.
Similarly, in the case \ref{item:lem_possibility_J_is_D_N_case_2} (respectively, the case \ref{item:lem_possibility_J_is_D_N_case_3}), by applying successive local transformations and using symmetry, we may assume without loss of generality that $I \cap J$ consists of elements $r_{2j}$ with $1 \leq j \leq (N-1)/2$ (respectively, $r_{2j}$ with $1 \leq j \leq (N-1)/2$ and $r_N$).
In any case, we have $\Phi_J^{\perp I} \subseteq \Phi_{J \smallsetminus \{r_1\}}$ (see Table \ref{tab:positive_roots_D_n}), therefore the subspace spanned by $\Pi^{J,I \cap J} \cup \Pi_{I \cap J}$ is contained in $V_{J \smallsetminus \{r_1\}}$.
Hence $\Pi^{J,I \cap J} \cup \Pi_{I \cap J}$ cannot span $V_J$, concluding the proof.
\end{proof}

We divide the following argument into two cases.

\subsubsection{Case $w \cdot \Pi_J \not\subseteq \Phi^+$}
\label{sec:proof_special_first_case_subcase_1}

In order to prove that $w \cdot \Pi_J \subseteq \Phi^+$, here we assume contrary that $w \cdot \Pi_J \not\subseteq \Phi^+$ and deduce a contradiction.

In this setting, we construct a decomposition of $w$ in the following manner.
Take an element $s \in J$ with $w \cdot \alpha_s \in \Phi^-$.
By Lemma \ref{lem:rightdivisor}, the element $w_{x_I}^s$ is a right divisor of $w$.
This implies that $\Phi^{\perp I}[w_{x_I}^s] \subseteq \Phi^{\perp I}[w] = \emptyset$ (see Lemma 2.2 of \cite{Nui11} for the first inclusion), therefore we have $w_{x_I}^s \in Y_{y,x_I}$ where we put $y := \varphi(x_I,s) \in S^{(\Lambda)}$.
By Proposition \ref{prop:charofBphi}, we have $y \neq x_I$.
This element $w_{x_I}^s$ induces a local transformation $x_I \mapsto y$.
Now if $w(w_{x_I}^s)^{-1} \cdot \Pi_J \not\subseteq \Phi^+$, then we can similarly factor out from $w(w_{x_I}^s)^{-1}$ a right divisor of the form $w_y^t \in Y_{\varphi(y,t),y}$ with $t \in J$.
Iterating this process, we finally obtain a decomposition of $w$ of the form $w = u w_{y_{n-1}}^{s_{n-1}} \cdots w_{y_1}^{s_1} w_{y_0}^{s_0}$ satisfying that $n \geq 1$, $u \in Y_{x_I,z}$ with $z \in S^{(\Lambda)}$, $w_{y_i}^{s_i} \in Y_{y_{i+1},y_i} \cap W_J$ for every $0 \leq i \leq n-1$ where we put $y_0 = x_I$ and $y_n = z$, and $u \cdot \Pi_J \subseteq \Phi^+$.

Put $u' := w_{y_{n-1}}^{s_{n-1}} \cdots w_{y_1}^{s_1} w_{y_0}^{s_0} \neq 1$.
By the construction, the action of $u' \in Y_{z,x_I} \cap W_J$ induces (as the composition of successive local transformations) an isomorphism $\sigma \colon I \cap J \to [z] \cap J$, $t \mapsto u' \ast t$, while $u'$ fixes every element of $\Pi_{I \smallsetminus J}$.
Now $\sigma$ is not an identity mapping; otherwise, we have $z = x_I$ and $1 \neq u' \in Y_{x_I,x_I}$, while $u'$ has finite order since $|W_J| < \infty$, contradicting Proposition \ref{prop:Yistorsionfree}.
On the other hand, we have $u \cdot \Phi_J = wu'{}^{-1} \cdot \Phi_J = w \cdot \Phi_J = \Phi_J$, therefore $u \cdot \Phi_J^+ = \Phi_J^+$ since $u \cdot \Pi_J \subseteq \Phi^+$.
This implies that $u \cdot \Pi_J = \Pi_J$, therefore the action of $u$ defines an automorphism $\tau$ of $J$.
Since $w = u u' \in Y_I$, the composite mapping $\tau \circ \sigma$ is the identity mapping on $I \cap J$, while $\sigma$ is not identity as above.
As a consequence, we have $\tau^{-1}|_{I \cap J} = \sigma$ and hence $\tau^{-1}$ is a nontrivial automorphism of $J$, therefore the possibilities of the type of $J$ are $D_N$, $E_6$ and $F_4$ (recall that $J$ is neither of type $A_N$ nor of type $I_2(m)$).
\begin{lem}
\label{lem:proof_special_first_case_subcase_1_not_F_4}
In this setting, $J$ is not of type $F_4$.
\end{lem}
\begin{proof}
Assume contrary that $J = \{r_1,r_2,r_3,r_4\}$ is of type $F_4$.
In this case, each of $r_1$ and $r_2$ is not conjugate in $W_J$ to one of $r_3$ and $r_4$ by the well-known fact that the conjugacy classes for the simple reflections $r_i$ are determined by the connected components of the graph obtained from the Coxeter graph by removing all edges having non-odd labels.
Therefore, the mapping $\tau^{-1}|_{I \cap J} = \sigma$ induced by the action of $u' \in W_J$ cannot map an element $r_i$ ($1 \leq i \leq 4$) to $r_{5-i}$.
This contradicts the fact that $\tau^{-1}$ is a nontrivial automorphism of $J$.
Hence the claim holds.
\end{proof}
From now, we consider the remaining case that $J$ is either of type $D_N$ with $4 \leq N < \infty$ or of type $E_6$.
Take a standard decomposition $\mathcal{D} = \omega_{\ell(\mathcal{D})-1} \cdots \omega_1\omega_0$ of $u \in Y_{x_I,z}$ with respect to $J$ (see Proposition \ref{prop:standard_decomposition_existence}).
Note that $J$ is irreducible and $J \not\subseteq [z]$.
This implies that, if $0 \leq i \leq \ell(\mathcal{D})-1$ and $\omega_j$ is a narrow transformation for every $0 \leq j \leq i$, then it follows by induction on $0 \leq j \leq i$ that the support of $\omega_j$ is apart from $J$, the product $\omega_j \cdots \omega_1\omega_0$ fixes $\Pi_J$ pointwise, $[y^{(j+1)}] \cap J = [z] \cap J$, and $[y^{(j+1)}] \smallsetminus J$ is not adjacent to $J$ (note that $[z] \smallsetminus J = I \smallsetminus J$ is not adjacent to $J$).
By these properties, since $u$ does not fix $\Pi_J$ pointwise, $\mathcal{D}$ contains at least one wide transformation.
Let $\omega := \omega_i$ be the first (from right) wide transformation in $\mathcal{D}$, and write $y = y^{(i)}(\mathcal{D})$, $t = t^{(i)}(\mathcal{D})$ and $K = K^{(i)}(\mathcal{D})$ for simplicity.
Note that $J^{(i)}(\mathcal{D}) = J$ by the above argument.
Note also that $\Pi^{K,[y] \cap K} \subseteq \Pi^{[y]}$, since $[y] \smallsetminus K$ is not adjacent to $K$ by the definition of $K$.
Now the action of $\omega_{i-1} \cdots \omega_1\omega_0 u' \in Y_{y,x_I}$ induces an isomorphism $\Pi^I \to \Pi^{[y]}$ which maps $\Pi^{J,I \cap J}$ onto $\Pi^{J,[y] \cap J} = \Pi^{J,[z] \cap J}$.
Hence we have the following (recall that $\Pi^{J,I \cap J}$ is the union of some irreducible components of $\Pi^I$):
\begin{lem}
\label{lem:proof_special_first_case_subcase_1_local_irreducible_components}
In this setting, $\Pi^{J,[y] \cap J}$ is isomorphic to $\Pi^{J,I \cap J}$ and is the union of some irreducible components of $\Pi^{[y]}$.
In particular, each element of $\Pi^{J,[y] \cap J}$ is orthogonal to any element of $\Pi^{K,[y] \cap K} \smallsetminus \Phi_J$.
\end{lem}
Now note that $K = ([y] \cup J)_{\sim t}$ is irreducible and of finite type, and $t$ is adjacent to $[y]$.
Moreover, by Lemma \ref{lem:another_decomposition_Y_no_loop}, the element $\omega = \omega_{y,J}^{t}$ does not fix $\Pi_{K \smallsetminus \{t\}}$ pointwise.
By these properties and symmetry, we may assume without loss of generality that the possibilities of $K$ are as follows:
\begin{enumerate}
\item \label{item:proof_special_first_case_subcase_1_J_E_6}
$J$ is of type $E_6$, and;
\begin{enumerate}
\item \label{item:proof_special_first_case_subcase_1_J_E_6_K_E_8}
$K = J \cup \{t,t'\}$ is of type $E_8$ where $t$ is adjacent to $r_6$ and $t'$, and $t' \in [y]$,
\item \label{item:proof_special_first_case_subcase_1_J_E_6_K_E_7}
$K = J \cup \{t\}$ is of type $E_7$ where $t$ is adjacent to $r_6$, and $r_6 \in [y]$,
\end{enumerate}
\item \label{item:proof_special_first_case_subcase_1_J_D_7}
$J$ is of type $D_7$, $K = J \cup \{t\}$ is of type $E_8$ where $t$ is adjacent to $r_7$, and $r_7 \in [y]$,
\item \label{item:proof_special_first_case_subcase_1_J_D_5}
$J$ is of type $D_5$, and;
\begin{enumerate}
\item \label{item:proof_special_first_case_subcase_1_J_D_5_K_E_7}
$K = J \cup \{t,t'\}$ is of type $E_7$ where $t$ is adjacent to $r_5$ and $t'$, and $t' \in [y]$,
\item \label{item:proof_special_first_case_subcase_1_J_D_5_K_E_6}
$K = J \cup \{t\}$ is of type $E_6$ where $t$ is adjacent to $r_5$, and $r_5 \in [y]$,
\end{enumerate}
\item \label{item:proof_special_first_case_subcase_1_J_D_N}
$J$ is of type $D_N$, $K = J \cup \{t\}$ is of type $D_{N+1}$ where $t$ is adjacent to $r_1$, and $r_1 \in [y]$.
\end{enumerate}

We consider Case \ref{item:proof_special_first_case_subcase_1_J_E_6_K_E_8}.
We have $|[y] \cap J| = |I \cap J| = 1$ by Lemma \ref{lem:possibility_J_is_E_6}.
Now by Tables \ref{tab:positive_roots_E_8_1}--\ref{tab:positive_roots_E_8_6} (where $r_7 = t$ and $r_8 = t'$), we have $\langle \beta,\beta' \rangle \neq 0$ for some $\beta \in \Pi^{J,[y] \cap J}$ and $\beta' \in \Pi^{K,[y] \cap K} \smallsetminus \Phi_J$ (namely, $(\beta,\beta') = (\alpha_4,\gamma_{84})$ when $[y] \cap J = \{r_1\}$; $(\beta,\beta') = (\gamma_{16},\gamma_{74})$ when $[y] \cap J = \{r_3\}$; and $(\beta,\beta') = (\alpha_1,\gamma_{74})$ when $[y] \cap J = \{r_j\}$ with $j \in \{2,4,5,6\}$, where the roots $\gamma_k$ are as in Tables \ref{tab:positive_roots_E_8_1}--\ref{tab:positive_roots_E_8_6}).
This contradicts Lemma \ref{lem:proof_special_first_case_subcase_1_local_irreducible_components}.

We consider Case \ref{item:proof_special_first_case_subcase_1_J_E_6_K_E_7}.
We have $|[y] \cap J| = |I \cap J| = 1$ by Lemma \ref{lem:possibility_J_is_E_6}, hence $[y] \cap J = \{r_6\}$.
Now we have $\alpha_5 + \alpha_6 + \alpha_t \in \Pi^{K,[y] \cap K} \smallsetminus \Phi_J$, $\alpha_4 \in \Pi^{J,[y] \cap J}$, and these two roots are not orthogonal, contradicting Lemma \ref{lem:proof_special_first_case_subcase_1_local_irreducible_components}.

We consider Case \ref{item:proof_special_first_case_subcase_1_J_D_7}.
Note that $N = 7 \geq |I \cap J| + 2 = |[y] \cap J| + 2$, therefore $|[y] \cap J| \leq 5$.
By Lemma \ref{lem:possibility_J_is_D_N} and $A_{>1}$-freeness of $I$, it follows that the possibilities of $[y] \cap J$ are as listed in Table \ref{tab:proof_special_first_case_subcase_1_J_D_7}, where we put $(r'_1,r'_2,r'_3,r'_4,r'_5,r'_6,r'_7,r'_8) = (t,r_6,r_7,r_5,r_4,r_3,r_2,r_1)$ (hence $K = \{r'_1,\dots,r'_8\}$ is the standard labelling of type $E_8$).
Now by Tables \ref{tab:positive_roots_E_8_1}--\ref{tab:positive_roots_E_8_6}, we have $\langle \beta,\beta' \rangle \neq 0$ for some $\beta \in \Pi^{J,[y] \cap J}$ and $\beta' \in \Pi^{K,[y] \cap K} \smallsetminus \Phi_J$ as listed in Table \ref{tab:proof_special_first_case_subcase_1_J_D_7}, where we write $\alpha'_j = \alpha_{r'_j}$ and the roots $\gamma_k$ are as in Tables \ref{tab:positive_roots_E_8_1}--\ref{tab:positive_roots_E_8_6}.
This contradicts Lemma \ref{lem:proof_special_first_case_subcase_1_local_irreducible_components}.

\begin{table}[hbt]
\centering
\caption{List of roots for Case \ref{item:proof_special_first_case_subcase_1_J_D_7}}
\label{tab:proof_special_first_case_subcase_1_J_D_7}
\begin{tabular}{|c|c|c|} \hline
$[y] \cap J$ & $\beta$ & $\beta'$ \\ \hline
$r'_3 \in [y] \cap J \subseteq \{r'_3,r'_6,r'_7,r'_8\}$ & $\alpha'_2$ & $\gamma_{16}$ \\ \hline
$\{r'_3,r'_5\}$ & $\alpha'_2$ & $\gamma_{31}$ \\ \hline
$\{r'_2,r'_3\} \subseteq [y] \cap J \subseteq \{r'_2,r'_3,r'_4,r'_5,r'_6\}$ & $\alpha'_8$ & $\gamma_{97}$ \\ \cline{1-2}
$\{r'_2,r'_3,r'_7\}$ & $\gamma_{22}$ & \\ \hline
$\{r'_2,r'_3,r'_8\}$ & $\alpha'_6$ & $\gamma_{104}$ \\ \hline
\end{tabular}
\end{table}

We consider Case \ref{item:proof_special_first_case_subcase_1_J_D_5_K_E_7}.
Note that $N = 5 \geq |I \cap J| + 2 = |[y] \cap J| + 2$, therefore $|[y] \cap J| \leq 3$.
By $A_{>1}$-freeness of $I$, every irreducible component of $[y] \cap J$ is of type $A_1$.
Now by Lemma \ref{lem:possibility_J_is_D_N}, the possibilities of $[y] \cap J$ are as listed in Table \ref{tab:proof_special_first_case_subcase_1_J_D_5_K_E_7}, where we put $(r'_1,r'_2,r'_3,r'_4,r'_5,r'_6,r'_7) = (r_1,r_4,r_2,r_3,r_5,t,t')$ (hence $K = \{r'_1,\dots,r'_7\}$ is the standard labelling of type $E_7$).
Now by Tables \ref{tab:positive_roots_E_8_1}--\ref{tab:positive_roots_E_8_6}, we have $\langle \beta,\beta' \rangle \neq 0$ for some $\beta \in \Pi^{J,[y] \cap J}$ and $\beta' \in \Pi^{K,[y] \cap K} \smallsetminus \Phi_J$ as listed in Table \ref{tab:proof_special_first_case_subcase_1_J_D_5_K_E_7}, where we write $\alpha'_j = \alpha_{r'_j}$ and the roots $\gamma_k$ are as in Tables \ref{tab:positive_roots_E_8_1}--\ref{tab:positive_roots_E_8_6}.
This contradicts Lemma \ref{lem:proof_special_first_case_subcase_1_local_irreducible_components}.

\begin{table}[hbt]
\centering
\caption{List of roots for Case \ref{item:proof_special_first_case_subcase_1_J_D_5_K_E_7}}
\label{tab:proof_special_first_case_subcase_1_J_D_5_K_E_7}
\begin{tabular}{|c|c|c|} \hline
$[y] \cap J$ & $\beta$ & $\beta'$ \\ \hline
$[y] \cap J \subseteq \{r'_2,r'_4,r'_5\}$ & $\alpha'_1$ & $\gamma_{61}$ \\ \cline{1-2}
$\{r'_3\}$ & $\gamma_{16}$ & \\ \hline
$\{r'_1\}$ & $\alpha'_4$ & $\gamma_{71}$ \\ \hline
\end{tabular}
\end{table}

We consider Case \ref{item:proof_special_first_case_subcase_1_J_D_5_K_E_6}.
By the same reason as Case \ref{item:proof_special_first_case_subcase_1_J_D_5_K_E_7}, every irreducible component of $[y] \cap J$ is of type $A_1$.
Now by Lemma \ref{lem:possibility_J_is_D_N}, we have only two possibilities of $[y] \cap J$; $\{r_5\}$ and $\{r_4,r_5\}$.
In the first case $[y] \cap J = \{r_5\}$, we have $\alpha_2 \in \Pi^{J,[y] \cap J}$, $\alpha_3 + \alpha_5 + \alpha_t \in \Pi^{K,[y] \cap K} \smallsetminus \Phi_J$, and these two roots are not orthogonal, contradicting Lemma \ref{lem:proof_special_first_case_subcase_1_local_irreducible_components}.
Hence we consider the second case $[y] \cap J = \{r_4,r_5\}$.
In this case, the action of the first wide transformation $\omega$ in $\mathcal{D}$ maps the elements $r_1$, $r_2$, $r_3$, $r_4$ and $r_5$ to $t$, $r_5$, $r_3$, $r_2$ and $r_4$, respectively (note that $\{t,r_5,r_3,r_2,r_4\}$ is the standard labelling of type $D_5$).
Now, by a similar argument as above, the possibility of the second wide transformation $\omega_{i'}$ in $\mathcal{D}$ (if exists) is as in Case \ref{item:proof_special_first_case_subcase_1_J_D_5_K_E_6}, where $t'' := t^{(i')}(\mathcal{D})$ is adjacent to either $r_2$ or $r_4$ (note that Case \ref{item:proof_special_first_case_subcase_1_J_D_5_K_E_7} cannot occur as discussed above, while Case \ref{item:proof_special_first_case_subcase_1_J_D_N} cannot occur by the shape of $J$ and the property $r_1 \not\in [y] \cap J$).
This implies that the action of $\omega_{i'}$ either maps the elements $t$, $r_5$, $r_3$, $r_4$ and $r_2$ to $t''$, $r_2$, $r_3$, $r_5$ and $r_4$, respectively (forming a subset of type $D_5$ with the ordering being the standard labelling), or maps the elements $t$, $r_5$, $r_3$, $r_2$ and $r_4$ to $t''$, $r_4$, $r_3$, $r_5$ and $r_2$, respectively (forming a subset of type $D_5$ with the ordering being the standard labelling).
By iterating the same argument, it follows that the sequence of elements $(r_2,r_3,r_4,r_5)$ is mapped by successive wide transformations in $\mathcal{D}$ to one of the following three sequences; $(r_2,r_3,r_4,r_5)$, $(r_5,r_3,r_2,r_4)$ and $(r_4,r_3,r_5,r_2)$.
Hence $u$ itself should map $(r_2,r_3,r_4,r_5)$ to one of the above three sequences; while the action of $u$ induces the nontrivial automorphism $\tau$ of $J$, which maps $(r_1,r_2,r_3,r_4,r_5)$ to $(r_1,r_2,r_3,r_5,r_4)$.
This is a contradiction.

Finally, we consider the case \ref{item:proof_special_first_case_subcase_1_J_D_N}.
First we have the following lemma:
\begin{lem}
\label{lem:proof_special_first_case_subcase_1_J_D_N}
In this setting, suppose further that there exists an integer $k \geq 1$ satisfying that $2k \leq N - 3$, $r_{2j - 1} \in [y]$ and $r_{2j} \not\in [y]$ for every $1 \leq j \leq k$, and $r_{2k + 1} \not\in [y]$.
Then there exist a root $\beta \in \Pi^{J,[y] \cap J}$ and a root $\beta' \in \Pi^{K,[y] \cap K} \smallsetminus \Phi_J$ with $\langle \beta,\beta' \rangle \neq 0$.
\end{lem}
\begin{proof}
Put $J' := \{r_j \mid 2k+1 \leq j \leq N\}$.
First, we have $\beta' := \alpha_t + \sum_{j=1}^{2k} \alpha_j \in \Pi^{K,[y] \cap K} \smallsetminus \Phi_J$ in this case.
On the other hand, $\Pi^{J,[y] \cap J} \smallsetminus \Phi_{J'}$ consists of $k$ roots $\gamma^{(4)}_{2j-1,2j}$ with $1 \leq j \leq k$ (see Table \ref{tab:positive_roots_D_n} for the notation), while $\Pi_{[y] \cap J} \smallsetminus \Phi_{J'}$ consists of $k$ roots $\alpha_{2j-1}$ with $1 \leq j \leq k$.
Hence $|(\Pi^{J,[y] \cap J} \cup \Pi_{[y] \cap J}) \smallsetminus \Phi_{J'}| = 2k$.
Since $\Pi^{J,[y] \cap J} \cup \Pi_{[y] \cap J}$ is a basis of the space $V_J$ of dimension $N$, it follows that the subset $(\Pi^{J,[y] \cap J} \cup \Pi_{[y] \cap J}) \cap \Phi_{J'}$ spans a subspace of dimension $N - 2k = |J'|$.
This implies that $(\Pi^{J,[y] \cap J} \cup \Pi_{[y] \cap J}) \cap \Phi_{J'} \not\subseteq \Phi_{J' \smallsetminus \{r_{2k+1}\}}$, therefore (since $\alpha_{2k+1} \not\in \Pi_{[y] \cap J}$) we have $\Pi^{J,[y] \cap J} \cap \Phi_{J'} \not\subseteq \Phi_{J' \smallsetminus \{r_{2k+1}\}}$, namely there exists a root $\beta \in \Pi^{J,[y] \cap J} \cap \Phi_{J'}$ which has non-zero coefficient of $\alpha_{2k+1}$.
These $\beta$ and $\beta'$ satisfy $\langle \beta,\beta' \rangle \neq 0$ by the construction, concluding the proof.
\end{proof}
By Lemma \ref{lem:proof_special_first_case_subcase_1_J_D_N} and Lemma \ref{lem:proof_special_first_case_subcase_1_local_irreducible_components}, the hypothesis of Lemma \ref{lem:proof_special_first_case_subcase_1_J_D_N} should not hold.
By this fact, $A_{>1}$-freeness of $I$ and the property $N \geq |I \cap J| + 2 = |[y] \cap J| + 2$, it follows that the possibilities of $[y] \cap J$ are as follows (up to the symmetry $r_{N-1} \leftrightarrow r_N$); (I) $[y] \cap J = J \smallsetminus \{r_{2j} \mid 1 \leq j \leq k\}$ for an integer $k$ with $2 \leq k \leq (N-2)/2$ and $2k \neq N-3$; (II) $N$ is odd and $[y] \cap J = \{r_{2j-1} \mid 1 \leq j \leq (N-1)/2\}$; (III) $N$ is even and $[y] \cap J = \{r_{2j-1} \mid 1 \leq j \leq (N-2)/2\}$; (IV) $N$ is even and $[y] \cap J = \{r_{2j-1} \mid 1 \leq j \leq N/2\}$.
For Case (I), by the shape of $J$ and $[y] \cap J$, it follows that $I \cap J = [y] \cap J$, and each local transformation can permute the irreducible components of $I \cap J$ containing neither $r_{N-1}$ nor $r_N$ but it fixes pointwise the irreducible component(s) of $I \cap J$ containing $r_{N-1}$ or $r_N$.
This contradicts the fact that $\sigma = \tau^{-1}|_{I \cap J}$ for a nontrivial automorphism $\tau^{-1}$ of $J$ (note that $\tau^{-1}$ exchanges $r_{N-1}$ and $r_N$).
Case (II) contradicts Lemma \ref{lem:possibility_J_is_D_N}(\ref{item:lem_possibility_J_is_D_N_case_2}).
For Case (III), the roots $\alpha_{N-1} \in \Pi^{J,[y] \cap J}$ and $\alpha_t + \sum_{j=1}^{N-2} \alpha_j \in \Pi^{K,[y] \cap K} \smallsetminus \Phi_J$ are not orthogonal, contradicting Lemma \ref{lem:proof_special_first_case_subcase_1_local_irreducible_components}.

Finally, for the remaining case, i.e., Case (IV), by the shape of $J$ and $[y] \cap J$, it follows that $I \cap J = [y] \cap J$ and each local transformation leaves the set $I \cap J$ invariant.
By this result and the property that $\sigma = \tau^{-1}|_{I \cap J}$ for a nontrivial automorphism $\tau^{-1}$ of $J$, only the possibility of $[y] \cap J$ is that $N = 4$ and $[y] \cap J = I \cap J = \{r_1,r_3\}$, and $\sigma$ exchanges $r_1$ and $r_3$.
Now we arrange the standard decomposition $\mathcal{D}$ of $u$ as $u = \omega''_{\ell} \omega'_{\ell-1} \omega''_{\ell-1} \cdots \omega'_2 \omega''_2 \omega'_1 \omega''_1$, where each $\omega'_j$ is a wide transformation and each $\omega''_j$ is a (possibly empty) product of narrow transformations.
Let each wide transformation $\omega'_j$ belong to $Y_{z'_j,z_j}$ with $z_j,z'_j \in S^{(\Lambda)}$.
In particular, we have $\omega'_1 = \omega$ and $z_1 = y$.
Now we give the following lemma:
\begin{lem}
\label{lem:proof_special_first_case_subcase_1_J_D_N_case_IV}
In this setting, the following properties hold for every $1 \leq j \leq \ell - 1$: The action of the element $u_j := \omega''_j \omega'_{j-1} \omega''_{j-1} \cdots \omega'_1 \omega''_1$ maps $(r_1,r_2,r_3,r_4)$ to $(r_1,r_2,r_3,r_4)$ when $j$ is odd and to $(r_1,r_2,r_4,r_3)$ when $j$ is even; the subsets $J$ and $[z_j] \smallsetminus J$ are not adjacent; the support of $\omega'_j$ is as in Case \ref{item:proof_special_first_case_subcase_1_J_D_N} above, with $t$ replaced by some element $t_j \in S$; and $\omega'_j$ maps $(r_1,r_2,r_3,r_4)$ to $(r_1,r_2,r_4,r_3)$.
\end{lem}
\begin{proof}
We use induction on $j$.
By the definition of narrow transformations, the first and the second parts of the claim hold obviously when $j = 1$ and follow from the induction hypothesis when $j > 1$.
In particular, we have $u_j \cdot \Pi_J = \Pi_J$.
Put $(h,h') := (3,4)$ when $j$ is odd and $(h,h') := (4,3)$ when $j$ is even.
Then we have $[z_j] \cap J = \{r_1,r_h\}$.
Now, by using the above argument, it follows that the support of $\omega'_j$ is of the form $\{r_1,r_2,r_3,r_4,t_j\}$ which is the standard labelling of type $D_5$, where $t_j$ is adjacent to one of the two elements of $[z_j] \cap J$.
We show that $t_j$ is adjacent to $r_1$, which already holds when $j = 1$ (note that $t_j = t$ when $j = 1$).
Suppose $j > 1$ and assume contrary that $t_j$ is adjacent to $r_h$.
In this case, $t_j$ is apart from $[z_j] \smallsetminus \{r_h\}$.
On the other hand, we have $[z'_{j-1}] \cap J = \{r_1,r_h\}$, the subsets $[z'_{j-1}] \smallsetminus J$ and $J$ are not adjacent, and the support of each narrow transformation in $\omega''_j$ is apart from to $J$.
Moreover, by the induction hypothesis, we have $[z_{j-1}] \cap J = \{r_1,r_{h'}\}$ and the action of $\omega'_{j-1}$ maps $(r_1,r_2,r_h,r_{h'})$ to $(r_1,r_2,r_{h'},r_h)$ while it fixes every element of $[z_{j-1}] \smallsetminus J$.
This implies that $\omega''_j \in Y_{z'',z_{j-1}}$ for the element $z'' \in S^{(\Lambda)}$ obtained from $z_j$ by replacing $r_h$ with $r_{h'}$.
Now we have $\alpha_{t_j} \in \Pi^{[z'']}$ since $t_j$ is not adjacent to $[z''] = ([z_j] \smallsetminus \{r_h\}) \cup \{r_{h'}\}$, therefore $\beta' := (\omega''_j)^{-1} \cdot \alpha_{t_j} \in \Pi^{[z_{j-1}]}$.
This root belongs to $\Phi_{S \smallsetminus J}$ and has non-zero coefficient of $\alpha_{t_j}$, since the support of each narrow transformation in $\omega''_j$ is not adjacent to $J$ and hence does not contain $t_j$.
Therefore, the roots $\beta' \in \Pi^{[z_{j-1}]} \smallsetminus \Pi^{J,[z_{j-1}] \cap J}$ and $\alpha_1 + 2\alpha_2 + \alpha_3 + \alpha_4 \in \Pi^{J,[z_{j-1}] \cap J}$ are not orthogonal.
This contradicts the fact that $\Pi^{J,[y] \cap J}$ is the union of some irreducible components of $\Pi^{[y]}$ (see Lemma \ref{lem:proof_special_first_case_subcase_1_local_irreducible_components}) and the isomorphism $\Pi^{[y]} \to \Pi^{[z_{j-1}]}$ induced by the action of $\omega''_{j-1}\omega'_{j-2}\omega''_{j-2} \cdots \omega''_2\omega'_1$ maps $\Pi^{J,[y] \cap J}$ to $\Pi^{J,[z_{j-1}] \cap J}$ (since the action of this element leaves the set $\Pi_J$ invariant).
This contradiction proves that $t_j$ is adjacent to $r_1$, therefore the third part of the claim holds.
Finally, the fourth part of the claim follows immediately from the third part.
Hence the proof of Lemma \ref{lem:proof_special_first_case_subcase_1_J_D_N_case_IV} is concluded.
\end{proof}
By Lemma \ref{lem:proof_special_first_case_subcase_1_J_D_N_case_IV}, the action of the element $\omega'_{\ell-1} \omega''_{\ell-1} \cdots \omega'_2 \omega''_2 \omega'_1 \omega''_1$, hence of $u = \omega''_{\ell}\omega'_{\ell-1}u_{\ell-1}$, maps the elements $(r_1,r_2,r_3,r_4)$ to either $(r_1,r_2,r_3,r_4)$ or $(r_1,r_2,r_4,r_3)$.
This contradicts the above-mentioned fact that $\sigma$ exchanges $r_1$ and $r_3$.

Summarizing, we have derived a contradiction in each of the six possible cases, Cases \ref{item:proof_special_first_case_subcase_1_J_E_6_K_E_8}--\ref{item:proof_special_first_case_subcase_1_J_D_N}.
Hence we have proven that the assumption $w \cdot \Pi_J \not\subseteq \Phi^+$ implies a contradiction, as desired.

\subsubsection{Case $w \cdot \Pi_J \subseteq \Phi^+$}
\label{sec:proof_special_first_case_subcase_2}

By the result of Section \ref{sec:proof_special_first_case_subcase_1}, we have $w \cdot \Pi_J \subseteq \Phi^+$.
Since $w \cdot \Phi_J = \Phi_J$ by Lemma \ref{lem:property_w}, it follows that $w \cdot \Phi_J^+ \subseteq \Phi_J^+$, therefore $w \cdot \Phi_J^+ = \Phi_J^+$ (note that $|\Phi_J| < \infty$).
Hence the action of $w$ defines an automorphism $\tau$ of $J$ (in particular, $w \cdot \Pi_J = \Pi_J$).
To show that $\tau$ is the identity mapping (which implies the claim that $w$ fixes $\Pi^{J,I \cap J}$ pointwise), assume contrary that $\tau$ is a nontrivial automorphism of $J$.
Then the possibilities of the type of $J$ are as follows: $D_N$, $E_6$ and $F_4$ (recall that $J$ is neither of type $A_N$ nor of type $I_2(m)$).
Moreover, since the action of $w \in Y_I$ fixes every element of $I \cap J$, the subset $I \cap J$ of $J$ is contained in the fixed point set of $\tau$.
This implies that $J$ is not of type $F_4$, since the nontrivial automorphism of a Coxeter graph of type $F_4$ has no fixed points.

Suppose that $J$ is of type $E_6$.
Then, by the above argument on the fixed points of $\tau$ and Lemma \ref{lem:possibility_J_is_E_6}, we have $I \cap J = \{r_2\}$ or $I \cap J = \{r_4\}$.
Now take a standard decomposition of $w$ with respect to $J$ (see Proposition \ref{prop:standard_decomposition_existence}).
Then no wide transformation can appear due to the shape of $J$ and the position of $I \cap J$ in $J$ (indeed, we cannot obtain a subset of finite type by adding to $J$ an element of $S$ adjacent to $I \cap J$).
This implies that the decomposition of $w$ consists of narrow transformations only, therefore $w$ fixes $\Pi_J$ pointwise, contradicting the fact that $\tau$ is a nontrivial automorphism.

Secondly, suppose that $J$ is of type $D_N$ with $N \geq 5$.
Then, by the above argument on the fixed points of $\tau$, we have $I \cap J \subseteq J \smallsetminus \{r_{N-1},r_N\}$, therefore every irreducible component of $I \cap J$ is of type $A_1$ (by $A_{>1}$-freeness of $I$).
Now take a standard decomposition $\mathcal{D}$ of $w$ with respect to $J$ (see Proposition \ref{prop:standard_decomposition_existence}).
Note that $\mathcal{D}$ involves at least one wide transformation, since $\tau$ is not the identity mapping.
By the shape of $J$ and the position of $I \cap J$ in $J$, only the possibility of the first (from right) wide transformation $\omega = \omega_i$ in $\mathcal{D}$ is as follows: $K = J \cup \{t\}$ is of type $D_{N+1}$, $t$ is adjacent to $r_1$, and $r_1 \in [y]$, where we put $y = y^{(i)}(\mathcal{D})$, $t = t^{(i)}(\mathcal{D})$, and $K = K^{(i)}(\mathcal{D})$.
Now the claim of Lemma \ref{lem:proof_special_first_case_subcase_1_J_D_N} in Section \ref{sec:proof_special_first_case_subcase_1} also holds in this case, while $\Pi^{J,[y] \cap J}$ is the union of some irreducible components of $\Pi^{[y]}$ by the same reason as in Section \ref{sec:proof_special_first_case_subcase_1}.
Hence the hypothesis of Lemma \ref{lem:proof_special_first_case_subcase_1_J_D_N} should not hold.
This argument and the properties that $N \geq |I \cap J| + 2 = |[y] \cap J| + 2$ and $I \cap J \subseteq J \smallsetminus \{r_{N-1},r_N\}$ imply that the possibilities of $[y] \cap J$ are the followings: $N$ is odd and $[y] \cap J$ consists of elements $r_{2j-1}$ with $1 \leq j \leq (N-1)/2$; or, $N$ is even and $[y] \cap J$ consists of elements $r_{2j-1}$ with $1 \leq j \leq (N-2)/2$.
The former possibility contradicts Lemma \ref{lem:possibility_J_is_D_N}(\ref{item:lem_possibility_J_is_D_N_case_2}).
On the other hand, for the latter possibility, the roots $\alpha_{N-1} \in \Pi^{J,[y] \cap J}$ and $\alpha_t + \sum_{j=1}^{N-2} \alpha_j \in \Pi^{[y]} \smallsetminus \Pi^{J,[y] \cap J}$ are not orthogonal, contradicting the above-mentioned fact that $\Pi^{J,[y] \cap J}$ is the union of some irreducible components of $\Pi^{[y]}$.
Hence we have a contradiction for any of the two possibilities.

Finally, we consider the remaining case that $J$ is of type $D_4$.
By the property $N = 4 \geq |I \cap J| + 2$ and $A_{>1}$-freeness of $I$, it follows that $I \cap J$ consists of at most two irreducible components of type $A_1$.
On the other hand, by the shape of $J$, the fixed point set of the nontrivial automorphism $\tau$ of $J$ is of type $A_1$ or $A_2$.
Since $I \cap J$ is contained in the fixed point set of $\tau$ as mentioned above, it follows that $|I \cap J| = 1$.
If $I \cap J = \{r_1\}$, then we have $\Pi^{J,I \cap J} = \{\alpha_3,\alpha_4,\beta\}$ where $\beta = \alpha_1 + 2\alpha_2 + \alpha_3 + \alpha_4$ (see Table \ref{tab:positive_roots_D_n}), and every element of $\Pi^{J,I \cap J}$ forms an irreducible component of $\Pi^{J,I \cap J}$.
However, now the property $w \cdot \Pi_J = \Pi_J$ implies that $w$ fixes $\alpha_2$ and permutes the three simple roots $\alpha_1$, $\alpha_3$ and $\alpha_4$, therefore $w \cdot \beta = \beta$, contradicting the fact that $\langle w \rangle$ acts transitively on the set of the irreducible components of $\Pi^{J,I \cap J}$ (see Lemma \ref{lem:property_w}).
By symmetry, the same result holds when $I \cap J = \{r_3\}$ or $\{r_4\}$.
Hence we have $I \cap J = \{r_2\}$.
Take a standard decomposition of $w$ with respect to $J$ (see Proposition \ref{prop:standard_decomposition_existence}).
Then no wide transformation can appear due to the shape of $J$ and the position of $I \cap J$ in $J$ (indeed, we cannot obtain a subset of finite type by adding to $J$ an element of $S$ adjacent to $I \cap J$).
This implies that the decomposition of $w$ consists of narrow transformations only, therefore $w$ fixes $\Pi_J$ pointwise, contradicting the fact that $\tau$ is a nontrivial automorphism.

Summarizing, we have derived in any case a contradiction from the assumption that $\tau$ is a nontrivial automorphism.
Hence it follows that $\tau$ is the identity mapping, therefore our claim has been proven in the case $\Pi^{J,I \cap J} \not\subseteq \Phi_{I^{\perp}}$.

\subsection{The second case $\Pi^{J,I \cap J} \subseteq \Phi_{I^{\perp}}$}
\label{sec:proof_special_second_case}

In this subsection, we consider the remaining case that $\Pi^{J,I \cap J} \subseteq \Phi_{I^{\perp}}$.
In this case, we have $\Pi_{I^{\perp}} \subseteq \Pi^I$, therefore $\Pi^{J,I \cap J} = \Pi_{J \smallsetminus I}$.
Let $L$ be an irreducible component of $J \smallsetminus I$.
Then $L$ is of finite type.
The aim of the following argument is to show that $w$ fixes $\Pi_L$ pointwise; indeed, if this is satisfied, then we have $\Pi^{J,I \cap J} = \Pi_{J \smallsetminus I} = \Pi_L$ since $\langle w \rangle$ acts transitively on the set of irreducible components of $\Pi^{J,I \cap J}$ (see Lemma \ref{lem:property_w}), therefore $w$ fixes $\Pi^{J,I \cap J}$ pointwise, as desired.
Note that $w \cdot \Pi_L \subseteq \Pi_{J \smallsetminus I}$, since now $w$ leaves the set $\Pi^{J,I \cap J} = \Pi_{J \smallsetminus I}$ invariant.

\subsubsection{Possibilities of semi-standard decompositions}
\label{sec:proof_special_second_case_transformations}

Here we investigate the possibilities of narrow and wide transformations in a semi-standard decomposition of the element $w$, in a somewhat wider context.
Let $\mathcal{D} = \omega_{\ell(\mathcal{D})-1} \cdots \omega_1\omega_0$ be a semi-standard decomposition of an element $u$ of $W$, with the property that $[y^{(0)}]$ is isomorphic to $I$, $J^{(0)}$ is irreducible and of finite type, and $J^{(0)}$ is apart from $[y^{(0)}]$.
Note that any semi-standard decomposition of the element $w \in Y_I$ with respect to the set $L$ defined above satisfies the condition.
Note also that $\mathcal{D}^{-1} := (\omega_0)^{-1}(\omega_1)^{-1} \cdots (\omega_{\ell(\mathcal{D})-1})^{-1}$ is also a semi-standard decomposition of $u^{-1}$, and $(\omega_i)^{-1}$ is a narrow (respectively, wide) transformation if and only if $\omega_i$ is a narrow (respectively, wide) transformation.

The proof of the next lemma uses a concrete description of root systems of all finite irreducible Coxeter groups except types $A$ and $I_2(m)$.
Table \ref{tab:positive_roots_B_n} shows the list for type $B_n$, where the notational conventions are similar to the case of type $D_n$ (Table \ref{tab:positive_roots_D_n}).
For the list for type $F_4$ (Table \ref{tab:positive_roots_F_4}), the list includes only one of the two conjugacy classes of positive roots (denoted by $\gamma_i^{(1)}$), and the other positive roots (denoted by $\gamma_i^{(2)}$) are obtained by using the symmetry $r_1 \leftrightarrow r_4$, $r_2 \leftrightarrow r_3$.
In the list, $[c_1,c_2,c_3,c_4]$ signifies a positive root $c_1 \alpha_1 + c_2 \alpha_2 + c_3\alpha_3 + c_4\alpha_4$, and the description in the columns for actions of generators is similar to the case of type $E_8$ (Tables \ref{tab:positive_roots_E_8_1}--\ref{tab:positive_roots_E_8_6}).
The list for type $H_4$ is divided into two parts (Tables \ref{tab:positive_roots_H_4_1} and \ref{tab:positive_roots_H_4_2}).
In the list, $[c_1,c_2,c_3,c_4]$ signifies a positive root $c_1 \alpha_1 + c_2 \alpha_2 + c_3\alpha_3 + c_4\alpha_4$, where we put $c = 2\cos(\pi/5)$ for simplicity and therefore $c^2 = c+1$.
The other description is in a similar manner as the case of type $E_8$, and the marks \lq\lq $H_3$'' indicate the positive roots of the parabolic subgroup of type $H_3$ generated by $\{r_1,r_2,r_3\}$.
\begin{table}[hbt]
\centering
\caption{List of positive roots for Coxeter group of type $B_n$}
\label{tab:positive_roots_B_n}
\begin{tabular}{|c|c|} \hline
roots & actions of generators \\ \hline
$\gamma^{(1)}_{i,j} := \sum_{h=i}^{j} \alpha_h$ & $r_{i-1} \cdot \gamma^{(1)}_{i,j} = \gamma^{(1)}_{i-1,j}$ ($i \geq 2$) \\
($1 \leq i \leq j \leq n-1$) & $r_i \cdot \gamma^{(1)}_{i,j} = \gamma^{(1)}_{i+1,j}$ ($i \leq j-1$) \\
($\gamma^{(1)}_{i,i} = \alpha_i$) & $r_j \cdot \gamma^{(1)}_{i,j} = \gamma^{(1)}_{i,j-1}$ ($i \leq j-1$) \\
& $r_{j+1} \cdot \gamma^{(1)}_{i,j} = \gamma^{(1)}_{i,j+1}$ ($j \leq n-2$) \\
& $r_n \cdot \gamma^{(1)}_{i,n-1} = \gamma^{(2)}_{i,n}$ \\ \hline
$\gamma^{(2)}_{i,j} := \sum_{h=i}^{j-1} \alpha_h + \sum_{h=j}^{n-1} 2\alpha_h + \sqrt{2}\alpha_n$ & $r_{i-1} \cdot \gamma^{(2)}_{i,j} = \gamma^{(2)}_{i-1,j}$ ($i \geq 2$) \\
($1 \leq i < j \leq n$) & $r_i \cdot \gamma^{(2)}_{i,j} = \gamma^{(2)}_{i+1,j}$ ($i \leq j-2$) \\
& $r_{j-1} \cdot \gamma^{(2)}_{i,j} = \gamma^{(2)}_{i,j-1}$ ($i \leq j-2$) \\
& $r_j \cdot \gamma^{(2)}_{i,j} = \gamma^{(2)}_{i,j+1}$ ($j \leq n-1$) \\
& $r_n \cdot \gamma^{(2)}_{i,n} = \gamma^{(1)}_{i,n-1}$ \\ \hline
$\gamma^{(3)}_i := \sum_{h=i}^{n-1} \sqrt{2}\alpha_h + \alpha_n$ & $r_{i-1} \cdot \gamma^{(3)}_i = \gamma^{(3)}_{i-1}$ ($i \geq 2$) \\
($1 \leq i \leq n$) & $r_i \cdot \gamma^{(3)}_i = \gamma^{(3)}_{i+1}$ ($i \leq n-1$) \\
($\gamma^{(3)}_n = \alpha_n$) & \\ \hline
\end{tabular}
\end{table}
\begin{table}[hbt]
\centering
\caption{List of positive roots for Coxeter group of type $F_4$}
\label{tab:positive_roots_F_4}
The data of the remaining positive roots $\gamma^{(2)}_i$ are obtained by replacing $[c_1,c_2,c_3,c_4]$ with $[c_4,c_3,c_2,c_1]$ and replacing each $r_j$ with $r_{5-j}$.\\
\begin{tabular}{|c||c|c|c|c|c|c|} \cline{1-7}
height & $i$ & root $\gamma^{(1)}_i$ & \multicolumn{4}{|c|}{$k$; $r_j \cdot \gamma^{(1)}_i = \gamma^{(1)}_k$} \\ \cline{4-7}
& & & $r_1$ & $r_2$ & $r_3$ & $r_4$ \\ \cline{1-7}
$1$ & $1$ & $[1,0,0,0]$ & --- & $3$ & & \\ \cline{2-7}
& $2$ & $[0,1,0,0]$ & $3$ & --- & $4$ & \\ \cline{1-7}
$2$ & $3$ & $[1,1,0,0]$ & $2$ & $1$ & $5$ & \\ \cline{2-7}
& $4$ & $[0,1,\sqrt{2},0]$ & $5$ & & $2$ & $6$ \\ \cline{1-7}
$3$ & $5$ & $[1,1,\sqrt{2},0]$ & $4$ & $7$ & $3$ & $8$ \\ \cline{2-7}
& $6$ & $[0,1,\sqrt{2},\sqrt{2}]$ & $8$ & & & $4$ \\ \cline{1-7}
$4$ & $7$ & $[1,2,\sqrt{2},0]$ & & $5$ & & $9$ \\ \cline{2-7}
& $8$ & $[1,1,\sqrt{2},\sqrt{2}]$ & $6$ & $9$ & & $5$ \\ \cline{1-7}
$5$ & $9$ & $[1,2,\sqrt{2},\sqrt{2}]$ & & $8$ & $10$ & $7$ \\ \cline{1-7}
$6$ & $10$ & $[1,2,2\sqrt{2},\sqrt{2}]$ & & $11$ & $9$ & \\ \cline{1-7}
$7$ & $11$ & $[1,3,2\sqrt{2},\sqrt{2}]$ & $12$ & $10$ & & \\ \cline{1-7}
$8$ & $12$ & $[2,3,2\sqrt{2},\sqrt{2}]$ & $11$ & & & \\ \cline{1-7}
\end{tabular}
\end{table}
\begin{table}[hbt]
\centering
\caption{List of positive roots for Coxeter group of type $H_4$ (part $1$), where $c = 2\cos(\pi/5)$, $c^2 = c + 1$}
\label{tab:positive_roots_H_4_1}
\begin{tabular}{|c||c|c|c|c|c|c|c} \cline{1-7}
height & $i$ & root $\gamma_i$ & \multicolumn{4}{|c|}{$k$; $r_j \cdot \gamma_i = \gamma_k$} \\ \cline{4-7}
& & & $r_1$ & $r_2$ & $r_3$ & $r_4$ \\ \cline{1-7}
$1$ & $1$ & $[1,0,0,0]$ & --- & $5$ & & & $H_3$ \\ \cline{2-7}
& $2$ & $[0,1,0,0]$ & $6$ & --- & $7$ & & $H_3$ \\ \cline{2-7}
& $3$ & $[0,0,1,0]$ & & $7$ & --- & $8$ & $H_3$ \\ \cline{2-7}
& $4$ & $[0,0,0,1]$ & & & $8$ & --- \\ \cline{1-7}
$2$ & $5$ & $[1,c,0,0]$ & $9$ & $1$ & $10$ & & $H_3$ \\ \cline{2-7}
& $6$ & $[c,1,0,0]$ & $2$ & $9$ & $11$ & & $H_3$ \\ \cline{2-7}
& $7$ & $[0,1,1,0]$ & $11$ & $3$ & $2$ & $12$ & $H_3$ \\ \cline{2-7}
& $8$ & $[0,0,1,1]$ & & $12$ & $4$ & $3$ \\ \cline{1-7}
$3$ & $9$ & $[c,c,0,0]$ & $5$ & $6$ & $13$ & & $H_3$ \\ \cline{2-7}
& $10$ & $[1,c,c,0]$ & $13$ & & $5$ & $14$ & $H_3$ \\ \cline{2-7}
& $11$ & $[c,1,1,0]$ & $7$ & $15$ & $6$ & $16$ & $H_3$ \\ \cline{2-7}
& $12$ & $[0,1,1,1]$ & $16$ & $12$ & & $7$ \\ \cline{1-7}
$4$ & $13$ & $[c,c,c,0]$ & $10$ & $17$ & $9$ & $18$ & $H_3$ \\ \cline{2-7}
& $14$ & $[1,c,c,c]$ & $18$ & & & $10$ \\ \cline{2-7}
& $15$ & $[c,c+1,1,0]$ & $19$ & $11$ & $17$ & $20$ & $H_3$ \\ \cline{2-7}
& $16$ & $[c,1,1,1]$ & $12$ & $20$ & & $11$ \\ \cline{1-7}
$5$ & $17$ & $[c,c+1,c,0]$ & $21$ & $13$ & $15$ & $22$ & $H_3$ \\ \cline{2-7}
& $18$ & $[c,c,c,c]$ & $14$ & $22$ & & $13$ \\ \cline{2-7}
& $19$ & $[c+1,c+1,1,0]$ & $15$ & & $21$ & $23$ & $H_3$ \\ \cline{2-7}
& $20$ & $[c,c+1,1,1]$ & $23$ & $16$ & $24$ & $15$ \\ \cline{1-7}
$6$ & $21$ & $[c+1,c+1,c,0]$ & $17$ & $25$ & $19$ & $26$ & $H_3$ \\ \cline{2-7}
& $22$ & $[c,c+1,c,c]$ & $26$ & $18$ & $27$ & $17$ \\ \cline{2-7}
& $23$ & $[c+1,c+1,1,1]$ & $20$ & & $28$ & $19$ \\ \cline{2-7}
& $24$ & $[c,c+1,c+1,1]$ & $28$ & & $20$ & $27$ \\ \cline{1-7}
\end{tabular}
\end{table}
\begin{table}[p]
\centering
\caption{List of positive roots for Coxeter group of type $H_4$ (part $2$), where $c = 2\cos(\pi/5)$, $c^2 = c + 1$}
\label{tab:positive_roots_H_4_2}
\begin{tabular}{|c||c|c|c|c|c|c|c} \cline{1-7}
height & $i$ & root $\gamma_i$ & \multicolumn{4}{|c|}{$k$; $r_j \cdot \gamma_i = \gamma_k$} \\ \cline{4-7}
& & & $r_1$ & $r_2$ & $r_3$ & $r_4$ \\ \cline{1-7}
$7$ & $25$ & $[c+1,2c,c,0]$ & & $21$ & & $29$ & $H_3$ \\ \cline{2-7}
& $26$ & $[c+1,c+1,c,c]$ & $22$ & $29$ & $30$ & $21$ \\ \cline{2-7}
& $27$ & $[c,c+1,c+1,c]$ & $30$ & & $22$ & $24$ \\ \cline{2-7}
& $28$ & $[c+1,c+1,c+1,1]$ & $24$ & $31$ & $23$ & $30$ \\ \cline{1-7}
$8$ & $29$ & $[c+1,2c,c,c]$ & & $26$ & $32$ & $25$ \\ \cline{2-7}
& $30$ & $[c+1,c+1,c+1,c]$ & $27$ & $33$ & $26$ & $28$ \\ \cline{2-7}
& $31$ & $[c+1,2c+1,c+1,1]$ & $34$ & $28$ & & $33$ \\ \cline{1-7}
$9$ & $32$ & $[c+1,2c,2c,c]$ & & $35$ & $29$ & \\ \cline{2-7}
& $33$ & $[c+1,2c+1,c+1,c]$ & $36$ & $30$ & $35$ & $31$ \\ \cline{2-7}
& $34$ & $[2c+1,2c+1,c+1,1]$ & $31$ & $37$ & & $36$ \\ \cline{1-7}
$10$ & $35$ & $[c+1,2c+1,2c,c]$ & $38$ & $32$ & $33$ & \\ \cline{2-7}
& $36$ & $[2c+1,2c+1,c+1,c]$ & $33$ & $39$ & $38$ & $34$ \\ \cline{2-7}
& $37$ & $[2c+1,2c+2,c+1,1]$ & & $34$ & $40$ & $39$ \\ \cline{1-7}
$11$ & $38$ & $[2c+1,2c+1,2c,c]$ & $35$ & $41$ & $36$ & \\ \cline{2-7}
& $39$ & $[2c+1,2c+2,c+1,c]$ & & $36$ & $42$ & $37$ \\ \cline{2-7}
& $40$ & $[2c+1,2c+2,c+2,1]$ & & & $37$ & $43$ \\ \cline{1-7}
$12$ & $41$ & $[2c+1,3c+1,2c,c]$ & $44$ & $38$ & $45$ & \\ \cline{2-7}
& $42$ & $[2c+1,2c+2,2c+1,c]$ & & $45$ & $39$ & $46$ \\ \cline{2-7}
& $43$ & $[2c+1,2c+2,c+2,c+1]$ & & & $46$ & $40$ \\ \cline{1-7}
$13$ & $44$ & $[2c+2,3c+1,2c,c]$ & $41$ & & $47$ & \\ \cline{2-7}
& $45$ & $[2c+1,3c+1,2c+1,c]$ & $47$ & $42$ & $41$ & $48$ \\ \cline{2-7}
& $46$ & $[2c+1,2c+2,2c+1,c+1]$ & & $48$ & $43$ & $42$ \\ \cline{1-7}
$14$ & $47$ & $[2c+2,3c+1,2c+1,c]$ & $45$ & $49$ & $44$ & $50$ \\ \cline{2-7}
& $48$ & $[2c+1,3c+1,2c+1,c+1]$ & $50$ & $46$ & & $45$ \\ \cline{1-7}
$15$ & $49$ & $[2c+2,3c+2,2c+1,c]$ & $51$ & $47$ & & $52$ \\ \cline{2-7}
& $50$ & $[2c+2,3c+1,2c+1,c+1]$ & $48$ & $52$ & & $47$ \\ \cline{1-7}
$16$ & $51$ & $[3c+1,3c+2,2c+1,c]$ & $49$ & & & $53$ \\ \cline{2-7}
& $52$ & $[2c+2,3c+2,2c+1,c+1]$ & $53$ & $50$ & $54$ & $49$ \\ \cline{1-7}
$17$ & $53$ & $[3c+1,3c+2,2c+1,c+1]$ & $52$ & & $55$ & $51$ \\ \cline{2-7}
& $54$ & $[2c+2,3c+2,2c+2,c+1]$ & $55$ & & $52$ & \\ \cline{1-7}
$18$ & $55$ & $[3c+1,3c+2,2c+2,c+1]$ & $54$ & $56$ & $53$ & \\ \cline{1-7}
$19$ & $56$ & $[3c+1,3c+3,2c+2,c+1]$ & $57$ & $55$ & & \\ \cline{1-7}
$20$ & $57$ & $[3c+2,3c+3,2c+2,c+1]$ & $56$ & $58$ & & \\ \cline{1-7}
$21$ & $58$ & $[3c+2,4c+2,2c+2,c+1]$ & & $57$ & $59$ & \\ \cline{1-7}
$22$ & $59$ & $[3c+2,4c+2,3c+1,c+1]$ & & & $58$ & $60$ \\ \cline{1-7}
$23$ & $60$ & $[3c+2,4c+2,3c+1,2c]$ & & & & $59$ \\ \cline{1-7}
\end{tabular}
\end{table}

Then, for the wide transformations in $\mathcal{D}$, we have the following:
\begin{lem}
\label{lem:proof_special_second_case_transformations_wide}
In this setting, if $\omega_i$ is a wide transformation, then there exist only the following two possibilities, where $K^{(i)} = \{r_1,r_2,\dots,r_N\}$ is the standard labelling of $K^{(i)}$ given in Section \ref{sec:longestelement}:
\begin{enumerate}
\item $K^{(i)}$ is of type $A_N$ with $N \geq 3$, $t^{(i)} = r_2$, $[y^{(i)}] \cap K^{(i)} = \{r_1\}$ and $J^{(i)} = \{r_3,\dots,r_N\}$; now the action of $\omega_i$ maps $r_1$ to $r_N$ and $(r_3,r_4,\dots,r_N)$ to $(r_1,r_2,\dots,r_{N-2})$;
\item $K^{(i)}$ is of type $E_7$, $t^{(i)} = r_6$, $[y^{(i)}] \cap K^{(i)} = \{r_1,r_2,r_3,r_4,r_5\}$ and $J^{(i)} = \{r_7\}$; now the action of $\omega_i$ maps $(r_1,r_2,r_3,r_4,r_5)$ to $(r_1,r_5,r_3,r_4,r_2)$ and $r_7$ to $r_7$.
\end{enumerate}
Hence, if $\mathcal{D}$ involves a wide transformation, then $J^{(0)}$ is of type $A_{N'}$ with $1 \leq N' < \infty$.
\end{lem}
\begin{proof}
The latter part of the claim follows from the former part and the fact that the sets $J^{(i)}$ for $0 \leq i \leq \ell(\mathcal{D})$ are all isomorphic to each other.
For the former part, note that $J^{(i)}$ is an irreducible subset of $K^{(i)}$ which is not adjacent to $[y^{(i)}]$ (by the above condition that $J^{(0)}$ is apart from $[y^{(0)}]$), $t^{(i)}$ is adjacent to $[y^{(i)}] \cap K^{(i)}$, and $\omega_i$ cannot fix the set $\Pi_{K^{(i)} \smallsetminus \{t\}}$ pointwise (see Lemma \ref{lem:another_decomposition_Y_no_loop}).
Moreover, since $I$ is $A_{>1}$-free, $[y^{(i)}]$ is also $A_{>1}$-free.
By these properties, a case-by-case argument shows that the possibilities of $K^{(i)}$, $[y^{(i)}]$ and $t^{(i)}$ are as enumerated in Table \ref{tab:lem_proof_special_second_case_transformations_wide} up to symmetry (note that $J^{(i)} = K^{(i)} \smallsetminus ([y^{(i)}] \cup \{t^{(i)}\})$).
Now, for each case in Table \ref{tab:lem_proof_special_second_case_transformations_wide} except the two cases specified in the statement, it follows by using the tables for the root systems of finite irreducible Coxeter groups that there exists a root $\beta \in (\Phi_{K^{(i)}}^{\perp [y^{(i)}] \cap K^{(i)}})^+$ that has non-zero coefficient of $\alpha_{t^{(i)}}$, as listed in Table \ref{tab:lem_proof_special_second_case_transformations_wide} (where the notations for the roots $\beta$ are as in the tables).
This implies that $\omega_i \cdot \beta \in \Phi^-$.
Moreover, the definition of $K^{(i)}$ implies that the set $[y^{(i)}] \smallsetminus K^{(i)}$ is apart from $K^{(i)}$, therefore $\beta \in \Phi^{\perp [y^{(i)}]}$ and $\Phi^{\perp [y^{(i)}]}[\omega_i] \neq \emptyset$.
However, this contradicts the property $\omega_i \in Y_{y^{(i+1)},y^{(i)}}$.
Hence one of the two conditions specified in the statement should be satisfied, concluding the proof of Lemma \ref{lem:proof_special_second_case_transformations_wide}.
\end{proof}
\begin{table}[hbt]
\centering
\caption{List for the proof of Lemma \ref{lem:proof_special_second_case_transformations_wide}}
\label{tab:lem_proof_special_second_case_transformations_wide}
\begin{tabular}{|c|c|c|c|} \hline
type of $K^{(i)}$ & $[y^{(i)}] \cap K^{(i)}$ & $t^{(i)}$ & $\beta$ \\ \hline
$A_N$ ($N \geq 3$) & $\{r_1\}$ & $r_2$ & --- \\ \hline
$B_N$ ($N \geq 4$) & $\{r_{k+1},\dots,r_N\}$ ($3 \leq k \leq N-1$) & $r_k$ & $\gamma^{(3)}_{k}$ \\ \hline
$D_N$ & $\{r_{N-1},r_N\}$ & $r_{N-2}$ & $\gamma^{(4)}_{1,2}$ \\ \cline{2-3}
& $\{r_{k+1},\dots,r_{N-1},r_N\}$ ($2 \leq k \leq N-4$) & $r_k$ & \\ \hline
$E_N$ ($6 \leq N \leq 8$) & $\{r_1\}$ & $r_3$ & $\gamma_{44}$ \\ \hline
$E_7$ & $\{r_1,r_2,r_3,r_4,r_5\}$ & $r_6$ & --- \\ \cline{2-4}
& $\{r_7\}$ & $r_6$ & $\gamma_{61}$ \\ \hline
$E_8$ & $\{r_1,r_2,r_3,r_4,r_5\}$ & $r_6$ & $\gamma_{119}$ \\ \cline{2-3}
& $\{r_1,r_2,r_3,r_4,r_5,r_6\}$ & $r_7$ & \\ \cline{2-4}
& $\{r_8\}$ & $r_7$ & $\gamma_{74}$ \\ \hline
$F_4$ & $\{r_1\}$ & $r_2$ & $\gamma^{(1)}_7$ \\ \hline
$H_4$ & $\{r_1\}$ & $r_2$ & $\gamma_{40}$ \\ \cline{2-3}
& $\{r_1,r_2\}$ & $r_3$ & \\ \cline{2-4}
& $\{r_4\}$ & $r_3$ & $\gamma_{32}$ \\ \hline
\end{tabular}
\end{table}

On the other hand, for the narrow transformations in $\mathcal{D}$, we have the following:
\begin{lem}
\label{lem:proof_special_second_case_transformations_narrow}
In this setting, suppose that $\omega_i$ is a narrow transformation, $[y^{(i+1)}] \neq [y^{(i)}]$, and $K^{(i)} \cap [y^{(i)}] = K^{(i)} \smallsetminus \{t^{(i)}\}$ has an irreducible component of type $A_1$.
Then $K^{(i)}$ is of type $A_2$ or of type $I_2(m)$ with $m$ an odd number.
\end{lem}
\begin{proof}
First, by the condition $[y^{(i+1)}] \neq [y^{(i)}]$ and the definition of $\omega_i$, the action of the longest element of $W_{K^{(i)}}$ induces a nontrivial automorphism of $K^{(i)}$ which does not fix the element $t^{(i)}$.
This property restricts the possibilities of $K^{(i)}$ to one of the followings (where we use the standard labelling of $K^{(i)}$): $K^{(i)} = \{r_1,\dots,r_N\}$ is of type $A_N$ and $t^{(i)} \neq r_{(N+1)/2}$; $K^{(i)} = \{r_1,\dots,r_N\}$ is of type $D_N$ with $N$ odd and $t^{(i)} \in \{r_{N-1},r_N\}$; $K^{(i)} = \{r_1,\dots,r_6\}$ is of type $E_6$ and $t^{(i)} \not\in \{r_2,r_4\}$; or $K^{(i)}$ is of type $I_2(m)$ with $m$ odd.
Secondly, by considering the $A_{>1}$-freeness of $I$ (hence of $[y^{(i)}]$), the possibilities are further restricted to the followings: $K^{(i)}$ is of type $A_2$; $K^{(i)}$ is of type $E_6$ and $t^{(i)} \in \{r_1,r_6\}$; and $K^{(i)}$ is of type $I_2(m)$ with $m$ odd.
Moreover, by the hypothesis that $K^{(i)} \cap [y^{(i)}]$ has an irreducible component of type $A_1$, the above possibility of type $E_6$ is denied.
Hence the claim holds.
\end{proof}

\subsubsection{Proof of the claim}
\label{sec:finitepart_secondcase_N_2}

From now, we prove our claim that $w$ fixes the set $\Pi_L$ pointwise.
First, we have $w \cdot \Pi_L \subseteq \Pi_{J \smallsetminus I}$ as mentioned above, therefore Proposition \ref{prop:standard_decomposition_existence} implies that there exists a standard decomposition of $w$ with respect to $L$.
Moreover, $L$ is apart from $I = [x_I]$, since $\Pi_L$ is an irreducible component of $\Pi^I$.
Now if $L$ is not of type $A_N$ with $1 \leq N < \infty$, then Lemma \ref{lem:proof_special_second_case_transformations_wide} implies that the standard decomposition of $w$ involves no wide transformations, therefore $w$ fixes $\Pi_L$ pointwise, as desired (note that any narrow transformation $\omega_i$ fixes $\Pi_{J^{(i)}}$ pointwise by the definition).
Hence, from now, we consider the case that $L$ is of type $A_N$ with $1 \leq N < \infty$.

First, we present some definitions:
\begin{defn}
\label{defn:admissible_for_N_large}
Suppose that $2 \leq N < \infty$.
Let $\mathcal{D} = \omega_{\ell(\mathcal{D})-1} \cdots \omega_1\omega_0$ be a semi-standard decomposition of an element of $W$.
We say that a sequence $s_1,s_2,\dots,s_{\mu}$ of distinct elements of $S$ is \emph{admissible of type $A_N$} with respect to $\mathcal{D}$, if $J^{(0)}$ is of type $A_N$, $\mu \equiv N \pmod{2}$, and the following conditions are satisfied, where we put $M := \{s_1,s_2\dots,s_{\mu}\}$ (see Figure \ref{fig:finitepart_secondcase_N_2_definition}):
\begin{enumerate}
\item \label{item:admissible_N_large_J_irreducible_component}
$\Pi_{J^{(0)}}$ is an irreducible component of $\Pi^{[y^{(0)}]}$.
\item \label{item:admissible_N_large_M_line}
$m(s_j,s_{j+1}) = 3$ for every $1 \leq j \leq \mu-1$.
\item \label{item:admissible_N_large_J_in_M}
For each $0 \leq h \leq \ell(\mathcal{D})$, there exists an odd number $\lambda(h)$ with $1 \leq \lambda(h) \leq \mu - N + 1$ satisfying the following conditions, where we put $\rho(h) := \lambda(h) + N - 1$:
\begin{displaymath}
J^{(h)} = \{s_j \mid \lambda(h) \leq j \leq \rho(h)\} \enspace,
\end{displaymath}
\begin{displaymath}
\begin{split}
[y^{(h)}] \cap M &= \{s_j \mid 1 \leq j \leq \lambda(h) - 2 \mbox{ and } j \equiv 1 \pmod{2}\} \\
&\cup \{s_j \mid \rho(h) + 2 \leq j \leq \mu \mbox{ and } j \equiv \mu \pmod{2}\} \enspace.
\end{split}
\end{displaymath}
\item \label{item:admissible_N_large_y_isolated}
For each $0 \leq h \leq \ell(\mathcal{D})$, every element of $[y^{(h)}] \cap M$ forms an irreducible component of $[y^{(h)}]$ of type $A_1$.
\item \label{item:admissible_N_large_narrow_transformation}
For each $0 \leq h \leq \ell(\mathcal{D})-1$, if $\omega_h$ is a narrow transformation, then one of the following two conditions is satisfied:
\begin{itemize}
\item $K^{(h)}$ intersects with $[y^{(h)}] \cap M$, and $[y^{(h+1)}] = [y^{(h)}]$;
\item $K^{(h)}$ is apart from $[y^{(h)}] \cap M$ (hence $[y^{(h+1)}] \cap M = [y^{(h)}] \cap M$).
\end{itemize}
\item \label{item:admissible_N_large_wide_transformation}
For each $0 \leq h \leq \ell(\mathcal{D})-1$, if $\omega_h$ is a wide transformation, then one of the following two conditions is satisfied:
\begin{itemize}
\item $\lambda(h+1) = \lambda(h) - 2$, $K^{(h)} = J^{(h)} \cup \{s_{\lambda(h)-2},s_{\lambda(h)-1}\}$ is of type $A_{N+2}$, $t^{(h)} = s_{\lambda(h)-1}$, and the action of $\omega_h$ maps $s_{\lambda(h)+j} \in J^{(h)}$ ($0 \leq j \leq N-1$) to $s_{\lambda(h+1)+j}$ and maps $s_{\lambda(h)-2} \in [y^{(h)}]$ to $s_{\rho(h)+2}$;
\item $\lambda(h+1) = \lambda(h) + 2$, $K^{(h)} = J^{(h)} \cup \{s_{\rho(h)+1},s_{\rho(h)+2}\}$ is of type $A_{N+2}$, $t^{(h)} = s_{\rho(h)+1}$, and the action of $\omega_h$ maps $s_{\lambda(h)+j} \in J^{(h)}$ ($0 \leq j \leq N-1$) to $s_{\lambda(h+1)+j}$ and maps $s_{\rho(h)+2} \in [y^{(h)}]$ to $s_{\lambda(h)-2}$.
\end{itemize}
\end{enumerate}
Moreover, we say that such a sequence $s_1,s_2,\dots,s_{\mu}$ is \emph{tight} if $M = \bigcup_{h=0}^{\ell(\mathcal{D})} J^{(h)}$.
\end{defn}
\begin{figure}[hbt]
\centering
\begin{picture}(350,115)(-0,-115)
\multiput(10,-30)(30,0){5}{\circle*{8}}
\multiput(25,-30)(30,0){5}{\circle{8}}
\multiput(160,-30)(15,0){7}{\circle{8}}
\multiput(265,-30)(30,0){3}{\circle{8}}
\multiput(280,-30)(30,0){3}{\circle*{8}}
\multiput(14,-30)(15,0){22}{\line(1,0){7}}
\put(155,-35){\framebox(100,10){}}
\put(10,-10){\hbox to0pt{\hss$s_1$\hss}}
\put(10,-15){\vector(0,-1){8}}
\put(25,-10){\hbox to0pt{\hss$s_2$\hss}}
\put(25,-15){\vector(0,-1){8}}
\put(40,-10){\hbox to0pt{\hss$\cdots$\hss}}
\put(205,-20){\hbox to0pt{\hss$J^{(h)}$\hss}}
\put(130,-10){\hbox to0pt{\hss$s_{\lambda(h)-2}$\hss}}
\put(130,-15){\vector(0,-1){8}}
\put(160,-10){\hbox to0pt{\hss$s_{\lambda(h)}$\hss}}
\put(160,-15){\vector(0,-1){8}}
\put(250,-10){\hbox to0pt{\hss$s_{\rho(h)}$\hss}}
\put(250,-15){\vector(0,-1){8}}
\put(325,-10){\hbox to0pt{\hss$\cdots$\hss}}
\put(340,-10){\hbox to0pt{\hss$s_{\mu}$\hss}}
\put(340,-15){\vector(0,-1){8}}
\put(125,-40){$\underbrace{\hspace*{130pt}}$}
\put(260,-50){$K^{(h)}$}
\put(125,-80){$\overbrace{\hspace*{130pt}}$}
\put(135,-33){\vector(2,-1){110}}
\multiput(160,-38)(15,0){7}{\vector(-2,-3){30}}
\multiput(10,-90)(30,0){4}{\circle*{8}}
\multiput(25,-90)(30,0){4}{\circle{8}}
\multiput(130,-90)(15,0){7}{\circle{8}}
\multiput(235,-90)(30,0){4}{\circle{8}}
\multiput(250,-90)(30,0){4}{\circle*{8}}
\multiput(14,-90)(15,0){22}{\line(1,0){7}}
\put(125,-95){\framebox(100,10){}}
\put(10,-115){\hbox to0pt{\hss$s_1$\hss}}
\put(10,-105){\vector(0,1){8}}
\put(25,-115){\hbox to0pt{\hss$s_2$\hss}}
\put(25,-105){\vector(0,1){8}}
\put(40,-115){\hbox to0pt{\hss$\cdots$\hss}}
\put(175,-110){\hbox to0pt{\hss$J^{(h+1)}$\hss}}
\put(130,-115){\hbox to0pt{\hss$s_{\lambda(h+1)}$\hss}}
\put(130,-105){\vector(0,1){8}}
\put(220,-115){\hbox to0pt{\hss$s_{\rho(h+1)}$\hss}}
\put(220,-105){\vector(0,1){8}}
\put(325,-115){\hbox to0pt{\hss$\cdots$\hss}}
\put(340,-115){\hbox to0pt{\hss$s_{\mu}$\hss}}
\put(340,-105){\vector(0,1){8}}
\end{picture}
\caption{Admissible sequence when $N \geq 2$; here $N = 7$, black circles in the top and the bottom rows indicate elements of $[y^{(h)}] \cap M$ and $[y^{(h+1)}] \cap M$, respectively, and $\omega_h$ is a wide transformation with $t^{(h)} = s_{\lambda(h)-1}$}
\label{fig:finitepart_secondcase_N_2_definition}
\end{figure}
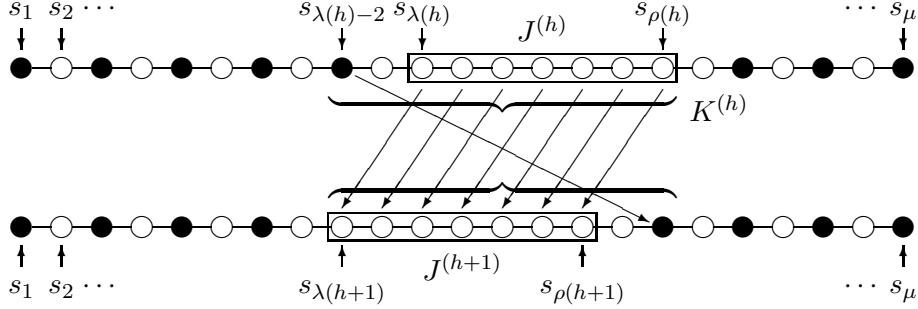
\begin{defn}
\label{defn:admissible_for_N_1}
Suppose that $N = 1$.
Let $\mathcal{D} = \omega_{\ell(\mathcal{D})-1} \cdots \omega_1\omega_0$ be a semi-standard decomposition of an element of $W$.
We say that a sequence $s_1,s_2,\dots,s_{\mu}$ of distinct elements of $S$ is \emph{admissible of type $A_1$} with respect to $\mathcal{D}$, if $J^{(0)}$ is of type $A_1$ and the following conditions are satisfied, where we put $M = \{s_1,s_2\dots,s_{\mu}\}$ (see Figure \ref{fig:finitepart_secondcase_N_1_definition}):
\begin{enumerate}
\item \label{item:admissible_N_1_J_irreducible_component}
$\Pi_{J^{(0)}}$ is an irreducible component of $\Pi^{[y^{(0)}]}$.
\item \label{item:admissible_N_1_J_in_M}
For each $0 \leq h \leq \ell(\mathcal{D})$, we have $J^{(h)} \subseteq M$ and $M \smallsetminus J^{(h)} \subseteq [y^{(h)}]$.
\item \label{item:admissible_N_1_y_isolated}
For each $0 \leq h \leq \ell(\mathcal{D})$, every element of $[y^{(h)}] \cap M$ forms an irreducible component of $[y^{(h)}]$ of type $A_1$.
\item \label{item:admissible_N_1_narrow_transformation}
For each $0 \leq h \leq \ell(\mathcal{D})-1$, if $\omega_h$ is a narrow transformation, then one of the following two conditions is satisfied:
\begin{itemize}
\item $K^{(h)}$ intersects with $[y^{(h)}] \cap M$, and $[y^{(h+1)}] = [y^{(h)}]$;
\item $K^{(h)}$ is apart from $[y^{(h)}] \cap M$, hence $[y^{(h+1)}] \cap M = [y^{(h)}] \cap M$.
\end{itemize}
\item \label{item:admissible_N_1_wide_transformation}
For each $0 \leq h \leq \ell(\mathcal{D})-1$, if $\omega_h$ is a wide transformation, then one of the following two conditions is satisfied:
\begin{itemize}
\item $J^{(h+1)} \neq J^{(h)}$, $K^{(h)}$ is of type $A_3$, $K^{(h)} \smallsetminus \{t^{(h)}\} = J^{(h)} \cup J^{(h+1)}$, $J^{(h+1)} \subseteq [y^{(h)}] \cap M$, and the action of $\omega_h$ exchanges the unique element of $J^{(h)}$ and the unique element of $J^{(h+1)}$;
\item $J^{(h+1)} = J^{(h)}$ and $[y^{(h+1)}] = [y^{(h)}]$.
\end{itemize}
\end{enumerate}
Moreover, we say that such a sequence $s_1,s_2,\dots,s_{\mu}$ is \emph{tight} if $M = \bigcup_{h=0}^{\ell(\mathcal{D})} J^{(h)}$.
\end{defn}
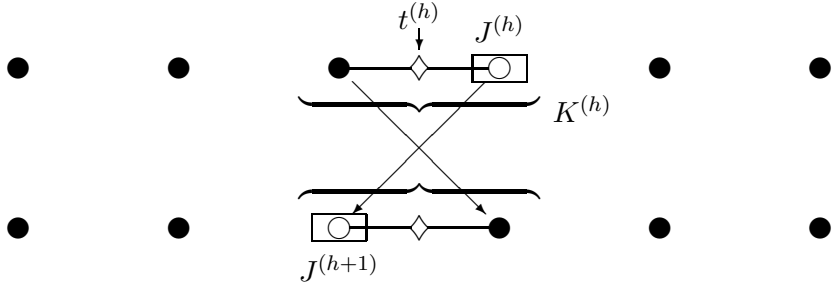
\begin{figure}[hbt]
\centering
\begin{picture}(350,110)(-0,-110)
\multiput(10,-30)(60,0){3}{\circle*{8}}
\put(190,-30){\circle{8}}
\multiput(250,-30)(60,0){2}{\circle*{8}}
\put(180,-35){\framebox(20,10){}}
\put(160,-33){\hbox to0pt{\hss$\diamondsuit$\hss}}
\put(134,-30){\line(1,0){23}}
\put(186,-30){\line(-1,0){23}}
\put(190,-20){\hbox to0pt{\hss$J^{(h)}$\hss}}
\put(160,-15){\hbox to0pt{\hss$t^{(h)}$\hss}}
\put(160,-15){\vector(0,-1){8}}
\put(115,-40){$\underbrace{\hspace*{90pt}}$}
\put(210,-50){$K^{(h)}$}
\put(115,-80){$\overbrace{\hspace*{90pt}}$}
\put(135,-35){\vector(1,-1){50}}
\put(185,-35){\vector(-1,-1){50}}
\multiput(10,-90)(60,0){2}{\circle*{8}}
\put(130,-90){\circle{8}}
\multiput(190,-90)(60,0){3}{\circle*{8}}
\put(120,-95){\framebox(20,10){}}
\put(160,-93){\hbox to0pt{\hss$\diamondsuit$\hss}}
\put(134,-90){\line(1,0){23}}
\put(186,-90){\line(-1,0){23}}
\put(130,-110){\hbox to0pt{\hss$J^{(h+1)}$\hss}}
\end{picture}
\caption{Admissible sequence when $N = 1$; here $\omega_h$ is a wide transformation of the first type in Definition \ref{defn:admissible_for_N_1}(\ref{item:admissible_N_1_wide_transformation}), the circles in each row signify elements of $M$, and the diamond signifies the element $t^{(h)}$}
\label{fig:finitepart_secondcase_N_1_definition}
\end{figure}

Note that, if a sequence $s_1,s_2,\dots,s_{\mu}$ is admissible of type $A_N$ with respect to a semi-standard decomposition $\mathcal{D} = \omega_{\ell(\mathcal{D})-1} \cdots \omega_1\omega_0$, then the subsequence of $s_1,s_2,\dots,s_{\mu}$ consisting of the elements of $\bigcup_{j=0}^{\ell(\mathcal{D})} J^{(j)}$ is admissible of type $A_N$ with respect to $\mathcal{D}$ and is tight (for the case $N \geq 2$, the property of wide transformations in Definition \ref{defn:admissible_for_N_large}(\ref{item:admissible_N_large_wide_transformation}) implies that $\bigcup_{j=0}^{\ell(\mathcal{D})} J^{(j)} = \{s_i \mid \lambda(k) \leq i \leq \rho(k')\}$ for some $k,k' \in \{0,1,\dots,\ell(\mathcal{D})\}$).
Moreover, the sequence $s_1$, $s_2,\dots,s_{\mu}$ is also admissible of type $A_N$ with respect to $\mathcal{D}^{-1}$.

The above definitions are relevant to our purpose in the following manner:
\begin{lem}
\label{lem:claim_holds_when_admissible}
Let $\mathcal{D} = \omega_{\ell(\mathcal{D})-1} \cdots \omega_1\omega_0$ be a semi-standard decomposition of $w$ with respect to $L$.
If there exists a sequence which is admissible of type $A_N$ with respect to $\mathcal{D}$, then $w$ fixes $\Pi_L$ pointwise.
\end{lem}
\begin{proof}
First, note that $y^{(\ell(\mathcal{D}))} = x_I = y^{(0)}$ since $w \in Y_I$, therefore $[y^{(\ell(\mathcal{D}))}] \cap M = [y^{(0)}] \cap M$ where $M$ is as defined in Definition \ref{defn:admissible_for_N_large} (when $N \geq 2$) or Definition \ref{defn:admissible_for_N_1} (when $N = 1$).
Now it follows from the properties in Definition \ref{defn:admissible_for_N_large}(\ref{item:admissible_N_large_J_in_M}) when $N \geq 2$, or Definition \ref{defn:admissible_for_N_1}(\ref{item:admissible_N_1_J_in_M}) when $N = 1$, that $J^{(\ell(\mathcal{D}))} = J^{(0)} = L$.
Hence $w$ fixes $\Pi_L$ pointwise when $N = 1$.
Moreover, when $N \geq 2$, the property in Definition \ref{defn:admissible_for_N_large}(\ref{item:admissible_N_large_wide_transformation}) implies that $\omega_h \ast s_{\lambda(h)+j} = s_{\lambda(h+1)+j}$ for every $0 \leq h \leq \ell(\mathcal{D})-1$ and $0 \leq j \leq N-1$.
Now by this property and the above-mentioned property $J^{(\ell(\mathcal{D}))} = J^{(0)}$, it follows that $w$ fixes the set $\Pi_{J^{(0)}} = \Pi_{L}$ pointwise.
Hence the proof is concluded.
\end{proof}

As mentioned above, a standard decomposition of $w$ with respect to $L$ exists.
Therefore, by virtue of Lemma \ref{lem:claim_holds_when_admissible}, it suffices to show that there exists a sequence which is admissible with respect to this standard decomposition.
More generally, we prove the following proposition (note that the above-mentioned standard decomposition of $w$ satisfies the assumption in this proposition):
\begin{prop}
\label{prop:admissible_sequence_exists}
Let $\mathcal{D} = \omega_{\ell(\mathcal{D})-1} \cdots \omega_1\omega_0$ be a semi-standard decomposition of an element.
Suppose that $J^{(0)}$ is of type $A_N$ with $1 \leq N < \infty$, and $\Pi_{J^{(0)}}$ is an irreducible component of $\Pi^{[y^{(0)}]}$.
Then there exists a sequence which is admissible of type $A_N$ with respect to $\mathcal{D}$.
\end{prop}
To prove Proposition \ref{prop:admissible_sequence_exists}, we give the following key lemma, which will be proven below:
\begin{lem}
\label{lem:admissible_sequence_extends}
Let $n \geq 0$.
Let $\mathcal{D} = \omega_n\omega_{n-1} \cdots \omega_1\omega_0$ be a semi-standard decomposition of an element, and put $\mathcal{D}' := \omega_{n-1} \cdots \omega_1\omega_0$, which is also a semi-standard decomposition of an element satisfying that $y^{(0)}(\mathcal{D}') = y^{(0)}(\mathcal{D})$ and $J^{(0)}(\mathcal{D}') = J^{(0)}(\mathcal{D})$.
Suppose that $s_1,\dots,s_{\mu}$ is a sequence which is admissible of type $A_N$ with respect to $\mathcal{D}'$.
For simplicity, put $y^{(j)} = y^{(j)}(\mathcal{D})$, $J^{(j)} = J^{(j)}(\mathcal{D})$, $t^{(j)} = t^{(j)}(\mathcal{D})$, and $K^{(j)} = K^{(j)}(\mathcal{D})$ for each index $j$.
\begin{enumerate}
\item \label{item:lem_admissible_sequence_extends_narrow}
If $\omega_n$ is a narrow transformation, then we have either $[y^{(n+1)}] = [y^{(n)}]$, or $K^{(n)}$ is apart from $[y^{(n)}] \cap \bigcup_{j=0}^{n} J^{(j)}$.
\item \label{item:lem_admissible_sequence_extends_wide_N_1_fix}
If $N = 1$, $\omega_n$ is a wide transformation and $J^{(n+1)} = J^{(n)}$, then we have $[y^{(n+1)}] = [y^{(n)}]$.
\item \label{item:lem_admissible_sequence_extends_wide_N_1}
If $N = 1$, $\omega_n$ is a wide transformation and $J^{(n+1)} \neq J^{(n)}$, then $K^{(n)}$ is of type $A_3$, $K^{(n)} \smallsetminus (J^{(n)} \cup \{t^{(n)}\}) \subseteq [y^{(n)}]$, and the action of $\omega_n$ exchanges the unique element of $J^{(n)}$ and the unique element of $K^{(n)} \smallsetminus (J^{(n)} \cup \{t^{(n)}\})$ (the latter belonging to $[y^{(n)}] \cap J^{(n+1)}$).
\item \label{item:lem_admissible_sequence_extends_wide_N_large}
If $N \geq 2$ and $\omega_n$ is a wide transformation, then $K^{(n)}$ is of type $A_{N+2}$, the unique element $s'$ of $K^{(n)} \smallsetminus (J^{(n)} \cup \{t^{(n)}\})$ belongs to $[y^{(n)}]$, and one of the following two conditions is satisfied:
\begin{enumerate}
\item \label{item:lem_admissible_sequence_extends_wide_N_large_left}
$t^{(n)}$ is adjacent to $s'$ and $s_{\lambda(n)}$, and the action of $\omega_n$ maps the elements $s_{\lambda(n)}$, $s_{\lambda(n)+1}$, $s_{\lambda(n)+2},\dots,s_{\rho(n)}$ and $s'$ to $s'$, $t^{(n)}$, $s_{\lambda(n)},\dots,s_{\rho(n)-2}$ and $s_{\rho(n)}$, respectively.
Moreover;
\begin{enumerate}
\item \label{item:lem_admissible_sequence_extends_wide_N_large_left_branch}
if $\lambda(n) \geq 3$ and $s_{\lambda(n)-2} \in \bigcup_{j=0}^{n} J^{(j)}$, then we have $s' = s_{\lambda(n)-2}$ and $t^{(n)} = s_{\lambda(n)-1}$;
\item \label{item:lem_admissible_sequence_extends_wide_N_large_left_terminal}
otherwise, we have $s' \not\in \bigcup_{j=0}^{n} J^{(j)}$.
\end{enumerate}
\item \label{item:lem_admissible_sequence_extends_wide_N_large_right}
$t^{(n)}$ is adjacent to $s'$ and $s_{\rho(n)}$, and the action of $\omega_n$ maps the elements $s_{\rho(n)}$, $s_{\rho(n)-1}$, $s_{\rho(n)-2},\dots,s_{\lambda(n)}$ and $s'$ to $s'$, $t^{(n)}$, $s_{\rho(n)},\dots,s_{\lambda(n)+2}$ and $s_{\lambda(n)}$, respectively.
Moreover;
\begin{enumerate}
\item \label{item:lem_admissible_sequence_extends_wide_N_large_right_branch}
if $\rho(n) \leq \mu - 2$ and $s_{\rho(n)+2} \in \bigcup_{j=0}^{n} J^{(j)}$, then we have $s' = s_{\rho(n)+2}$ and $t^{(n)} = s_{\rho(n)+1}$;
\item \label{item:lem_admissible_sequence_extends_wide_N_large_right_terminal}
otherwise, we have $s' \not\in \bigcup_{j=0}^{n} J^{(j)}$.
\end{enumerate}
\end{enumerate}
\end{enumerate}
\end{lem}
Then Proposition \ref{prop:admissible_sequence_exists} is deduced by applying Lemma \ref{lem:admissible_sequence_extends} and the next lemma to the semi-standard decompositions $\mathcal{D}_{\nu} := \omega_{\nu-1} \cdots \omega_1\omega_0$ ($0 \leq \nu \leq \ell(\mathcal{D})$) successively (note that, when $\nu = 0$, i.e., $\mathcal{D}_{\nu}$ is an empty expression, the sequence $s_1,\dots,s_N$, where $J^{(0)} = \{s_1,\dots,s_N\}$ is the standard labelling of type $A_N$, is admissible of type $A_N$ with respect to $\mathcal{D}_{\nu}$):
\begin{lem}
\label{lem:admissible_sequence_existence_from_lemma}
In the situation of Lemma \ref{lem:admissible_sequence_extends}, we define a sequence $\sigma$ of elements of $S$ in the following manner: For Cases \ref{item:lem_admissible_sequence_extends_narrow}, \ref{item:lem_admissible_sequence_extends_wide_N_1_fix}, \ref{item:lem_admissible_sequence_extends_wide_N_large_left_branch} and \ref{item:lem_admissible_sequence_extends_wide_N_large_right_branch}, let $\sigma$ be the sequence $s_1,\dots,s_{\mu}$; for Case \ref{item:lem_admissible_sequence_extends_wide_N_1}, let $s'$ be the unique element of $K^{(n)} \smallsetminus (J^{(n)} \cup \{t^{(n)}\}) = J^{(n+1)}$, and let $\sigma$ be the sequence $s_1,\dots,s_{\mu},s'$ when $s' \not\in \{s_1,\dots,s_{\mu}\}$ and the sequence $s_1,\dots,s_{\mu}$ when $s' \in \{s_1,\dots,s_{\mu}\}$; for Case \ref{item:lem_admissible_sequence_extends_wide_N_large_left_terminal}, let $\sigma$ be the sequence $s'$, $t^{(n)}$, $s_{\lambda(n)}$, $s_{\lambda(n)+1},\dots,s_{\rho'}$, where $\rho'$ denotes the largest index $1 \leq \rho' \leq \mu$ with $s_{\rho'} \in \bigcup_{j=0}^{n} J^{(j)}$; for the case \ref{item:lem_admissible_sequence_extends_wide_N_large_right_terminal}, let $\sigma$ be the sequence $s'$, $t^{(n)}$, $s_{\rho(n)}$, $s_{\rho(n)-1},\dots,s_{\lambda'}$, where $\lambda'$ denotes the smallest index $1 \leq \lambda' \leq \mu$ with $s_{\lambda'} \in \bigcup_{j=0}^{n} J^{(j)}$.
Then $\sigma$ is admissible of type $A_N$ with respect to $\mathcal{D} = \omega_n \cdots \omega_1\omega_0$.
\end{lem}

Now our remaining task is to prove Lemma \ref{lem:admissible_sequence_extends} and Lemma \ref{lem:admissible_sequence_existence_from_lemma}.
For the purpose, we present an auxiliary result:
\begin{lem}
\label{lem:admissible_sequence_non_adjacent_pairs}
Let $s_1,\dots,s_{\mu}$ be a sequence which is admissible of type $A_N$, where $N \geq 2$, with respect to a semi-standard decomposition $\mathcal{D}$ of an element of $W$.
Suppose that the sequence $s_1,\dots,s_{\mu}$ is tight.
If $1 \leq j_1 < j_2 \leq \mu$, $j_2 - j_1 \geq 2$, and either $j_1 \equiv 1 \pmod{2}$ or $j_2 \equiv \mu \pmod{2}$, then $s_{j_1}$ is not adjacent to $s_{j_2}$.
\end{lem}
\begin{proof}
By symmetry, we may assume without loss of generality that $j_1 \equiv 1 \pmod{2}$.
Put $\mathcal{D} = \omega_{n-1} \cdots \omega_1\omega_0$.
Since the sequence $s_1,\dots,s_\mu$ is tight, there exists an index $0 \leq h \leq n$ with $s_{j_2} \in J^{(h)}$.
Now the properties \ref{item:admissible_N_large_M_line} and \ref{item:admissible_N_large_J_in_M} in Definition \ref{defn:admissible_for_N_large} imply that $J^{(h)} = \{s_{\lambda(h)},s_{\lambda(h)+1},\dots,s_{\rho(h)}\}$ is the standard labelling of type $A_N$, therefore the claim holds if $s_{j_1} \in J^{(h)}$ (note that $j_2 - j_1 \geq 2$).
On the other hand, if $s_{j_1} \not\in J^{(h)}$, then the property \ref{item:admissible_N_large_J_in_M} in Definition \ref{defn:admissible_for_N_large} and the fact $j_1 < j_2$ imply that $j_1 < \lambda(h)$, therefore $s_{j_1} \in [y^{(h)}]$ since $j_1 \equiv 1 \pmod{2}$.
Hence the claim follows from the fact that $J^{(h)}$ is apart from $[y^{(h)}]$ (see the property \ref{item:admissible_N_large_J_irreducible_component} in Definition \ref{defn:admissible_for_N_large}).
\end{proof}

From now, we prove the pair of Lemma \ref{lem:admissible_sequence_extends} and Lemma \ref{lem:admissible_sequence_existence_from_lemma} by induction on $n \geq 0$.
First, we give a proof of Lemma \ref{lem:admissible_sequence_existence_from_lemma} for $n = n_0$ by assuming Lemma \ref{lem:admissible_sequence_extends} for $0 \leq n \leq n_0$.
Secondly, we will give a proof of Lemma \ref{lem:admissible_sequence_extends} for $n = n_0$ by assuming Lemma \ref{lem:admissible_sequence_extends} for $0 \leq n < n_0$ and Lemma \ref{lem:admissible_sequence_existence_from_lemma} for $0 \leq n < n_0$.
\begin{proof}
[Proof of Lemma \ref{lem:admissible_sequence_existence_from_lemma} (for $n = n_0$) from Lemma \ref{lem:admissible_sequence_extends} (for $n \leq n_0$).]
When $n_0 = 0$, the claim is obvious from the property of $\omega_{n_0}$ specified in Lemma \ref{lem:admissible_sequence_extends}.
From now, we suppose that $n_0 > 0$.
We may assume without loss of generality that the sequence $s_1,\dots,s_{\mu}$ (denoted here by $\sigma'$) which is admissible with respect to $\mathcal{D}'$ is tight, therefore we have $M' := \{s_1,\dots,s_{\mu}\} = \bigcup_{j=0}^{n_0} J^{(j)}$.
We divide the proof according to the possibility of $\omega_{n_0}$ listed in Lemma \ref{lem:admissible_sequence_extends}.
By symmetry, we may omit the argument for Case \ref{item:lem_admissible_sequence_extends_wide_N_large_right} without loss of generality.

In Case \ref{item:lem_admissible_sequence_extends_narrow}, since $M' = \bigcup_{j=0}^{n_0} J^{(j)}$ as above, $\omega_{n_0}$ satisfies the condition for $\sigma'$ in Definition \ref{defn:admissible_for_N_large}(\ref{item:admissible_N_large_narrow_transformation}) (when $N \geq 2$) or Definition \ref{defn:admissible_for_N_1}(\ref{item:admissible_N_1_narrow_transformation}) (when $N = 1$), hence $\sigma = \sigma'$ is admissible of type $A_N$ with respect to $\mathcal{D}$.
Similarly, in Case \ref{item:lem_admissible_sequence_extends_wide_N_1_fix}, Case \ref{item:lem_admissible_sequence_extends_wide_N_large_left_branch}, and Case \ref{item:lem_admissible_sequence_extends_wide_N_1} with $s' \in M'$, respectively, the wide transformation $\omega_{n_0}$ satisfies the condition for $\sigma'$ in Definition \ref{defn:admissible_for_N_1}(\ref{item:admissible_N_1_wide_transformation}), Definition \ref{defn:admissible_for_N_large}(\ref{item:admissible_N_large_wide_transformation}), and Definition \ref{defn:admissible_for_N_1}(\ref{item:admissible_N_1_wide_transformation}), respectively.
Hence $\sigma = \sigma'$ is admissible of type $A_N$ with respect to $\mathcal{D}$ in these three cases.

From now, we consider the remaining two cases: Case \ref{item:lem_admissible_sequence_extends_wide_N_1} with $s' \not\in M'$, and Case \ref{item:lem_admissible_sequence_extends_wide_N_large_left_terminal}.
Note that, in Case \ref{item:lem_admissible_sequence_extends_wide_N_large_left_terminal}, the tightness of $\sigma'$ implies that $\lambda(n_0) = 1$ and $\rho' = \mu$, therefore $\sigma$ is the sequence $s'$, $t^{(n_0)}$, $s_1,\dots,s_{\mu}$.
Moreover, in this case the unique element $s'$ of $K^{(n_0)} \cap [y^{(n_0)}]$ does not belong to $\bigcup_{j=0}^{n_0} J^{(j)} = M'$, therefore $t^{(n_0)}$ cannot be adjacent to $[y^{(n_0)}] \cap M'$; hence $t^{(n_0)} \not\in M'$ by the property of $\sigma'$ in Definition \ref{defn:admissible_for_N_large}(\ref{item:admissible_N_large_J_in_M}).
Note also that, in both of the two cases, we have $s' \in J^{(n_0+1)}$ and $\{s'\}$ is an irreducible component of $[y^{(n_0)}]$.

We prove by induction on $0 \leq \nu \leq n_0$ that the sequence $\sigma$ is admissible of type $A_N$ with respect to $\mathcal{D}_{\nu}$ and $s' \in [y^{(\nu+1)}(\mathcal{D}_{\nu})]$, where
\begin{displaymath}
\mathcal{D}_{\nu} = \omega'_{\nu}\omega'_{\nu-1} \cdots \omega'_1\omega'_0 := (\omega_{n_0-\nu})^{-1}(\omega_{n_0-\nu+1})^{-1} \cdots (\omega_{n_0-1})^{-1}(\omega_{n_0})^{-1}
\end{displaymath}
is a semi-standard decomposition of an element with respect to $J^{(n_0+1)}$.
Note that $y^{(j)}(\mathcal{D}_{\nu}) = y^{(n_0-j+1)}$, $J^{(j)}(\mathcal{D}_{\nu}) = J^{(n_0-j+1)}$, $t^{(j)}(\mathcal{D}_{\nu}) = t^{(n_0-j+1)}$ and $K^{(j)}(\mathcal{D}_{\nu}) = K^{(n_0-j+1)}$ for each index $j$.
When $\nu = 0$, this claim follows immediately from the property of $\omega_{n_0}$ specified in Lemma \ref{lem:admissible_sequence_extends}, properties of $\sigma'$ and the definition of $\sigma$.
Suppose that $\nu > 0$.
Note that $s' \in [y^{(\nu)}(\mathcal{D}_{\nu-1})]$ (which is equal to $[y^{(\nu)}(\mathcal{D}_{\nu})] = [y^{(n_0-\nu+1)}]$) by the induction hypothesis.
First, we consider the case that $\omega'_{\nu}$ (or equivalently, $\omega_{n_0-\nu}$) is a wide transformation.
In this case, the possibility of $\omega_{n_0-\nu}$ is as specified in the condition of $\sigma'$ in Definition \ref{defn:admissible_for_N_large}(\ref{item:admissible_N_large_wide_transformation}) (when $N \geq 2$) or Definition \ref{defn:admissible_for_N_1}(\ref{item:admissible_N_1_wide_transformation}) (when $N = 1$), where $h = n_0-\nu$; in particular, we have $K^{(n_0-\nu)} \smallsetminus \{t^{(n_0-\nu)}\} \subseteq M'$, therefore $[y^{(n_0-\nu+1)}] \smallsetminus M' = [y^{(n_0-\nu)}] \smallsetminus M'$.
Hence the element $s'$ of $[y^{(n_0-\nu+1)}] \smallsetminus M'$ belongs to $[y^{(n_0-\nu)}] = [y^{(\nu+1)}(\mathcal{D}_{\nu})]$, and the property of $\omega_{n_0-\nu} = (\omega'_{n_0})^{-1}$ implies that $\sigma$ is admissible of type $A_N$ with respect to $\mathcal{D}_{\nu}$ as well as $\mathcal{D}_{\nu-1}$.
Secondly, we consider the case that $\omega'_{\nu}$ (or equivalently, $\omega_{n_0-\nu}$) is a narrow transformation.
By applying Lemma \ref{lem:admissible_sequence_extends} (for $n = \nu$) to the pair $\mathcal{D}_{\nu}$, $\mathcal{D}_{\nu-1}$ and the sequence $\sigma$, it follows that either $[y^{(\nu+1)}(\mathcal{D}_{\nu})] = [y^{(\nu)}(\mathcal{D}_{\nu})]$, or the support of $\omega'_{\nu}$ is apart from $[y^{(\nu)}(\mathcal{D}_{\nu})] \cap \bigcup_{j=0}^{\nu} J^{(j)}(\mathcal{D}_{\nu})$.
Now in the former case, we have $s' \in [y^{(\nu)}(\mathcal{D}_{\nu})] = [y^{(\nu+1)}(\mathcal{D}_{\nu})]$.
On the other hand, in the latter case, we have $s' \in [y^{(\nu)}(\mathcal{D}_{\nu})] \cap \bigcup_{j=0}^{\nu} J^{(j)}(\mathcal{D}_{\nu})$ since $s' \in [y^{(\nu)}(\mathcal{D}_{\nu})]$ as above and $s' \in J^{(0)}(\mathcal{D}_{\nu}) = J^{(n_0+1)}$ by the choice of $s'$, therefore $s'$ is apart from the support of $\omega'_{\nu}$.
Hence, it follows in any case that $s' \in [y^{(\nu+1)}(\mathcal{D}_{\nu})]$; and the property of $\omega_{n_0-\nu} = (\omega'_{\nu})^{-1}$ specified by the condition of $\sigma'$ in Definition \ref{defn:admissible_for_N_large}(\ref{item:admissible_N_large_narrow_transformation}) (when $N \geq 2$) or Definition \ref{defn:admissible_for_N_1}(\ref{item:admissible_N_1_narrow_transformation}) (when $N = 1$), where $h = n_0-\nu$, implies that $\sigma$ is admissible of type $A_N$ with respect to $\mathcal{D}_{\nu}$ as well as $\mathcal{D}_{\nu-1}$.
Hence the claim of this paragraph follows.

By using the result of the previous paragraph with $\nu = n_0$, the sequence $\sigma$ is admissible of type $A_N$ with respect to $\mathcal{D}_{n_0} = \mathcal{D}^{-1}$, hence with respect to $\mathcal{D}$ as well.
This completes the proof.
\end{proof}

By virtue of the above result, our remaining task is finally to prove Lemma \ref{lem:admissible_sequence_extends} for $n = n_0$ by assuming Lemma \ref{lem:admissible_sequence_extends} for $0 \leq n < n_0$ and Lemma \ref{lem:admissible_sequence_existence_from_lemma} for $0 \leq n < n_0$ (in particular, with no assumptions when $n_0 = 0$).
Put $M' := \{s_1,\dots,s_{\mu}\}$.
In the proof, we may assume without loss of generality that the sequence $s_1,\dots,s_{\mu}$ (denoted here by $\sigma'$) which is admissible with respect to $\mathcal{D}'$ is tight (hence we have $J^{(0)} = M'$ when $n_0 = 0$).
Now by Lemma \ref{lem:proof_special_second_case_transformations_wide}, the claim of Lemma \ref{lem:admissible_sequence_extends} holds for the case that $N = 1$ and $\omega_{n_0}$ is a wide transformation.
From now, we consider the other case that either $N \geq 2$ or $\omega_{n_0}$ is a narrow transformation.
Assume contrary that the claim of Lemma \ref{lem:admissible_sequence_extends} does not hold.
Then, by Lemma \ref{lem:proof_special_second_case_transformations_wide}, Lemma \ref{lem:proof_special_second_case_transformations_narrow} and the properties of the tight sequence $\sigma'$ in Definition \ref{defn:admissible_for_N_large} (when $N \geq 2$) or Definition \ref{defn:admissible_for_N_1} (when $N = 1$), it follows that the possibilities for the $\omega_{n_0}$ is as follows (up to symmetry):
\begin{description}
\item[Case (I):] $\omega_{n_0}$ is a narrow transformation, $K^{(n_0)}$ is of type $A_2$ or type $I_2(m)$ with $m$ odd, and we have $s_{\eta} \in K^{(n_0)} \cap [y^{(n_0)}]$ for some index $1 \leq \eta \leq \mu$; hence $t^{(n_0)} \not\in M'$, $K^{(n_0)} = \{s_{\eta},t^{(n_0)}\}$ and the action of $\omega_{n_0}$ exchanges the two elements of $K^{(n_0)}$.
\item[Case (II):] $N \geq 2$, $\omega_{n_0}$ is a wide transformation, $K^{(n_0)}$ is of type $A_{N+2}$, and $t^{(n_0)}$ is adjacent to $s_{\lambda(n_0)}$ and the unique element $s'$ of $[y^{(n_0)}] \cap K^{(n_0)}$; hence the action of $\omega_{n_0}$ maps the elements $s_{\lambda(n_0)}$, $s_{\lambda(n_0)+1}$, $s_{\lambda(n_0)+2},\cdots,s_{\rho(n_0)}$ and $s'$ to $s'$, $t^{(n_0)}$, $s_{\lambda(n_0)},\dots,s_{\rho(n_0)-2}$ and $s_{\rho(n_0)}$, respectively.
Moreover, $t^{(n_0)} \not\in M'$, and
\begin{description}
\item[Case (II-1):] $s' = s_{j_0}$ for an index $\rho(n_0)+2 \leq j_0 \leq \mu$ with $j_0 \equiv \mu \pmod{2}$;
\item[Case (II-2):] $\lambda(n_0) \geq 3$ and $s' \not\in \{s_{\lambda(n_0)-2},s_{\lambda(n_0)-1},\dots,s_{\mu}\}$;
\item[Case (II-3):] $\lambda(n_0) \geq 3$ and $s' = s_{\lambda(n_0)-2}$.
\end{description} 
\end{description}
In particular, by the tightness of $\sigma'$, the conditions in the above four cases cannot be satisfied when $n_0 = 0$.
Hence the claim holds when $n_0 = 0$.
From now, we suppose that $n_0 > 0$.

For each of the four cases, we determine an element $\overline{s} \in [y^{(n_0)}] \cap M'$ and an element $\overline{t} \in S \smallsetminus [y^{(n_0)}]$ in the following manner: $\overline{s} = s_{\eta}$ and $\overline{t} = t^{(n_0)}$ in Case (I); $\overline{s} = s_{j_0}$ and $\overline{t} = t^{(n_0)}$ in Case (II-1); $\overline{s} = s_{\lambda(n_0)-2}$ and $\overline{t} = s_{\lambda(n_0)-1}$ in Case (II-2); and $\overline{s} = s_{\lambda(n_0)-2}$ and $\overline{t} = t^{(n_0)}$ in Case (II-3).
Note that $\overline{s}$ and $\overline{t}$ are adjacent by the definition.
Since $\sigma'$ is tight, there exists an index $0 \leq h_0 \leq n_0-1$ with $\overline{s} \in J^{(h_0)}$; let $h_0$ be the largest index with this property.
By the definition of $h_0$, $\omega_{h_0}$ is a wide transformation and $J^{(h_0+1)} \neq J^{(h_0)}$.
Let $\overline{r}$ denote the element of $J^{(h_0+1)}$ with $\omega_{h_0} \ast \overline{s} = \overline{r}$.
Then we have $\overline{r} \in [y^{(h_0)}]$ by the property of $\omega_{h_0}$ and the choice of $\overline{s}$.

Let $\overline{\mathcal{D}} := \omega'_{n'-1} \cdots \omega'_1\omega'_0$ denote the simplification of
\begin{displaymath}
(\omega_{n_0-1} \cdots \omega_{h_0+2}\omega_{h_0+1})^{-1} = (\omega_{h_0+1})^{-1}(\omega_{h_0+2})^{-1} \cdots (\omega_{n_0-1})^{-1}
\end{displaymath}
(see Section \ref{sec:finitepart_decomposition_Y} for the terminology), and let $\overline{u}$ be the element of $W$ expressed by the product $\overline{\mathcal{D}}$.
Here we present the following lemma:
\begin{lem}
\label{lem:proof_lem_admissible_sequence_extends_simplification}
In this setting, the support of each transformation in $\overline{\mathcal{D}}$ does not contain $\overline{t}$ and is apart from $\overline{s}$.
\end{lem}
\begin{proof}
We prove by induction on $0 \leq \nu' \leq n'-1$ that the support $K'$ of $\omega'_{\nu'}$ does not contain $\overline{t}$ and is apart from $\overline{s}$.
Let $(\omega_{\nu})^{-1}$ be the term in $(\omega_{h_0+1})^{-1}(\omega_{h_0+2})^{-1} \cdots (\omega_{n_0-1})^{-1}$ corresponding to the term $\omega'_{\nu'}$ in the simplification $\overline{\mathcal{D}}$.
First, by the definition of simplification and the property of narrow transformations specified in Definition \ref{defn:admissible_for_N_large} (when $N \geq 2$) or Definition \ref{defn:admissible_for_N_1} (when $N = 1$), $K'$ is apart from $[y^{(\nu')}(\overline{\mathcal{D}})] \cap M' = [y^{(\nu+1)}] \cap M'$ (see Lemma \ref{lem:another_decomposition_Y_reduce_redundancy} for the equality) if $\omega'_{\nu'}$ (or equivalently, $\omega_{\nu}$) is a narrow transformation.
Now we have $\overline{s} \in [y^{(n_0)}] = [y^{(0)}(\overline{\mathcal{D}})]$ and $\overline{s} \in M'$ by the definition, therefore the induction hypothesis implies that $\overline{s} \in [y^{(\nu')}(\overline{\mathcal{D}})] \cap M'$.
Hence $K'$ is apart from $\overline{s}$ if $\omega'_{\nu'}$ is a narrow transformation.
This also implies that $\overline{t} \not\in K'$ if $\omega'_{\nu'}$ is a narrow transformation, since $\overline{t}$ is adjacent to $\overline{s}$.

From now, we consider the other case that $\omega'_{\nu'}$ (or equivalently, $\omega_{\nu}$) is a wide transformation.
Recall that $\overline{s} \in [y^{(\nu')}(\overline{\mathcal{D}})]$ as mentioned above.
Then, by the property of wide transformation $\omega_{\nu}$ specified in Definition \ref{defn:admissible_for_N_large} (when $N \geq 2$) or Definition \ref{defn:admissible_for_N_1} (when $N = 1$) and the definition of simplification, it follows that $\overline{s} \in J^{(\nu'+1)}(\overline{\mathcal{D}})$ provided $K'$ is not apart from $\overline{s}$.
On the other hand, by the definition of $h_0$, we have $\overline{s} \not\in J^{(j)}$ for any $h_0+1 \leq j \leq n_0$.
This implies that $K'$ should be apart from $\overline{s}$; therefore we have $\overline{t} \not\in K'$, since $\overline{t}$ is adjacent to $\overline{s}$.
Hence the proof of Lemma \ref{lem:proof_lem_admissible_sequence_extends_simplification} is concluded.
\end{proof}
Now, in all the cases except Case (II-2), the following property holds:
\begin{lem}
\label{lem:proof_lem_admissible_sequence_extends_root_other_cases}
In Cases (I), (II-1) and (II-3), there exists a root $\beta \in \Pi^{[y^{(h_0)}]}$ in which the coefficient of $\alpha_{\overline{s}}$ is zero and the coefficient of $\alpha_{\overline{t}} = \alpha_{t^{(n_0)}}$ is non-zero.
\end{lem}
\begin{proof}
First, Lemma \ref{lem:another_decomposition_Y_reduce_redundancy} implies that $\overline{u} \cdot \Pi_{J^{(n_0)}} = \Pi_{J^{(h_0+1)}}$ and $[y'] = [y^{(h_0+1)}]$ where $y' := y^{(n')}(\overline{\mathcal{D}})$.
Put $r' := \overline{u}^{-1} \ast \overline{r} \in J^{(n_0)}$.
Then by Lemma \ref{lem:proof_lem_admissible_sequence_extends_simplification} and Lemma \ref{lem:another_decomposition_Y_shift_x}, we have $\overline{u} \in Y_{z',z}$, where $z$ and $z'$ are elements of $S^{(\Lambda)}$ obtained from $y^{(0)}(\overline{\mathcal{D}}) = y^{(n_0)}$ and $y'$ by replacing the element $\overline{s}$ with $r'$ and $\overline{r}$, respectively.
Now by the property of the wide transformation $\omega_{h_0}$, it follows that $y^{(h_0)}$ is obtained from $y^{(h_0+1)}$ by replacing $\overline{s}$ with $\overline{r}$; hence we have $[z'] = [y^{(h_0)}]$.

We show that there exists a root $\beta' \in \Pi^{[z]}$ in which the coefficient of $\alpha_{\overline{s}}$ is zero and the coefficient of $\alpha_{t^{(n_0)}}$ is non-zero.
In Case (I), $t^{(n_0)}$ is apart from both $[y^{(n_0)}] \smallsetminus \{\overline{s}\}$ and $J^{(n_0)}$, while we have $[z] \subseteq ([y^{(n_0)}] \smallsetminus \{\overline{s}\}) \cup J^{(n_0)}$ by the definition; hence $\beta' := \alpha_{t^{(n_0)}}$ satisfies the required condition.
In Case (II-1), we have $\overline{r} = s_{\lambda(h_0+1)}$ by the property of $\omega_{h_0}$, therefore $r' = s_{\lambda(n_0)}$ by the property of wide transformations in $\overline{\mathcal{D}}$ (see Definition \ref{defn:admissible_for_N_large}(\ref{item:admissible_N_large_wide_transformation})).
Put $\beta' := \alpha_{t^{(n_0)}} + \alpha_{s_{\lambda(n_0)}} + \alpha_{s_{\lambda(n_0)+1}} \in \Pi^{K^{(n_0)},\{r'\}}$ (note that $N \geq 2$ and $K^{(n_0)}$ is of type $A_{N+2}$).
Now $K^{(n_0)}$ is apart from $[y^{(n_0)}] \smallsetminus \{\overline{s}\} = [z] \smallsetminus \{r'\}$, therefore we have $\beta' \in \Pi^{[z]}$ and $\beta'$ satisfies the required condition.
Moreover, in Case (II-3), we have $\overline{r} = s_{\rho(h_0+1)}$ by the property of $\omega_{h_0}$, therefore $r' = s_{\rho(n_0)}$ by the property of wide transformations in $\overline{\mathcal{D}}$ (see Definition \ref{defn:admissible_for_N_large}(\ref{item:admissible_N_large_wide_transformation})).
Now, since $N \geq 2$ and $K^{(n_0)}$ is of type $A_{N+2}$, $t^{(n_0)}$ is not adjacent to $r'$, while $K^{(n_0)}$ is apart from $[y^{(n_0)}] \smallsetminus \{\overline{s}\} = [z] \smallsetminus \{r'\}$.
Hence $\beta' := \alpha_{t^{(n_0)}}$ satisfies the required condition.

By Lemma \ref{lem:proof_lem_admissible_sequence_extends_simplification}, the action of $\overline{u}$ does not change the coefficients of $\alpha_{\overline{s}}$ and $\alpha_{\overline{t}}$.
Hence by the result of the previous paragraph, the root $\beta := \overline{u} \cdot \beta' \in \Pi^{[z']} = \Pi^{[y^{(h_0)}]}$ satisfies the required condition, concluding the proof of Lemma \ref{lem:proof_lem_admissible_sequence_extends_root_other_cases}.
\end{proof}
Since $\overline{t} \not\in J^{(h_0)}$ and $\overline{t}$ is adjacent to $\overline{s}$, the root $\beta \in \Pi^{[y^{(h_0)}]}$ given by Lemma \ref{lem:proof_lem_admissible_sequence_extends_root_other_cases} does not belong to $\Pi_{J^{(h_0)}}$ and is not orthogonal to $\alpha_{\overline{s}}$.
However, since $\overline{s} \in J^{(h_0)}$, this contradicts the fact that $\Pi_{J^{(h_0)}}$ is an irreducible component of $\Pi^{[y^{(h_0)}]}$ (see Definition \ref{defn:admissible_for_N_large}(\ref{item:admissible_N_large_J_irreducible_component}) when $N \geq 2$, or Definition \ref{defn:admissible_for_N_1}(\ref{item:admissible_N_1_J_irreducible_component}) when $N = 1$).
Hence we have derived a contradiction in the three cases in Lemma \ref{lem:proof_lem_admissible_sequence_extends_root_other_cases}.

From now, we consider the remaining case, i.e., Case (II-2).
In this case, the following property holds:
\begin{lem}
\label{lem:proof_lem_admissible_sequence_extends_simplification_2}
In this setting, the support of each transformation in $\overline{\mathcal{D}}$ does not contain $t^{(n_0)}$ and is apart from $s'$.
\end{lem}
\begin{proof}
For each $0 \leq i \leq n_0 - h_0 - 1$, let $\mathcal{D}_i$ denote the semi-standard decomposition of an element defined by
\begin{displaymath}
\mathcal{D}_i = \omega''_i \cdots \omega''_1\omega''_0 := (\omega_{n_0-i})^{-1} \cdots (\omega_{n_0-1})^{-1}(\omega_{n_0})^{-1} \enspace.
\end{displaymath}
For each $0 \leq i \leq n_0 - h_0 - 1$, let $\sigma_i$ denote the sequence $s'$, $t^{(n_0)}$, $s_{\lambda(n_0)}$, $s_{\lambda(n_0)+1},\dots,s_{\overline{\rho}(i)}$, where $\overline{\rho}(i)$ denotes the largest index $s_{\lambda(n_0)} \leq \overline{\rho}(i) \leq \mu$ with $s_{\overline{\rho}(i)} \in \bigcup_{j=0}^{i+1} J^{(j)}(\mathcal{D}_i)$ ($= \bigcup_{j=n_0-i}^{n_0+1} J^{(j)}$).
We prove the following properties by induction on $1 \leq i \leq n_0 - h_0 - 1$: The sequence $\sigma_i$ is admissible with respect to $\mathcal{D}_i$; we have $s' \in [y^{(i+1)}(\mathcal{D}_i)]$; and we have either $[y^{(i+1)}(\mathcal{D}_i)] = [y^{(i)}(\mathcal{D}_i)]$ and $J^{(i+1)}(\mathcal{D}_i) = J^{(i)}(\mathcal{D}_i)$, or the support $K'' = K^{(i)}(\mathcal{D}_i)$ of $\omega''_i$ is apart from $s'$.
Note that, by the properties of $\omega_{n_0}$ and $\sigma'$, we have $s' \in [y^{(n_0)}] = [y^{(1)}(\mathcal{D}_0)]$, and the sequence $\sigma_0$ (which is $s'$, $t^{(n_0)}$, $s_{\lambda(n_0)},\dots,s_{\rho(n_0)}$) is admissible with respect to $\mathcal{D}_0$.

By the induction hypothesis and Lemma \ref{lem:admissible_sequence_extends} for $n = i$ applied to the sequence $\sigma_{i-1}$ and the pair $\mathcal{D}_i$ and $\mathcal{D}_{i-1}$ (note that $i \leq n_0 - h_0 - 1 \leq n_0 - 1$), it follows that the possibilities of $\omega''_i = (\omega_{n_0-i})^{-1}$ are as listed in Lemma \ref{lem:admissible_sequence_extends}.
Now if $\omega''_i$ is a narrow transformation, then as in Case \ref{item:lem_admissible_sequence_extends_narrow} of Lemma \ref{lem:admissible_sequence_extends}, we have either $[y^{(i+1)}(\mathcal{D}_i)] = [y^{(i)}(\mathcal{D}_i)]$, or $K''$ is apart from $s'$ (note that $s' \in [y^{(i)}(\mathcal{D}_i)]$ by the induction hypothesis, while $s' \in J^{(0)}(\mathcal{D}_i) = J^{(n_0+1)}$).
On the other hand, suppose that $\omega''_i = (\omega_{n_0-i})^{-1}$ is a wide transformation.
Then, by the property of $\sigma'$, the support $K''$ of the wide transformation $\omega_{n_0-i}$ is contained in $M'$, therefore $s' \not\in K''$.
This implies that $K''$ is apart from $s'$, since we have $s' \in [y^{(i)}(\mathcal{D}_i)]$ by the induction hypothesis.
Moreover, in any case of $\omega''_i$, we have $s' \in [y^{(i+1)}(\mathcal{D}_i)]$ by the above-mentioned fact $s' \in [y^{(i)}(\mathcal{D}_i)]$ and the above argument.
On the other hand, the sequence $\sigma$ in Lemma \ref{lem:admissible_sequence_existence_from_lemma} corresponding to the current case is equal to $\sigma_i$, therefore $\sigma_i$ is admissible with respect to $\mathcal{D}_i$ by Lemma \ref{lem:admissible_sequence_existence_from_lemma} for $n = i$ (note again that $i \leq n_0 - 1$).
Hence the claim of the previous paragraph holds.

By the above result, the simplification $\overline{D} = \omega'_{n'-1} \cdots \omega'_0$ of $\omega''_{n_0-h_0-1} \cdots \omega''_2\omega''_1$ satisfies the following conditions: For each $0 \leq \nu' \leq n'-1$, we have $s' \in [y^{(\nu')}(\overline{D})]$, and the support of $\omega'_{\nu'}$ is apart from $s'$.
Since $t^{(n_0)}$ is adjacent to $s'$, this implies that the support of each $\omega'_{\nu'}$ does not contain $t^{(n_0)}$.
Hence the proof of Lemma \ref{lem:proof_lem_admissible_sequence_extends_simplification_2} is concluded.
\end{proof}

By Lemma \ref{lem:proof_lem_admissible_sequence_extends_simplification_2}, we have $s' \in [y^{(n')}(\overline{\mathcal{D}})] = [y^{(h_0+1)}]$, therefore the set $K^{(h_0)}$ of type $A_{N+2}$ consisting of $s_{\lambda(n_0)-2}$, $s_{\lambda(n_0)-1},\dots,s_{\rho(n_0)}$ is apart from $s'$.
On the other hand, since $s_{\lambda(n_0)-2} \in [y^{(n_0)}]$, the set $K^{(n_0)}$ of type $A_{N+2}$ is apart from $s'$.
From now, by using these properties, we construct a root $\beta' \in \Pi^{[y^{(n_0)}]} \smallsetminus \Pi_{J^{(n_0)}}$ which is not orthogonal to $\alpha_{s_{\lambda(n_0)+1}} \in \Pi_{J^{(n_0)}}$ (note that $N \geq 2$), in the following five steps.

\textbf{Step 1.}
Note that the set $K^{(n_0)}$ is apart from $[y^{(n_0)}] \smallsetminus K^{(n_0)}$.
Put $z^{(0)} := y^{(n_0)}$.
Then we have $u_1 := w_{z^{(0)}}^{t^{(n_0)}} = s' t^{(n_0)} \in Y_{z^{(1)},z^{(0)}}$, where $z^{(1)} \in S^{(\Lambda)}$ is obtained from $z^{(0)}$ by replacing $s'$ with $t^{(n_0)}$.
Similarly, we have $u_2 := w_{z^{(1)}}^{s_{\lambda(n_0)}} = t^{(n_0)} s_{\lambda(n_0)} \in Y_{z^{(2)},z^{(1)}}$, where $z^{(2)} \in S^{(\Lambda)}$ is obtained from $z^{(1)}$ by replacing $t^{(n_0)}$ with $s_{\lambda(n_0)}$.
Now, since $\beta_0 := \alpha_{s_{\lambda(n_0)}}$ and $\beta'_0 := \alpha_{s_{\lambda(n_0)+1}}$ are non-orthogonal elements of $\Pi_{J^{(n_0)}} \subseteq \Pi^{[z^{(0)}]}$, the roots $\beta_2 := u_2u_1 \cdot \beta_0 = \alpha_{s'}$ and $\beta'_2 := u_2u_1 \cdot \beta'_0 = \alpha_{t^{(n_0)}} + \alpha_{s_{\lambda(n_0)}} + \alpha_{s_{\lambda(n_0)+1}}$ are non-orthogonal elements of $\Pi^{[z^{(2)}]}$.

\textbf{Step 2.}
By the construction, $z^{(2)}$ is obtained from $y^{(n_0)}$ by replacing $s'$ with $s_{\lambda(n_0)}$.
On the other hand, we have $J^{(n_0)} = J^{(h_0+1)}$ and $\overline{u} \ast s_{\lambda(n_0)} = s_{\lambda(n_0)}$ by the property of wide transformations in $\overline{\mathcal{D}}$.
Now by Lemma \ref{lem:another_decomposition_Y_shift_x}, we have $u_3 := \overline{u} \in Y_{z^{(3)},z^{(2)}}$, where $z^{(3)} \in S^{(\Lambda)}$ is obtained from $y^{(n')}(\overline{\mathcal{D}})$ by replacing $s'$ with $s_{\lambda(n_0)}$.
Note that $[z^{(3)}] = ([y^{(n')}(\overline{\mathcal{D}})] \smallsetminus \{s'\}) \cup \{s_{\lambda(n_0)}\} = ([y^{(h_0+1)}] \smallsetminus \{s'\}) \cup \{s_{\lambda(n_0)}\}$.
Put $\beta_3 := u_3 \cdot \beta_2$ and $\beta'_3 := u_3 \cdot \beta'_2$.
Then we have $\beta_3,\beta'_3 \in \Pi^{[z^{(3)}]}$ and $\langle \beta_3,\beta'_3 \rangle \neq 0$.
Moreover, by Lemma \ref{lem:proof_lem_admissible_sequence_extends_simplification} and Lemma \ref{lem:proof_lem_admissible_sequence_extends_simplification_2}, $u_3$ fixes $\alpha_{s'}$, hence $\beta_3 = \alpha_{s'}$; and the action of $u_3$ does not change the coefficients of $\alpha_{s'}$, $\alpha_{t^{(n_0)}}$, $\alpha_{\overline{s}}$ and $\alpha_{\overline{t}}$, hence the coefficients of these four simple roots in $\beta'_3$ are $0$, $1$, $0$ and $0$, respectively.
This also implies that the coefficient of $\alpha_{s_{\lambda(n_0)}}$ in $\beta'_3$ is non-zero, since $t^{(n_0)}$ is adjacent to $s_{\lambda(n_0)} \in [z^{(3)}]$.

\textbf{Step 3.}
Note that the set $K^{(h_0)}$ is apart from $[y^{(h_0+1)}] \smallsetminus K^{(h_0)}$, hence from $[z^{(3)}] \smallsetminus K^{(h_0)}$.
Then we have $u_4 := w_{z^{(3)}}^{\overline{t}} = \overline{t} s_{\lambda(n_0)} \overline{s} \overline{t} \in Y_{z^{(4)},z^{(3)}}$, where $z^{(4)} \in S^{(\Lambda)}$ is obtained from $z^{(3)}$ by exchanging $s_{\lambda(n_0)}$ and $\overline{s}$.
Now we have $\beta_4 := u_4 \cdot \beta_3 = \alpha_{s'} \in \Pi^{[z^{(4)}]}$, $\beta'_4 := u_4 \cdot \beta'_3 \in \Pi^{[z^{(4)}]}$ and $\langle \beta_4,\beta'_4 \rangle \neq 0$.
Moreover, by the property of coefficients in $\beta'_3$ mentioned in Step 2 and the fact that $\overline{t}$ is adjacent to $s_{\lambda(n_0)}$ and $\overline{s}$, it follows that the coefficient of $\alpha_{\overline{s}}$ in $\beta'_4$ is non-zero.

\textbf{Step 4.}
Since $[z^{(4)}] = [z^{(3)}]$, there exists an element $z^{(5)} \in S^{(\Lambda)}$ satisfying that $[z^{(5)}] = [z^{(2)}]$ and $u_5 := \overline{u}{}^{-1} \in Y_{z^{(5)},z^{(4)}}$.
We have $\beta_5 := u_5 \cdot \beta_4 \in \Pi^{[z^{(5)}]}$, $\beta'_5 := u_5 \cdot \beta'_4 \in \Pi^{[z^{(5)}]}$ and $\langle \beta_5,\beta'_5 \rangle \neq 0$.
Now by Lemma \ref{lem:proof_lem_admissible_sequence_extends_simplification} and Lemma \ref{lem:proof_lem_admissible_sequence_extends_simplification_2}, $u_5$ fixes $\alpha_{s'}$, hence $\beta_5 = \alpha_{s'}$; and the action of $u_5$ does not change the coefficient of $\alpha_{\overline{s}}$, hence the coefficient of $\alpha_{\overline{s}}$ in $\beta'_5$ is non-zero.

\textbf{Step 5.}
Put $u_6 := u_2{}^{-1}$ and $u_7 := u_1{}^{-1}$.
Since $[z^{(5)}] = [z^{(2)}]$ as above, there exists an element $z^{(7)} \in S^{(\Lambda)}$ satisfying that $[z^{(7)}] = [z^{(0)}] = [y^{(n_0)}]$ and $u_7u_6 \in Y_{z^{(7)},z^{(5)}}$.
Now we have $\beta_7 := u_7u_6 \cdot \beta_5 = \alpha_{s'}$, since $\beta_5 = \beta_2$.
On the other hand, put $\beta'_7 := u_7u_6 \cdot \beta'_5$.
Then we have $\beta'_7 \in \Pi^{[z^{(7)}]} = \Pi^{[y^{(n_0)}]}$ and $\langle \beta_7,\beta'_7 \rangle \neq 0$.
Moreover, since $u_7u_6 \in W_{S \smallsetminus \{\overline{s}\}}$, the coefficient of $\alpha_{\overline{s}}$ in $\beta'_7$ is the same as the coefficient of $\alpha_{\overline{s}}$ in $\beta'_5$, which is non-zero as mentioned in Step 4.

Hence we have constructed a root $\beta' = \beta'_7$ satisfying the above condition.
However, this contradicts the fact that $\Pi_{J^{(n_0)}}$ is an irreducible component of $\Pi^{[y^{(n_0)}]}$ (see Definition \ref{defn:admissible_for_N_large}(\ref{item:admissible_N_large_J_irreducible_component})).

Summarizing, we have derived a contradiction in any of the four cases, Case (I)--Case (II-3), therefore Lemma \ref{lem:admissible_sequence_extends} for $n = n_0$ holds.
Hence our claim has been proven in the case $\Pi^{J,I \cap J} \subseteq \Phi_{I^{\perp}}$.

This completes the proof of Theorem \ref{thm:YfixesWperpIfin}.

\section{A counterexample for the general case}
\label{sec:counterexample}

In this section, we present an example which shows that our main theorem, Theorem \ref{thm:YfixesWperpIfin}, will not generally hold when the assumption on the $A_{>1}$-freeness of $I \subseteq S$ is removed.

We consider a Coxeter system $(W,S)$ of rank $7$ with Coxeter graph $\Gamma$ in Figure \ref{fig:counterexample}, where the vertex labelled by an integer $i$ corresponds to a generator $s_i \in S$.
Put $I = \{s_4,s_5\}$ which is of type $A_2$ (hence is not $A_{>1}$-free).
\begin{figure}[hbt]
\centering
\begin{picture}(144,80)(0,-80)
\put(8,-8){\circle{16}}\put(8,-12){\hbox to0pt{\hss$1$\hss}}
\put(8,-72){\circle{16}}\put(8,-76){\hbox to0pt{\hss$2$\hss}}
\put(14,-14){\line(1,-1){20}}\put(14,-66){\line(1,1){20}}
\put(40,-40){\circle{16}}\put(40,-44){\hbox to0pt{\hss$3$\hss}}
\multiput(48,-40)(32,0){2}{\line(1,0){16}}
\put(72,-40){\circle{16}}\put(72,-44){\hbox to0pt{\hss$\mathbf{4}$\hss}}
\put(104,-40){\circle{16}}\put(104,-44){\hbox to0pt{\hss$\mathbf{5}$\hss}}
\put(130,-14){\line(-1,-1){20}}\put(130,-66){\line(-1,1){20}}
\put(136,-8){\circle{16}}\put(136,-12){\hbox to0pt{\hss$6$\hss}}
\put(136,-72){\circle{16}}\put(136,-76){\hbox to0pt{\hss$7$\hss}}
\put(128,-8){\line(-3,-1){81}}
\put(80,-20){\hbox to0pt{\hss$\infty$\hss}}
\put(128,-72){\line(-3,1){81}}
\put(80,-64){\hbox to0pt{\hss$\infty$\hss}}
\put(72,-40){\circle{18}}
\put(104,-40){\circle{18}}
\end{picture}
\caption{Coxeter graph $\Gamma$ and subset $I \subseteq S$ for the counterexample; here the two duplicated circles correspond to $I = \{s_4,s_5\}$}
\label{fig:counterexample}
\end{figure}
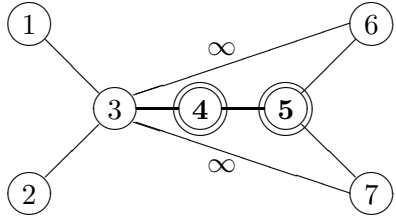

To determine the simple system $\Pi^I$ of $W^{\perp I}$, Proposition \ref{prop:factorization_C}(\ref{item:prop_factorization_C_generator_perp}) implies that each element of $\Pi^I$ is written as $u \cdot \gamma(y,s)$, where $y \in S^{(\Lambda)}$, $u \in Y_{x_I,y}$, $s \in S \smallsetminus [y]$, $[y]_{\sim s}$ is of finite type, $\varphi(y,s) = y$, and $\gamma(y,s)$ is the unique element of $(\Phi_{[y] \cup \{s\}}^{\perp [y]})^+$ as in Proposition \ref{prop:charofBphi}.
In this case, the element $u^{-1} \in Y_{y,x_I}$ admits a decomposition as in Proposition \ref{prop:factorization_C}(\ref{item:prop_factorization_C_generator_Y}).
In particular, such an element $y$ can be obtained from $x_I$ by applying a finite number of operations of the form $z \mapsto \varphi(z,t)$ with an appropriate element $t \in S$.
Table \ref{tab:list_counterexample} gives a list of all the element $y \in S^{(\Lambda)}$ obtained in this way.
In the second and the fourth columns of the table, we abbreviate each $s_i$ ($1 \leq i \leq 7$) to $i$ for simplicity.
This table shows, for each $y$, all the elements $t \in S \smallsetminus [y]$ satisfying that $[y]_{\sim t}$ is of finite type and $\varphi(y,t) \neq y$, as well as the corresponding element $\varphi(y,t) \in S^{(\Lambda)}$ (more precisely, the subset $[\varphi(y,t)]$ of $S$).
Now the list of the $y$ in the table is closed by the operations $y \mapsto \varphi(y,t)$, while it involves the starting point $x_I$ (No.~I in Table \ref{tab:list_counterexample}), therefore the list indeed includes a complete list of the possible $y$.
\begin{table}[hbt]
\centering
\caption{List for the counterexample}
\label{tab:list_counterexample}
\begin{tabular}{|c||c|c||c|c|} \hline
No. & $[y]$ & $\gamma \in \Phi^{\perp [y]}$ & $t$ & $\varphi(y,t)$ \\ \hline
I & $\{4,5\}$ & $[10|0\underline{00}|00]$, $[01|0\underline{00}|00]$ & $3$ & II \\ \cline{4-5}
& & & $6$ & III \\ \cline{4-5}
& & & $7$ & IV \\ \hline
II & $\{3,4\}$ & $[10|\underline{11}1|00]$, $[01|\underline{11}1|00]$ & $1$ & V \\ \cline{4-5}
& & & $2$ & VI \\ \cline{4-5}
& & & $5$ & I \\ \hline
III & $\{5,6\}$ & $[10|00\underline{0}|\underline{0}0]$, $[01|00\underline{0}|\underline{0}0]$ & $4$ & I \\ \cline{4-5}
& & & $7$ & IV \\ \hline
IV & $\{5,7\}$ & $[10|00\underline{0}|0\underline{0}]$, $[01|00\underline{0}|0\underline{0}]$ & $4$ & I \\ \cline{4-5}
& & & $6$ & III \\ \hline
V & $\{1,3\}$ & $[\underline{0}0|\underline{0}01|00]$, $[\underline{1}1|\underline{2}21|00]$ & $2$ & VI \\ \cline{4-5}
& & & $4$ & II \\ \hline
VI & $\{2,3\}$ & $[\underline{0}0|\underline{0}01|00]$, $[\underline{1}1|\underline{2}21|00]$ & $1$ & V  \\ \cline{4-5}
& & & $4$ & II \\ \hline
\end{tabular}
\end{table}

On the other hand, Table \ref{tab:list_counterexample} also includes some elements of $(\Phi^{\perp [y]})^+$ for each possible $y \in S^{(\Lambda)}$.
In the third column of the table, we abbreviate a root $\sum_{i=1}^{7} c_i \alpha_{s_i}$ to $[c_1c_2|c_3c_4c_5|c_6c_7]$.
Moreover, a line is drawn under the coefficient $c_i$ of $\alpha_{s_i}$ if $s_i$ belongs to $[y]$.
Now for each $y$, each root $\gamma \in (\Phi^{\perp [y]})^+$ and each $t$ appearing in the table, the root $w_y^t \cdot \gamma \in (\Phi^{\perp [\varphi(y,t)]})^+$ also appears in the row corresponding to the element $\varphi(y,t) \in S^{(\lambda)}$.
Moreover, for each $y$ in the table, if an element $s \in S \smallsetminus [y]$ satisfies that $[y]_{\sim s}$ is of finite type and $\varphi(y,s) = y$, then the corresponding root $\gamma(y,s)$ always appears in the row corresponding to the $y$.
By these properties, the above-mentioned characterization of the elements of $\Pi^I$ and the decompositions of elements of $Y_{x_I,y}$ given by Proposition \ref{prop:factorization_C}(\ref{item:prop_factorization_C_generator_Y}), it follows that all the elements of $\Pi^I$ indeed appear in the list.
Hence we have $\Pi^I = \{\alpha_{s_1},\alpha_{s_2}\}$ (see the row I in Table \ref{tab:list_counterexample}), therefore both elements of $\Pi^I$ satisfy that the corresponding reflection belongs to $W^{\perp I}{}_{\mathrm{fin}}$.

Moreover, we consider the following sequence of operations:
\begin{displaymath}
\begin{split}
x_I {}:={} & (s_4,s_5) \overset{3}{\to} (s_3,s_4) \overset{1}{\to} (s_1,s_3) \overset{2}{\to} (s_3,s_2) \overset{4}{\to} (s_4,s_3) \\
&\overset{5}{\to} (s_5,s_4) \overset{6}{\to} (s_6,s_5) \overset{7}{\to} (s_5,s_7) \overset{4}{\to} (s_4,s_5) = x_I \enspace,
\end{split}
\end{displaymath}
where we write $z \overset{i}{\to} z'$ to signify the operation $z \mapsto z' = \varphi(z,s_i)$.
Then a direct calculation shows that the element $w$ of $Y_I$ defined by the product of the elements $w_z^t$ corresponding to the above operations satisfies that $w \cdot \alpha_{s_1} = \alpha_{s_2}$.
Hence the conclusion of Theorem \ref{thm:YfixesWperpIfin} does not hold in this case where the assumption on the $A_{>1}$-freeness of $I$ is not satisfied.

\noindent
\textbf{Koji Nuida}\\
Present address: Research Institute for Secure Systems, National Institute of Advanced Industrial Science and Technology (AIST), AIST Tsukuba Central 2, 1-1-1 Umezono, Tsukuba, Ibaraki 305-8568, Japan\\
E-mail: k.nuida[at]aist.go.jp

\end{document}